\documentclass[twoside,11pt]{article}

\usepackage[abbrvbib, preprint]{jmlr2e}
\usepackage{xcolor}

\usepackage{amsmath}
\usepackage{amsfonts}
\usepackage{mathrsfs}	
\usepackage{amssymb}
\usepackage{dsfont}
\usepackage{bbm}
\usepackage{enumerate}
\DeclareMathOperator*{\argmin}{arg\,min}

\newcommand{\1}{\mathds{1}}

\newcommand{\er}{\mathbb R}

\newcommand{\Cc}{\mathcal C}

\newcommand{\Ll}{\mathcal L}
\newcommand{\Mm}{{\mathcal M}}

\newcommand{\Pp}{\mathcal P}

\newcommand{\Vv}{\mathcal V}
\newcommand{\Ww}{\mathcal W}

\newcommand{\EE}{\mathbb E}

\newcommand{\bG}{\mathbf{G}}
\newcommand{\bL}{\mathbf{f}}

\newcommand{\bH}{\mathbf{h}}
\newcommand{\bbH}{\mathit{h}}

\newcommand{\DataM}{\mathbf{\mathcal{M}}}
\newcommand{\ControlS}{\mathit{p}}
\newcommand{\ControlBrownian}{\mathit{B}}
\newcommand{\DataS}{\mathbb{R}^d\times\mathcal{S}}
\newcommand{\Lagrang}{{f}}

\ShortHeadings{Mean-Field Neural ODEs via Relaxed Optimal Control}{Jean-Fran\c{c}ois Jabir, David \v{S}i\v{s}ka and {\L}ukasz Szpruch}
\firstpageno{1}

\begin{document}

\title{Mean-Field Neural ODEs via Relaxed Optimal Control}

\author{\name Jean-Fran\c{c}ois Jabir$^{1}$ \email jjabir@hse.ru \\
       \AND
       \name David \v{S}i\v{s}ka$^{2,3}$ \email D.Siska@ed.ac.uk \\
       \AND
       \name {\L}ukasz Szpruch$^{2,4}$ \email L.Szpruch@ed.ac.uk \\~\\
       $^1$\addr Higher School of Economics, National Research University, Moscow, Russia.\\
       $^2$\addr School of Mathematics, University of Edinburgh, Edinburgh, UK.\\
       $^3$\addr Vega Protocol, Gibraltar, UK.\\
       $^4$\addr The Alan Turing Institute, London, UK.
       \AND
       }

\editor{Kevin Murphy and Bernhard Sch{\"o}lkopf}

\maketitle

\begin{abstract}
We develop a framework for the analysis of Bayesian neural ODE models that are trained with stochastic gradient algorithms.
We do that by identifying the connections between control theory, deep learning and theory of statistical sampling.
We derive Pontryagin's optimality principle and study the corresponding gradient flow in the form of Mean-Field Langevin dynamics (MFLD) for solving relaxed data-driven control problems. Subsequently, we study uniform-in-time propagation of chaos of time-discretised MFLD. We derive explicit convergence rate in terms of the learning rate, the number of particles/model parameters and the number of iterations of the gradient algorithm.
In addition, we study the error arising when using a finite training data set and thus provide quantitative bounds on the generalisation error.
Crucially, the obtained rates are dimension-independent. This is possible by exploiting the regularity of the model with respect to the measure over the parameter space.
\end{abstract}




%

\section{Introduction}\label{sec intro}

There is overwhelming empirical evidence that deep neural networks trained with stochastic gradient descent perform (extremely) well in high dimensional setting \cite{lecun2015deep,silver2016mastering,mallat2016understanding}.
Nonetheless, a complete mathematical theory that would provide theoretical guarantees why and when these methods work so well has been elusive.

In this work, we establish connections between high dimensional data-driven control problems,  deep neural network models and statistical sampling.
We demonstrate how all of them are fundamentally intertwined. Optimal (relaxed) control perspective on deep learning tasks provides new insights, with a solid theoretical foundation. In particular, the powerful idea of relaxed control, that  dates back to the work of L.C. Young on generalised solutions of problems of calculus of variations \cite{young2000lectures}, paves the way for efficient algorithms used in the theory of statistical sampling \cite{majka2018non,durmus2017nonasymptotic,eberle2016reflection,cheng2019quantitative}.
Indeed, in a recent series of works, see~\cite{hu2019mean,mei2018mean,rotskoff2018neural,chizat2018global,sirignano2018mean},
the task of learning the optimal weights in deep neural networks is viewed as a sampling problem.
The picture that emerges is that the aim of the learning algorithm is to find optimal distribution over the parameter space (rather than optimal values of the parameters).
As a consequence, individual values of the parameters are not important in the sense that different sets of weights sampled from the correct (optimal) distribution are equally good.
To learn optimal weights, one needs to find an algorithm that samples from the correct distribution.
It has been shown recently in~\cite{hu2019mean,mei2018mean} that in the case of Bayesian one-hidden layer network the noisy gradient algorithm does precisely that. The key mathematical tools to these results turn out to be the theory of gradient flows and differential calculus on the measure space.

Until recently the depth of neural networks used in practise has been limited. Practitioners reported that as number of layers in traditional DNN increases the training becomes harder. However, the arrival of the so-called residual neural networks (ResNet) in 2016 which outperformed traditional networks across a variety of tasks dramatically changed this situation, \cite{he2016deep}. To extend \cite{hu2019mean,mei2018mean} to multilayer setup we build upon the connection between deep learning and controlled ODEs that has been explored in the pioneering works~\cite{weinan2017proposal,li2017maximum,han2018mean,hu2019meanode}.

\subsection{Overview of the main results} \label{sec: overivew}

The key motivation behind this work is to demonstrate that continuous-time and space analysis of gradient flow on the space of probability measures combined with  probabilistic numerical analysis, that yield  quantitative convergence bounds in terms of the learning rate, the number of iterations of the gradient algorithm  and the size of the training set,  offers a general framework for studying of recurrent neural networks models trained with stochastic gradient algorithm. This perspective has been recently reinforced by \cite{weinan2019machine}.  

This section provides a high level overview of key findings of this work. Precise definitions of appropriate spaces, assumptions and theorems are presented in Section \ref{section main results}. As such the section provides a roadmap for the rest of the paper.

\subsubsection*{Bayesian Neural ODEs via Relaxed Control }

Let $(\xi,\zeta) \in  D$ represent training data (possibly e.g. continuous paths), distributed according to a distribution $\mathcal M \in \mathcal P(D)$ (typically unknown), and let $\varphi:\mathbb R \rightarrow R$ be an activation function. Each layer, indexed by $t$, of the deep network takes input $z\in \mathbb R^d$ and outputs
\[
\frac1n\sum_{i=1}^n \beta_{t,i} \varphi(\alpha_{t,i}\cdot z + \rho_{t,i} \cdot \zeta_t) 
= \int_{\mathbb R^d} \beta \varphi(\alpha \cdot z +\rho \cdot \zeta_t ) \,\nu_t^n(d\beta,d\alpha, d \rho)\,,  \quad 
\nu_t^n =\frac1n \sum \delta_{\{\beta_{t,i},\alpha_{t,i},\rho_{t,i}\}}\,.
\]	
In other words, applying the mean-field scaling at each layer of the neural networks, allows one to shift focus from specific values of the weights to the probability distribution over the weights. This perspective is widely adapted in deep learning community under the banner probabilistic or Bayesian neural networks and lends itself to the uncertainly quantification for deep learning, see e.g. \cite{mackay1995probable,neal2012bayesian,gal2015bayesian}.

Let $a:=(\beta,\alpha,\rho)$ and $\phi(z, a) :=  \beta \, \varphi(\alpha \cdot z +\rho \cdot \zeta ) $. An example of deep neural network architecture studied in this paper is given by
\[
X_{t_{k+1}}^{\nu^n,\xi,\zeta} = X_{t_{k}}^{\nu^n,\xi,\zeta} + \frac{\Delta t}{n} 
\sum_{i=1}^n \phi(X_{t_k}^{\nu,\xi,\zeta}, a^i_{t_k}), \,\,\, \text{where} \,\,\, \Delta t = (t_{k+1} - t_k)\,.
\]
Taking $\Delta t \rightarrow 0$ corresponds to sending the number of layers to infinity \cite{weinan2017proposal,liu2020selection}. Similarly note that as $n \rightarrow \infty$
the empirical distribution over the weights of neural networks  $\nu^n$ is expected to converge to continuous distribution over the weights \cite{hu2019mean}. 
Analysis of a model with infinite number of layers and neurons has the following advantages.
\begin{enumerate}[a)]
\item The continuous limit enjoys number of different numerical approximations.
Each approximation corresponds to a different network architecture.
Different approximations / architectures then lead to different mean-field approximation error, see~\cite{chassagneux2019weak} and different time discretization errors, see~\cite{hairer:norsett:wanner:1993}	.
\item Theoretical guarantees established for the performance of the continuous limit model will translate to networks with an arbitrary number of parameters and layers as long as the architecture after discretization preserves relevant properties of the continuous model. 
This provides a strategy for identifying good neural network architectures.  
\end{enumerate}
Motivated by these observations we consider probabilistic neural ODE model given by
\begin{equation}
\label{eq node intro}
X_t^{\nu,\xi,\zeta} = \xi + \int_0^t \int \phi(X_r^{\nu,\xi,\zeta}, a )\,\nu_r(da)\,dr
\end{equation}
Training of such model, can be recast as optimisation over the space of probability measures 
\begin{equation}\label{eq cost ode}
\min_{\nu} J^{\sigma,\mathcal M}(\nu), \quad \text{with}\quad  J^{\sigma,\mathcal M}(\nu) := \int_{D}\left( g(\xi,X_T^{\nu,\xi,\zeta}) + \mathcal R^{\sigma}(\xi,\zeta,\nu) \right)\mathcal M(d\xi,\zeta),
\end{equation}
where $g$ is an unregularised loss and $\mathcal R^{\sigma}$ is a regulariser. Motivated by the theory of regularised control we take
\[
R^{\sigma}(\xi,\zeta,\nu) :=  \int_{0}^T \left( \int_{\mathbb R^p} f(X_s^{\nu,\xi,\zeta},a)\nu_s(da) + \tfrac{\sigma^2}{2} \text{Ent}(\nu_s) \right)ds,
\]
where $f$ in theory of control is called a running cost function, and $Ent(\mu)$ is relative entropy with respect to prior distribution
$\gamma(x) \approx e^{-U(x)}$ for some potential $U$.  
The entropy term allows one to incorporate prior knowledge in the form on distribution over the weights into the training.
Note that unlike in mean-field modes for one hidden layer studied in \cite{hu2019mean,mei2018mean,rotskoff2018neural,chizat2018global,sirignano2018mean}
 the unregularised loss $J^{0,\mathcal M}$ is not convex even if $g$
 is. 
 Consequently, it is not clear a priori that the entropy term turns the problem strictly convex as it was the case in \cite{hu2019mean,mei2018mean}. 
 It is a common practice to consider the cost function $J^{0,\mathcal M}$, i.e. $\sigma=0$ corresponding to no entropic regularisation, but to train the network with stochastic gradient algorithm, including randomly sampling a mini batch at each step of training or training a random subset of the network's weights at each step i.e. dropout.
The randomness introduced by the stochastic gradient algorithm leads to so called implicit regularisation, see~\cite{neyshabur2017geometry,neyshabur2018towards,heiss2019implicit} and is modelled here by taking $\sigma>0$. 
To put it differently, the noise introduced during training means that some bias is introduced and one should not expect that $J^{0,\mathcal M}$ will decrease along the gradient flow. 
The variational perspective, that we take in this paper, makes the connection between randomness introduced during the training and the exact form of the (implicit) regularisation at the level of cost function $J^{\sigma,\mathcal M}$.

We tackle the optimisation problem \eqref{eq cost ode} using tools from relaxed control theory. One of the contributions of this work is a derivation of Pontryagin maximum principle for measure valued controls using variational calculus. Indeed we show (see Theorem~\ref{thm necessary cond linear}) that if $\nu$ is (locally) optimal then it must solve the  forward-backward system
given by
\begin{equation}
\label{pontryagin system}
\left\{
\begin{split}
 \nu_t^{\star,\mathcal M} & = \argmin_{m \in \mathcal P_2(\mathbb R^p)}  \int_D H^{\sigma}_t(X^{\xi,\zeta}_t,P^{\xi,\zeta}_t, m, \xi)\,\mathcal M(d\xi,d\zeta), \\
dX^{\xi,\zeta}_t & =  \Phi_t(X^{\xi,\zeta}_t,\nu_t,\zeta)\,dt \,,\,\,\,t\in [0,T]\,,X^{\xi,\zeta}_0 = \xi\in \mathbb R^d\,,\zeta \in \mathcal S\,,\\
dP^{\xi,\zeta}_t & = -(\nabla_x H^{0}_t)(X^{\xi,\zeta}_t, P^{\xi,\zeta}_t, \nu_t,\zeta)\,dt\,,\,\,\,t\in [0,T]\,,\,\,\,P^{\xi,\zeta}_T = (\nabla_x g)(X^{\xi,\zeta}_T,\zeta)\,, 
\end{split}	
\right.
\end{equation}
where the regularised Hamiltonian is given by
\begin{equation*}
\begin{split}
	H_t^\sigma(x,p,m,\zeta) & := \int h_t(x,p,a,\zeta)\,m(da)+ \tfrac{\sigma^2}{2}\text{Ent}(m)\,, \\
	h_t(x,p,a,\zeta) &:= \phi_t(x,a,\zeta)p + f_t(x,a,\zeta)\,.
	\end{split}
\end{equation*}
Note that the necessary condition described in \eqref{pontryagin system} involves ``layer-by-layer'' optimisation of the Hamiltonian with $(X^{\xi,\zeta}_t, P^{\xi,\zeta}_t)$ fixed. 
It is instructive to note the differences and similarities between ``training'' using~\eqref{pontryagin system} and usual algorithm for training neural network by gradient descent. 
We see that in~\eqref{pontryagin system}, just as is common, we ``run the network'' for each data input $(\xi,\zeta)$ (the equation for $X^{\xi,\zeta}$). 
We do the ``back-propagation'' again for each data input (the equation for $P^{\xi, \zeta}$). 
But thanks to Pontryagin optimality criteria (see Theorem~\ref{thm necessary cond linear} later) we can ``update the weights'' by minimizing the Hamiltonian ``layer-by-layer'' instead of the cost~\eqref{eq cost ode}.
We show (see Theorem~\ref{thm conv to inv meas rate}) that the optimal distribution over the space of parameters of the neural network, at least for sufficiently large $\sigma$ exists and is unique.
Moreover $\nu^{\star,\sigma} = \argmin_{\nu}J^{\sigma,\mathcal M}(\nu) $ is for each $t$ is given by coupling the forward-backward system for $X^{\xi,\zeta}$ and $P^{\xi,\zeta}$ in~\eqref{pontryagin system} with the functional equation
\begin{equation} \label{invariant measure} 
 \nu_t^{\star,\mathcal M}(a)= \tfrac{1}{Z_t(\nu^{\star})} e^{\int_D h_t(X_t^{\nu^{\star,\mathcal M},\xi,\zeta},P_t^{\nu^{\star,\mathcal M},\xi,\zeta},a,\zeta)\mathcal M(d\xi,d\zeta)}\gamma(a), 
 \end{equation} 
where $Z_t(\nu^{\star})$ is the normalising factor constant in $a$. 
We observe that $\nu_t^{\star,\mathcal M}(a)$, the solution to~\eqref{invariant measure}, enjoys Bayesian interpretation with $\gamma(a)$ being the prior and $\nu_t^{\star,\mathcal M}$ posterior distributions of weights of the neural ODE.

\subsubsection*{Mean-Field Langevin sampling  }

Thus, we showed that the training of neural ODEs in mean-field regime can be recast as a relaxed control problem which we solve using Pontryagin principle on the space of probability measures. 
The next step is to demonstrate that mean-field Langevin sampling algorithm, also known in literature as Wasserstein gradient flow, studied previously in literature in the context of training one hidden layer neural network \cite{mei2018mean,hu2019mean}, corresponds to a noisy stochastic gradient algorithm used for training deep neural networks. 
Let  $\nabla \cdot$ denote the divergence operator and let $\tfrac{\delta H_{t}^{\sigma }}{\delta m}$ be the linear functional derivative (first variation), see Appendix \ref{sec measure derivatives} for the definition.
From \eqref{pontryagin system} and the theory of gradient flows, cf. \cite{villani2008optimal}, one would hope that the gradient flow equation, for each ``layer'' $t\in[0,T]$, is given by 
\begin{equation}\label{eq intro mfsgd 0}
	\tfrac{d}{ds} \nu_{s,t}   =  \nabla \cdot \left( \left(\nabla_a\tfrac{\delta H_{t}^{\sigma }}{\delta m}\right)(X_{s,t}, P_{s,t}, \nu_{s,t}, a, \mathcal M)  \nu_{s,t}  \right), \quad \nu_{0,t}=\mu\,,  \quad s \geq 0\,,
	\end{equation}
where $X_{s,t} = (X_{s,t}^{\xi,\zeta})_{\xi,\zeta}$ and $P_{s,t} = (P_{s,t}^{\xi,\zeta})_{\xi,\zeta}$.
The gradient flow is coupled with, for each ``gradient flow time'' $s\geq 0$, with the forward backward system
\begin{equation}
\label{eq intro mfsgd 2}	
\left\{
\begin{aligned}
	X^{\xi,\zeta}_{s,t} & = \xi + \int_0^t \Phi_r(X^{\xi,\zeta}_{s,r},\nu_{s,r},\zeta)\,dr \,,\,\,\, t\in [0,T]\,, \\
P^{\xi,\zeta}_{s,t} & = (\nabla_x g)(X^{\xi,\zeta}_T,\zeta) + \int_t^T (\nabla_x H^{0}_r)(X^{\xi,\zeta}_{s,r},P^{\xi,\zeta}_{s,r},\nu_{s,r},\zeta)\,dr \,,\,\,\, t\in [0,T] \,.\\
\end{aligned}
\right.
\end{equation}
We stress again that in this setting the forward process $(X^{\xi,\zeta}_{s,t})_t$ 
plays the role of running the neural network with input $(\xi, \zeta)$ while the backward / adjoint process $(P^{\xi,\zeta}_{s,t})_t$ plays the role of back propagation. 
One of the main results of this work is to show that \eqref{eq intro mfsgd 0} is indeed a ``correct'' gradient flow in a sense that $\tfrac{d}{ds} J^{\sigma,\mathcal M}(\nu_{s,t}) \leq 0$ i.e the loss function is a decreasing function along the gradient flow \eqref{eq intro mfsgd 0}. Furthermore, for sufficeintly large $\sigma$ we show in Theorem \ref{thm conv to inv meas rate} that 
\begin{equation*}
	\mathcal W^T_2(\mathcal L( \theta_{s,\cdot}),\nu^{\star,\mathcal M})^2 \leq e^{-\lambda s}   {\cal W}^T_2(\mathcal L( \theta_{0,\cdot}),\nu^{\star,\mathcal M})^2\,, 
\end{equation*}
where $\mathcal W^T_2(\mu,\nu):=\left( \int_0^T \mathcal W_2(\mu_t,\nu_t)^2\,dt \right)^{1/2}$ with $\mathcal W_q$ being the usual Wasserstein distance. We remark that the rate of convergence does not dependent on the dimension of the state and parameter spaces and the results hold for either $\mathcal M$ or its empirical approximation. It is useful to contrast this result with recent work on the linearizations of neural networks around initialisation known as neural tangent kernel, see \cite{jacot2018neural,arora2019fine}, or lazy training regime \cite{chizat2018lazy}. For linearised model the loss function can be shown to go to zero exponentially fast. Furthermore, \cite{chizat2018lazy,mei2019mean}     
showed that by appropriately rescaling one-hidden layer neural network model one can show that, at least asymptotically, the distributions over the weights does not change from its initialisation. Subsequent works \cite{ghorbani2019linearized,ghorbani2020neural} studied the limitation of lazy regime. Here, we prove exponential convergence in mean-field regime and go beyond one hidden-layer model, but require strong regularisation. In fact, one can see that by taking large $\sigma$ we are imposing ``strong'' prior on the posterior distribution over weights  $\nu_t^{\star,\mathcal M}(a)$. 
That way, our work provides an alternative link between mean-field and lazy regimes. Similar observation for one hidden layer has recently been made in \cite{tzen2020mean}. 

Let 
\begin{equation*}
\mathbf h_t(a, \mu, \mathcal M) :=  \frac{\delta H^0}{\delta m}(X_t(\mu),P_t(\mu),a,\mathcal M) =   \int_{D} h_t(X^{\xi,\zeta}_t(\mu),P^{\xi,\zeta}_t(\mu),a,\zeta)\mathcal M(d\xi, d\zeta)\,.
\end{equation*}
To simulate $(\nu_{s,t})_{s>0}$ we rely on the probabilistic representation of~\eqref{eq intro mfsgd 0} given by 
 the mean-field equation 
\begin{equation}\label{eq into mfsgd prob}
	d \theta_{s,t}  =  -\left( ( \nabla_a \mathbf h_t ) (\theta_{s,t},\nu_{s,t},\mathcal M )\, + \frac{\sigma^2}{2}(\nabla_a U)(\theta_{s,t}) \right)ds + \sigma dB_s\, \quad s \geq 0\,,
	\end{equation}
where $(\theta^0_{t})_{t\in[0,T]}$ is a given initial condition and where
$\nu_{s,t}$ is the law of $\theta_{s,t}$. 
We must again couple~\eqref{eq into mfsgd prob} with~\eqref{eq intro mfsgd 2}.	
Mean-field Langevin dynamics~\eqref{eq into mfsgd prob} can be viewed as continuous time noisy gradient descent. 
We remark that perturbing weights during training with Gaussian noise is  common practice when designing differentially private models, see~\cite{dziugaite2018data}. 

Particle approximation to \eqref{eq into mfsgd prob} leads to familiar (noisy) stochastic gradient algorithms used to train neural networks. Indeed, consider  a sequence $(\xi^i,\zeta^i)_{i=1}^{N_1}$ of i.i.d copies of $(\xi,\zeta)$ and let $\mathcal M^{N_1} := \frac{1}{N_1}\sum_{j=1}^{N_1}\delta_{\{\xi^j, \zeta^j\}}$ be the empirical measure representing the available training sample.
Further, we fix an increasing sequence $0=s_0 < s_1 < s_2 \cdots$ and define the 
 $(\widetilde{\theta}^{i}_{l,t})_{l \in \mathbb N, 0\leq t\leq T}$ satisfying, for  $i=1,\ldots,N_2$ and $l \in \mathbb N$,
\begin{equation}\label{eq:EulerScheme1 intro}
\widetilde{\theta}^{i}_{s_{l+1},t}=\widetilde \theta^{i}_{s_{l},t}- \left(( \nabla_a \mathbf h_t ) (\widetilde \theta_{s_{l},t}, \widetilde{{\nu}}^{N_2}_{s_{l},t}, \mathcal M^{N_1} )
+\frac{\sigma^2}{2}(\nabla_aU)(\widetilde \theta^i_{s_{l},t})\right)\,(s_{l+1}-s_l)
+\sigma (B^i_{s_{l+1}} - B^i_{s_l})\,,
\end{equation}
where $\widetilde{{\nu}}^{N_2}_{s_{l},t}=\frac{1}{N_2}\sum_{j_2=1}^{N_2}\delta_{\widetilde{\theta}^{j_2}_{s_{l},t}}$. 
In this work we provide dimension independent and uniform-in-time bounds for convergence of \eqref{eq:EulerScheme1 intro} to \eqref{eq into mfsgd prob}, see Theorems    \ref{thm:WellPosed&PropagationChaos} and \ref{thm:EulerRate1}, 
and use them to study generalisation error. In practice one also considers a numerical approximation of the process $(X_t)_{t\in[0,T]}$ on a finite time partition $0=t_0\leq\cdots\leq t_n=T$ of the interval $[0,T]$.
Different numerical approximations of process $(X_t)_{t\in[0,T]}$ can be interpreted as different neural network architectures, see again~\cite{chen2018neural}.
We omit numerical error in the introduction and remark that numerical analysis of ODEs is a mature field of study, see~\cite{hairer:norsett:wanner:1993}.

\subsubsection*{Generalisation Error}

Let $\nu^{\star,\sigma,N_1} := \argmin_{\nu }J^{\sigma,\mathcal M^{N_1}}(\nu) $ be the optimal distribution over the parameter space when minimising empirical loss with $N_1$ data points, with noisy gradient descent \eqref{eq into mfsgd prob}. 
Let $\nu_{S,\cdot}^{\sigma,N_1,N_2,\Delta s}$ denote the distribution over the parameter space induced by the gradient algorithm when training with $N_1$ data samples, finite number of model parameters (the number of which is $\approx N_2 \times p \times n$, where $n$ is the number of grid points of $[0,T]$), the learning rate $\Delta s = \max_{0 < s_l < S} (s_l - s_{l-1})$, and training time $S$. The generalisation error  $J^{0,\mathcal M}(\nu_{S,\cdot}^{\sigma,N_1,N_2,\Delta s})$ is the the value of the loss function under population measure $\mathcal M$ evaluated at $\nu_{S,\cdot}^{\sigma,N_1,N_2,\Delta s}$. Note that we can write
\[
J^{0,\mathcal M}(\nu_{S,\cdot}^{\sigma,N_1,N_2,\Delta s}) = J^{0,\mathcal M}(\nu_{S,\cdot}^{\sigma,N_1,N_2,\Delta s}) - J^{0,\mathcal M}(\nu^{\star,\sigma}) - \frac{\sigma^2}{2} \int_0^T\text{Ent}(\nu^{\star,\sigma}_t)\,dt +  \min_{\mu\in \mathcal V_2}J^{\sigma,\mathcal M}(\mu)\,,
\]
since $\min_{\mu\in \mathcal V_2}J^{\sigma,\mathcal M}(\mu) = J^{0,\mathcal M}(\nu^{\ast,\sigma}) + \tfrac{\sigma^2}{2} \int_0^T \text{Ent}(\nu^{\ast,\sigma}_t)\,dt $.

Thus we see that the generalisation error consist of three errors: a) the numerical error of approximating an invariant measure with discrete time particle system, b) the relative entropy between the Gibbs measure $\gamma$ (a prior) and the $\nu^{\star,\sigma}$,  c) the minimum value of the cost function under population measure $J^{\sigma,\mathcal M}$. Under appropriate smoothness of the loss function with respect to measure over the parameter space obtain (in Theorem \ref{th generalisation}) 
\[
\mathbb E\left[\Big|J^{0,\mathcal M}(\nu^{\star,\sigma})- J^{0,\mathcal M}(\nu_{S,\cdot}^{\sigma,N_1,N_2,\Delta s})\Big|^2\right] \leq c\left(e^{-\lambda S} + \frac1{N_1} + \frac1{N_2} + h \right)\,,
\]
where $h := \max_{0 < s_l < S} (s_l - s_{l-1})$. We stress that the rates that we obtained are dimension independent, which is not common in the literature.
Indeed, such bound wouldn't hold in a dimension-free way for functions that are only Lipschitz continuous w.r.t. the Wasserstein distance, see~\cite{fournier2015rate} or \cite{dereich2013constructive}. 
For more regular functions it is possible to obtain such dimension-free estimates~\cite[Lem. 5.10]{delarue2019master}, \cite[Lem.2.2]{szpruch2019antithetic} and~\cite{jabir2019rate}.
In this paper we also exploit this additional regularity to obtain such estimates.

It is a common practice to terminate the training when $J^{0,\mathcal M^{N_1}}$ is negligible
\cite{zhang2016understanding,belkin2019reconciling,montanari2019generalization,hastie2019surprises,mei2019generalization} 
and such models have been observed to generalise well. 
In a situation when $J^{0,\mathcal M^{N_1}}$ is not negligible after training for sufficiently long time $S$, then the model is considered untrained and one would not expect it to generalise well.
If one assumes that for fixed $\varepsilon>0$ and $N_1>0$, $J^{0,M_1}(\nu^{\star,\sigma,N_1}) \leq \varepsilon$ then we show (see Theorem \ref{th conditional generalisation}) that
\[
\mathbb E\left[\Big|J^{0,\mathcal M}(\nu_{S,\cdot}^{\sigma,N_1,N_2,\Delta s})\Big|^2\right] \leq  \varepsilon^2 + c\left(e^{-\lambda S} + \frac1{N_1} + \frac1{N_2} + h \right)\,.
\]
 
 The assumption that $J^{0,M_1}(\nu^{\star,\sigma,N_1}) \leq \varepsilon$ could be verified by establishing universal approximation theorem for the neural ODEs in a spirit of \cite{sontag1997complete,cuchiero2019deep} and combining it with the analysis presented in this work. We postpone this direction of research to future work.

While this work is motivated by the desire to put deep learning on a solid mathematical foundation, as a byproduct, ideas emerging from machine learning provide new perspective on classical dynamic optimal control problems. Indeed, high dimensional control problems are ubiquitous in technology and science \cite{bertsekas1995dynamic,bensoussan2004stochastic,bensoussan2011applications,fleming2006controlled,carmona2018probabilistic}. There are many computational methods designed to find, or approximate, the optimal control functional, see e.g.~\cite{kushner2001numerical} or~\cite{gyongy2009normalized} and references therein.
These, typically, rely on dynamic programming, discrete space-time Markov chains, finite-difference methods or Pontryagin's maximum principle and, in general, do not scale well with the dimensions.
Indeed the term ``curse of dimensionality'' (computational effort grows exponentially with the dimension) has been coined by R. E. Bellman when considering problems in dynamic optimisation \cite{bellman1966dynamic}.
This work advances the study of a new class of algorithm for control problems that is particularly well adapted to high dimensional setting.

\subsection{Related work}

In \cite{weinan2017proposal,li2017maximum,han2018mean}, Pontryagin's optimality principle is leveraged
 and the convergence of the method of successive approximations for training the neural ODEs is studied.
The authors suggested the possibility of combining their algorithm with gradient descent but have not studied this connection in full detail.
The term Neural ODEs has been coined in~\cite{chen2018neural} where the authors exploited the computational advantage of Pontryagin's principle approach (with its connections to automatic differentiation~\cite{baydin2018automatic}).
Finally, \cite{hu2019meanode} (see also \cite{bo2019relaxed}) formulated the relaxed control problem and showed that the recently methodology developed in \cite{hu2019mean} can be successfully applied in this setup, to prove convergence of flows of measures induced by the mean-field Langevin dynamics to invariant measure that minimised relaxed control problem.
The starting point of~\cite{hu2019meanode} is a relaxed Pontryagin's representation for the control problem and the authors prove convergence of the continuous time dynamics in a fixed data regime and, in the case of the neural ODE application, infinite number of parameters.
Our work complements~\cite{hu2019meanode}: first we derive the Pontryagin's optimality principle for the data-driven relaxed control problem. From there we identify the gradient flow on the space of probability measures along which corresponding energy function is decreasing. Next, we study the complete algorithm and provide quantitative convergence bounds in terms of the learning rate, the number of iterations of the gradient algorithm  and the size of the training set. Finally we derive quantitative bounds on the generalisation error by exploiting smoothness of the energy function. 

 We remark that control problem perspective is fruitful when studying universal approximation results for (neural) ODEs, \cite{sontag1997complete,cuchiero2019deep,ma2019barron}.

\subsection{Examples of application to machine learning problems}

The setting in this paper is relevant to many types of machine learning problems. Below we give two examples of how the results of this paper can be applied to gain insight into machine learning tasks.

\begin{example}[Nonlinear regression and function approximation]
Consider a function $f:\mathbb R^{\bar d} \to \mathbb R^d$ one wishes to approximate.
This is to be done by sampling $(\xi,\zeta)$ from $\mathcal M$ where $\xi = f(\zeta)$.
We see that here $\mathcal S = \mathbb R^{\bar d}$. 	

The objective can be taken as
\[
J^\sigma(\nu) = \int_{\mathbb R^d \times \mathbb R^{\bar d}} |X_T^{\xi,\zeta} - f(\zeta)|^2 \, \mathcal M(d\xi,d\zeta) + \frac{\sigma^2}{2}\int_0^T \text{Ent}(\nu_t)\,dt\,.
\]
We now fix a
nonlinear activation function $\varphi:\mathbb R^d \to \mathbb R^d$ and take $\phi$ to be
\begin{equation*}
\phi(x,a) := a_1 \varphi(a_2 x)\,, \,\,\,a=(a_1,a_2)\in \mathbb R^p\,,	
\end{equation*}
with $a_1 \in \mathbb R^{d\times d}$, $a_2 \in \mathbb R^{d\times d}$ so $p := 2 d^2$.
The neural network will then be given by discretizing~\eqref{eq node intro}.
\end{example}

In the following example we consider a time-series application.
In this setting the input will be one or more ``time-series'' each a continuous a path on $[0,T]$.
The reason for considering paths on $[0,T]$ is that this covers the case of unevenly spaced observations.

\begin{example}[Missing data interpolation]
The learning data set consists of the true path $\zeta^{(2)} \in C([0,T];\mathbb R^d)$ and a set of observations $(\zeta_{t_i}^{(1)})_{i=1}^{N_{\text{obs}}}$ on $0\leq t_1 \leq \cdots \leq t_{N_{\text{obs}}} \leq T$.
We will extend this to entire $[0,T]$ by piecewise linear interpolation denoted $\zeta^1_t$.
The learning data is then $(\zeta^{(1)},\zeta^{(2)})=:\zeta \in \big(C([0,T];\mathbb R^{\bar d})\big)^2 =: \mathcal S$
distributed according to the data measure $\mathcal M \in \mathcal P(\mathcal S)$.
As before we fix a nonlinear activation function $\varphi:\mathbb R^d \to \mathbb R^d$.
Let us take $\phi$ to be
\begin{equation*}
\phi_t(x,a,\zeta) := a_1 \varphi(a_2 x + a_3 \zeta^{(1)}_t)\,, \,\,\,a=(a_1,a_2,a_3)\in \mathbb R^p\,,	
\end{equation*}
with $a_1 \in \mathbb R^{d\times d}$, $a_2 \in \mathbb R^{d\times d}$, $a_3 \in \mathbb R^{d\times \bar d}$, so that $d^2 + d^2 + d \bar d = p$.
Let $L \in \mathbb R^{\bar d\times d}$ be a matrix.
The learning task is then to minimize
\[
J^\sigma(\nu) = \int_{\mathcal S} \int_0^T | L X_t^{\nu,\zeta^{(1)}} - \zeta^{(2)}_t|^2 \,dt\, \mathcal M(d\zeta^1,d\zeta^2) + \frac{\sigma^2}{2}\int_0^T \text{Ent}(\nu_t)\,dt
\]
over all measures $\nu \in \mathcal V_2$
subject to the dynamics of the forward ODE given by~\eqref{eq node intro}.
\end{example}

\section{Main results}
\label{section main results}

In this section we state precisely the assumptions and the results discussed in Section \ref{sec intro}. 
Some notions from Section~\ref{sec intro} are repeated here so that this section can be read on its own.

\subsection{Statement of problem}

Given some metric space $E$ and $0< q <\infty$, let $\Pp_q(E)$ denote the set of probability measures defined on $E$ with finite $q$-th moment.
Let $\Pp_0(E)=\Pp(E)$ be the set of probability measures on $E$.
Let $\mathcal V_2$ denote the set of positive Borel measures on $[0, T] \times \mathbb R^p$ with the first marginal equal to the Lebesgue measure, the second marginal being a probability measure with finite second moment.
That is
\begin{equation}
\label{eq def Vq}	
\mathcal V_2 :=\Big\{ \nu \in \mathscr M([0,T]\times \mathbb R^p): \nu(dt,da)=\nu_t(da)dt, \, \nu_t \in \Pp_2(\er^p),\,  \int_{0}^{T}\!\!\!\int |a|^2\nu_t(da)dt < \infty \Big\}\,,
\end{equation}
where here and elsewhere any integral without an explicitly stated domain of integration is over $\mathbb R^p$.
We consider the following controlled ordinary differential equation (ODE):
\begin{equation}
\label{eq process}
X_t^{\xi,\zeta}(\nu) = \xi + \int_0^t \int \phi_r(X_r^{\xi,\zeta}(\nu), a, \zeta)\,\nu_r(da)\,dr
 \,,\,\,\, t\in [0,T]\,,
\end{equation}
where $\nu = (\nu_t)_{t\in [0,T]} \in \mathcal V_2$
is the control,
where $(\xi,\zeta) \in \mathbb R^d \times \mathcal S$ denotes 
some external data, distributed according to $\mathcal M \in \mathcal P_2(\mathbb R^d \times \mathcal S)$ and where $(\mathcal S, \|\cdot\|_{\mathcal S})$ is a normed space.

For $\nu \in \mathcal P(\mathbb R^p)$ let $\text{Ent}(\nu) := \infty$ if $\nu$ is not absolutely continuous with
respect to Lebesgue measure. 
Otherwise 
let
\begin{equation}
\label{eq Ent}
\text{Ent}(\nu) := \int \left[ \log \nu(a) - \log \gamma(a)\right]\nu(a)\,da\,,\,\,\,\text{where}\,\,\,\gamma(a) = e^{-U(a)}\,\,\,\text{with $U$ s.t.}\,\,\,  \int e^{-U(a)}\,da = 1\,.	
\end{equation}
Given $f$ and $g$ we define objective functionals as
\begin{equation}
\label{eq objective bar J}
\begin{split}
\bar J^{\sigma}(\nu, \xi, \zeta) & := \int_0^T \int f_t(X_t^{\xi,\zeta}(\nu), a, \zeta)\,\nu_t(da) \,dt + g(X^{\xi,\zeta}_T(\nu),\zeta)\,	+ \frac{\sigma^2}{2} \int_0^T \text{Ent}(\nu_t)\,dt \,\\
\end{split}
\end{equation}
as well as
\begin{equation}
\label{eq objective J}
J^{\sigma, \mathcal M}(\nu) :=  \int_{\mathbb R^d \times \mathcal S} \bar J^\sigma(\nu, \xi, \zeta) \, \mathcal M(d\xi,d\zeta) \,.
\end{equation}
Note that when $\mathcal M$ is fixed in~\eqref{eq objective J} then we don't emphasise the dependence of $J^{\sigma,\mathcal M}$ on $\mathcal M$ in our notation and write only $J^\sigma$ and we write $J^{0}=J$ and $J^{0,\mathcal M}=J^{\mathcal M}$. Once $\mathcal M$ is fixed, the aim is to minimize $J^\sigma$ over all controls $\nu \in \mathcal V_2$, subject to the controlled process $X^{\xi,\zeta}(\nu)$ satisfying~\eqref{eq process}.
In case $\sigma \neq 0$ we additionally require that $\nu_t$ is absolutely continuous with respect to the Lebesgue measure for almost all $t\in [0,T]$.
For convenience define
\[
\Phi_t(x,m,\zeta) := \int \phi_t(x, a, \zeta)\,m(da) \text{   and   }
F_t(x,m, \zeta) := \int f_t(x, a, \zeta)\,m(da)\,.
\]



\subsection{Assumptions and theorems}

We start by briefly introducing some terminology  and notation.
We will use $\frac\delta{\delta \nu}$ to denote the flat derivative on $\mathcal V_2$
and $\frac\delta{\delta m}$ to denote the flat derivative on $\mathcal P_2(\mathbb R^p)$, see Section~\ref{sec measure derivatives} for definitions of these objects.
We will say that some function $\psi = \psi_t(x,a,\zeta)$
is Lipschitz continuous in $(x,a)$, uniformly in $(t,\zeta)\in[0,T]\times\mathcal S$ if  we have that
\[
\sup_{t\in[0,T]}\sup_{\zeta\in \mathcal S}\sup_{\substack{x,x'\in\mathbb R^d,\\ a,a'\in \mathbb R^p,\\ \,(x,a)\neq(x',a')}}\frac{|\psi_t(x,a,\zeta)-\psi_t(x',a',\zeta)|}{|(x,a)-(x',a')|}<\infty\,.
\]
For $\psi = \psi(w)$ we will use
$\Vert \psi\Vert_{\infty}:=\sup_{w}|\psi(w)|$
to denote the supremum norm and we will use
$\Vert \psi\Vert_{\text{Lip}}:=\sup_{w\neq w'}\frac{|F(w)-F(w')|}{|w-w'|}$ to denote the Lipschitz norm.
We will use $c$ to denote a generic constant that may change from line to line but must be indepdent of $p$, $d$ and all other parameters that appear explicitly in the same expression.
\begin{assumption} \label{as coefficients}
\begin{enumerate}[i)]
\item $\int_{\mathbb R^d \times \mathcal S}[|\xi|^2 + \|\zeta\|_{\mathcal S}^2] \, \mathcal M(d\xi,d\zeta) <\infty$.
\item $\phi$, $\nabla_a\phi$, $\nabla_x\phi$, $f$, $\nabla_a\Lagrang$, $\nabla_x\Lagrang$ and $\nabla_xg$ are all Lipschitz continuous in $(x,a)$, uniformly in $(t,\zeta)\in[0,T]\times\mathcal S$.
Moreover, $x\mapsto \nabla_x\phi$, $\nabla_x\Lagrang$ and $\nabla_xg$ are all continuously differentiable.
\item

\begin{align*}
&\sup_{t\in[0,T]}\int_{\mathbb R^d\times \mathcal S} \Big[|g(0,\zeta)|^2+|\Lagrang_{t}(0,0,\zeta)|^2 + |\phi_{t}(0,0,\zeta)|^2\\
&\quad\quad+|\nabla_xg(0,\zeta)|^2+|\nabla_x\Lagrang_{t}(0,0,\zeta)|^2+|\nabla_a\Lagrang_{t}(0,0,\zeta)|^2\Big]\,\mathcal M(d\xi,d\zeta) < \infty.
\end{align*}

\end{enumerate}

\end{assumption}
%

\begin{definition}[Permissible flows]
\label{def permissible flow}
We will call $(b_{\cdot,t} \in C^{0,1}([0,\infty)\times \mathbb R^p; \mathbb R^p))_{t\in[0,T]}$ a permissible flow if for all $s\geq0, t\in [0,T]$ the map $a\mapsto b_{s,t}(a)$ has at most linear growth.
\end{definition}

\begin{lemma}
\label{lemma vect field flow}
If $b = b_{s,t}(a)$ is a {\em permissible flow} (c.f. Definition~\ref{def permissible flow}) then, for all $t\in [0,T]$ the equation
\[
\partial_s \nu_{s,t} = \nabla_a \cdot \Big( b_{s,t}\nu_{s,t}+\frac{\sigma^2}{2}\nabla_a \nu_{s,t} \Big)\,, s \in [0,\infty)\,,\nu_{0,t} \in \mathcal P_2(\mathbb R^p),
\]
has a unique solution $\nu_{\cdot,t} \in C^{1,2}((0,\infty)\times \mathbb R^p)$ 
and moreover $\nu_{s,t} > 0$ for all $s>0,t\in[0,T]$.
\end{lemma}
The existence and uniqueness stated in Lemma~\ref{lemma vect field flow} are proved in e.g.~\cite[Chapter IV]{LSU68}
and the fact that the solution is strictly positive can be obtained from its stochastic representation and via a Girsanov transform.

We now introduce the relaxed Hamiltonian:
\begin{equation}
\label{eq hamiltonian}
\begin{split}
H^0_t(x,p,m,\zeta) & := \int h_t(x,p,a, \zeta)\,m(da) \,,\,\,\, \text{where}\,\,\,h_t(x,p,a,\zeta) := \phi_t(x,a,\zeta)p + f_t(x,a,\zeta)\,,\\
H^\sigma_t(x,p,m,\zeta) & := H^0_t(x,p,m,\zeta) + \frac{\sigma^2}{2}\text{Ent}(m)\,.
\end{split}
\end{equation}
We will also use the adjoint process
\begin{equation}
\label{eq adjoint proc}
\begin{split}
dP_t^{\xi,\zeta}(\nu)& =-(\nabla_x H^0_t)(X^{\xi,\zeta}(\nu)_t,P^{\nu,\xi,\zeta}_t(\nu),\nu_t)\,dt \,,\,\,t\in [0,T]\,, \,\, P^{\xi,\zeta}_t(\nu) = (\nabla_x g)(X^{\xi,\zeta}_t(\nu),\zeta),\,
\end{split}
\end{equation}
and note that trivially $\nabla_x H^0 = \nabla_x H^\sigma$.
In Lemma~\ref{lemma odes} we will prove that, under Assumption~\ref{as coefficients}, the system~\eqref{eq process},~\eqref{eq adjoint proc} has a unique solution.
For $\mu \in \mathcal V_2$, $\mathcal M \in \mathcal P_2(\mathbb R^d\times\mathcal S)$ let
\begin{equation}
\label{eq fat h def}
\mathbf h_t(a, \mu, \mathcal M) := \int_{\mathbb R^d \times \mathcal S} h_t(X^{\xi,\zeta}_t(\mu),P^{\xi,\zeta}_t(\mu),a,\zeta)\mathcal M(d\xi, d\zeta)\,,	
\end{equation}
where $(X^{\xi,\zeta}(\mu),P^{\xi,\zeta}(\mu))$ is the unique solution to~\eqref{eq process} and~\eqref{eq adjoint proc}.

We now state a key result on how to choose the gradient flow to solve the control problem.

\begin{theorem}
\label{corollary with measure flow}
Fix $\sigma > 0$, let Assumption~\ref{as coefficients} hold, let $b$ be a permissible flow (c.f. Definition~\ref{def permissible flow}) and $(\nu_{s,\cdot})_{s\geq 0}$ the corresponding measure flow resulting from Lemma~\ref{lemma vect field flow}.
Let $X^{\xi,\zeta}_{s,\cdot}, P^{\xi,\zeta}_{s,\cdot}$ be the forward and backward processes arising from data $(\xi,\zeta)$ with control $\nu_{s,\cdot} \in \mathcal V_2$.
Then
\begin{equation}
\begin{split}
\label{eq der of J wrt flow}
& \frac{d}{ds} J^\sigma(\nu_{s,\cdot}) \\
& = - \int_0^T \int
\bigg[\int_{\mathbb R^d
\times \mathcal S} \left(\nabla_a\frac{\delta H^0_t}{\delta m}\right)(X^{\xi,\zeta}_{s,t},P^{\xi,\zeta}_{s,t},\nu_{s,t},a,\zeta) \,\mathcal M(d\xi,d\zeta) \\
& \qquad \qquad \,\,\,\,\,+ \frac{\sigma^2}{2}\nabla_a U(a)+\frac{\sigma^2}{2}\nabla_a \log \nu_{s,t}(a)\bigg]
\cdot \bigg[ b_{s,t}(a) + \frac{\sigma^2}{2}\nabla_a \log \nu_{s,t}(a)\bigg]\, \nu_{s,t}(da) \,dt\,.		
\end{split}
\end{equation}
\end{theorem}


The complete proof of Theorem~\ref{corollary with measure flow} will come in Section~\ref{section pontryagin derivation} but a sketch is given in Section~\ref{sec sketch of convergence proof}.

\begin{theorem}[Necessary condition for optimality]
\label{thm necessary cond linear}
Fix $\sigma \geq 0$.
Let Assumption~\ref{as coefficients} hold.
If $\nu\in \mathcal V_2$ is (locally) optimal for $J^{\sigma, \mathcal M}$ given by~\eqref{eq objective J}, $X^{\xi,\zeta}$ and $P^{\xi,\zeta}$ are
the associated optimally controlled state and adjoint processes for data $(\xi, \zeta)$, given by~\eqref{eq process} and~\eqref{eq adjoint proc} respectively,
then for any other $\mu \in \mathcal V_2$ we have
\begin{enumerate}[i)]
\item For a.a. $t\in (0,T)$
\[
\int \bigg(\int_{\mathbb R^d \times \mathcal S}  \frac{\delta H^0_t}{\delta m}(X^{\xi,\zeta}_t, P^{\xi,\zeta}_t, \nu_t,a,\zeta) \,\mathcal M(d\xi,d\zeta) + \frac{\sigma^2}{2}\log \nu_t(a)-\frac{\sigma^2}{2}U(a)\bigg)\,(\mu_t-\nu_t)(da)
\geq 0 \,.
\]
\item For a.a. $t\in (0,T)$ there exists $\varepsilon > 0$ (small and depending on $\mu_t$) such that
\[
\int_{\mathbb R^d \times \mathcal S}H^\sigma_t (X^{\xi,\zeta}_t,P^{\xi,\zeta}_t,\nu_t+\varepsilon(\mu_t-\nu_t))\,\mathcal M(d\xi,d\zeta)
\geq \int_{\mathbb R^d \times \mathcal S}H^\sigma_t(X^{\xi,\zeta}_t,P^{\xi,\zeta}_t,Z^{\xi,\zeta}_t,\nu_t)\,\mathcal M(d\xi,d\zeta).
\]
In other words, the optimal relaxed control 	$\nu\in \mathcal V_2$ locally minimizes
the Hamiltonian $H^\sigma$.
\end{enumerate}
\end{theorem}
We remark that we also proof a sufficient optimality  condition, but do not state it here.

The proof of Theorem~\ref{thm necessary cond linear} is postponed until Section~\ref{section pontryagin derivation}.
However it tells us that if $\nu \in \mathcal V_2$ is (locally) optimal then
it must solve (together with the forward and adjoint processes)
the following system:
\begin{equation}
\label{eq pontryagin system}
\left\{
\begin{split}
\nu_t & = \argmin_{\mu \in \mathcal P_2(\mathbb R^p)} \int_{\mathbb R^d \times \mathcal S} H^\sigma_t(X^{\xi,\zeta}_t, P^{\xi,\zeta}_t, \mu, \zeta)\,\mathcal M(d\xi,d\zeta)\,,\\
dX^{\xi,\zeta}_t & =  \Phi_t(X^{\xi,\zeta}_t,\nu_t,\zeta)\,dt \,,\,\,\,t\in [0,T]\,,X^{\xi,\zeta}_0 = \xi\in \mathbb R^d\,,\zeta \in \mathcal S\,,\\
dP^{\xi,\zeta}_t & = -(\nabla_x H_t)(X^{\xi,\zeta}_t, P^{\xi,\zeta}_t, \nu_t,\zeta)\,dt\,,\,\,\,t\in [0,T]\,,\,\,\,P^{\xi,\zeta}_T = (\nabla_x g)(X^{\xi,\zeta}_T,\zeta)\,.
\end{split}	
\right.
\end{equation}

Let us now introduce the probability space on which we can develop the probabilistic formulation for the  gradient descent which will solve the system~\eqref{eq pontryagin system}.
Let there be a probability space $(\Omega^B, \mathcal F^B, \mathbb P^B)$ equipped with a $\mathbb R^p$-Brownian motion $B=(B_s)_{s\geq 0}$ and the filtration $\mathbb F^B = (\mathcal F_t^B)$ where
$\mathcal F_s^B := \sigma(B_u:0\leq u \leq s)$.
Let $(\Omega^\theta, \mathcal F^\theta, \mathbb P^\theta)$
be another probability space and let $(\theta^0_t)_{t\in [0,T]}$ be an $\mathbb R^p$-valued stochastic process on this space.
Let $\Omega := \Omega^B \times \Omega^\theta \times \mathbb R^d \times \mathcal S$, $\mathcal F := \mathcal F^B \otimes \mathcal F^\theta \otimes \mathcal B(\mathbb R^d) \otimes \mathcal B(\mathcal S)$
and $\mathbb P := \mathbb P^B \otimes \mathbb P^\theta \otimes \mathcal M$.
Let $I := [0,\infty)$.
Consider the mean-field system given by:
\begin{equation}\label{eq mfsgd}
	d \theta_{s,t}  =  -\left(\int_{\mathbb R^d \times \mathcal S} (\nabla_a h_t)(X^{\xi,\zeta}_{s,t},P^{\xi,\zeta}_{s,t},\theta_{s,t},\zeta)\,\mathcal M(d\xi, d\zeta) + \frac{\sigma^2}{2}(\nabla_a U)(\theta_{s,t}) \right)ds + \sigma dB_s\, \quad s \in I\,,
	\end{equation}
where $(\theta^0_{t})_{t\in[0,T]}$ is a given initial condition and where for each $s\geq 0$
\begin{equation}
\label{eq mfsgd 2}	
\left\{
\begin{aligned}
	\nu_{s,t} & = \mathcal L(\theta_{s,t})\,,\,\,\, t\in [0,T]\,,\\
	X^{\xi,\zeta}_{s,t} & = \xi + \int_0^t \Phi_r(X^{\xi,\zeta}_{s,r},\nu_{s,r},\zeta)\,dr \,,\,\,\, t\in [0,T]\,, \\
P^{\xi,\zeta}_{s,t} & = (\nabla_x g)(X^{\xi,\zeta}_T,\zeta) + \int_t^T (\nabla_x H_r)(X^{\xi,\zeta}_{s,r},P^{\xi,\zeta}_{s,r},\nu_{s,r},\zeta)\,dr \,,\,\,\, t\in [0,T] \,.\\
\end{aligned}
\right.
\end{equation}

\begin{assumption}[For existence, uniqueness and invariant measure]
\label{ass exist and uniq}
\hfill
\begin{enumerate}[i)]	
\item Let $\int_{0}^{T}\EE[|\theta^{0}_{t}|^2]\,dt<\infty$.	
\item Let $\nabla_aU$ be Lipschitz continuous in $a$ and moreover let there $\kappa>0$ such that:
\[
\left(\nabla_aU(a')-\nabla_aU(a)\right)\cdot\left(a'-a\right)\geq \kappa|a'-a|^2,\,a,a'\in\er^{\ControlS}\,.
\]
\item One of the two following conditions holds:
\begin{enumerate}[a)]\item $(x,\zeta)\mapsto \nabla_x g(x,\zeta)$ is bounded on $\mathbb R^d \times \mathcal S$,
\item $(t,a,\zeta)\mapsto \phi_t(0,a,\zeta)$ is bounded on $[0,T]\times\mathbb R^p \times \mathcal S$ and $\mathcal M$ has compact support.
\end{enumerate}
\end{enumerate}
\end{assumption}

%

To introduce topology on the space $\mathcal V_2$ we define the
integrated Wasserstein metric as follows.
Let $q=1,2$ and let $\mu, \nu \in \mathcal V_2$.
Then
\begin{equation}
\label{eq def wasserstein qT}
\mathcal W^T_q(\mu,\nu):=\left( \int_0^T \mathcal W_q(\mu_t,\nu_t)^q\,dt \right)^{1/q}\,,
\end{equation}
where $\mathcal W_q$ is the usual Wasserstein metric on $\mathcal P_q(\mathbb R^p)$.

In Lemma \ref{lem hamilton lipschitz} we show that Assumptions \ref{as coefficients} and \ref{ass exist and uniq} imply that  for any $\mathcal M \in \mathcal P_2(\mathbb R^d\times\mathcal S)$ there exists $L>0$ such that for all $a,a' \in \mathbb R^p$ and $\mu,\mu' \in \mathcal V_2$
\begin{equation*}
| (\nabla_a \mathbf h_t)(a,\mu,\mathcal M) - (\nabla_a \mathbf h_t)(a',\mu',\mathcal M) | \leq L \left( |a-a'| + \mathcal W^T_1(\mu,\mu')\right)\,.
\end{equation*}	

\begin{theorem}
\label{thm conv to inv meas rate}
Let Assumptions~\ref{as coefficients} and~\ref{ass exist and uniq} hold.
Then~\eqref{eq mfsgd}-\eqref{eq mfsgd 2} has a unique solution.
Moreover, assume that $J$ defined in \eqref{eq objective J} is bounded from below and that there exists $\bar \nu$ such that $J^{\sigma}(\bar \nu)<\infty$
and that $\sigma > 0$.
Then
\begin{enumerate}[i)]
\item $\argmin_{\nu \in \mathcal V_2}J^\sigma(\nu) \neq \emptyset$,
\item if $\nu^\star \in \argmin_{\nu \in \mathcal V_2}J^\sigma(\nu)$
then for a.a. $t\in (0,T)$ we have that
\[
\mathbf h_t(a,\nu^\star,\mathcal M) + \frac{\sigma^2}{2}\log(\nu^\star(a)) + \frac{\sigma^2}{2}U(a)\,\,\,\text{is constant for a.a. $a\in \mathbb R^p$}
\]
and $\nu^\star$ is an invariant measure for~\eqref{eq mfsgd}-\eqref{eq mfsgd 2}.
Moreover,
\item if $\sigma^2\kappa - 4L >0$ then $\nu^\star$ is unique and for any solution $\theta_{s,t}$ to~\eqref{eq mfsgd}-\eqref{eq mfsgd 2} we have that for all $s\geq 0$
\begin{equation}
\label{eq inv meas conv rate}
	\mathcal W^T_2(\mathcal L( \theta_{s,\cdot}),\nu^\star)^2 \leq e^{-\lambda s}   {\cal W}^T_2(\mathcal L( \theta_{0,\cdot}),\nu^\star)^2\,.
\end{equation}
\end{enumerate}
\end{theorem}
Theorem~\ref{thm conv to inv meas rate} is proved in Section~\ref{sec exist uniq inv}.
Let us point out that parts i) and ii) are proved in~\cite{hu2019meanode} and are included for completeness.
Part iii) is proved in~\cite{hu2019meanode} under different assumptions.

Let us now consider the particle approximation and propagation of chaos property.
Consider a sequence $(\xi^i,\zeta^i)_{i=1}^{N_1}$ of i.i.d copies of $(\xi,\zeta)$
and let $\mathcal M^{N_1} := \frac{1}{N_1}\sum_{j=1}^{N_1}\delta_{\xi^j, \zeta^j}$.
Furthermore, we assume that initial distribution of weights $(\theta_{0,\cdot}^i)_{i=0}^{N_2}$ are i.i.d copies of $(\theta^0_{\cdot})$ and that $(B^{i})_{i=1}^{N_2}$ are independent Brownian motions
and we extend our probability space to accommodate these.
For $s\in [0,S]$, $t\in[0,T]$ and $1\leq i\leq N_2$ define
\begin{equation}\label{eq:ParticleApprox-bis}
\begin{aligned}
\theta^i_{s,t}&=\theta^{i}_{0,t} -
\int_0^s
\left(
(\nabla_a \mathbf h_t)(\theta^i_{v,t}, \nu^{N_2}_{v,\cdot}, \mathcal M^{N_1})
+\frac{\sigma^2}{2}(\nabla_a U)(\theta^i_{v,t})
\right)
\,dv
+\sigma \ControlBrownian^i_s\,,
\end{aligned}
\end{equation}
where $\nu^{N_2}_v \in \Vv_2$ is the empirical measures defined as $
\nu^{N_2}_v=\frac{1}{N_2}\sum_{j=1}^{N_2}\delta_{\theta^{j}_v}$, and
where $\mathbf h_t$ is defined in~\eqref{eq fat h def}.

\begin{theorem}\label{thm:WellPosed&PropagationChaos} Let Assumptions~\ref{as coefficients} and~\ref{ass exist and uniq} hold. Fix $\lambda = \frac{\sigma^2 \kappa}{2 } -\frac{L}{2}(3 + T) + \frac{1}{2} $.
Then, there exists a unique solution to \eqref{eq:ParticleApprox-bis} with $\Ll(\theta^i_{s,\cdot})\in\Vv_2$. Moreover, for $(\theta^{i,\infty}_{s,t})$ solution to
\[
\theta^{i,\infty}_{s,t}=\theta^{i}_{0,t} - \int_0^s \bigg( \nabla_a {\bH}_{t}(\theta^{i,\infty}_{v,t},\Ll(\theta^{i,\infty}_{v,.}),\DataM) + \frac{\sigma^2}2 \nabla_aU(\theta^{i,\infty}_{v,t})\bigg)\,dv+\sigma \ControlBrownian^i_s,\,0\leq s\leq S,\,0\leq t\leq T,
\]
there exists $c$, independent of $s,N_1,N_2,p,d$, such that, for all $i$
\begin{align*}
&\int_0^T\EE\left[\left|\theta^{i}_{s,t}-\theta^{i,\infty}_{s,t}\right|^2\right]\,dt \leq \frac{c}{\lambda}(1-e^{- \lambda s})  \left(\frac{1}{N_1}+\frac{1}{N_2}\right)\,.
\end{align*}
\end{theorem}
The proof of Theorem \ref{thm:WellPosed&PropagationChaos} is given in Section \ref{section chaos and euler} (see Lemma \ref{lem:UnifParticleMoment} for the well-posedness and Theorem \ref{prop:PropagationChaosBis} for propagation of chaos).

Euler--Maruyama approximations with non-homogeneous time steps can be used to obtain an algorithm for the gradient descent~\eqref{eq:ParticleApprox-bis}.
Fix an increasing sequence $0=s_0 < s_1 < s_2 \cdots$.
and define the family of processes
 $(\widetilde{\theta}^{i}_{l,t})_{l \in \mathbb N, 0\leq t\leq T}$ satisfying, for  $i=1,\ldots,N_2$ and $l \in \mathbb N$,
\begin{equation}\label{eq:EulerScheme1}
\widetilde{\theta}^{i}_{l+1,t}=\widetilde \theta^{i}_{l,t}- \left((\nabla_a{\bH}_{t})
\left(\widetilde{\theta}^{i}_{l,t},\widetilde{{\nu}}^{N_2}_{l,\cdot},\DataM^{N_1}\right)+\frac{\sigma^2}{2}(\nabla_aU)(\widetilde \theta^i_{l,t})\right)\,(s_{l+1}-s_l)
+\sigma (B^i_{s_{l+1}} - B^i_{s_l})\,,
\end{equation}
where $\widetilde{{\nu}}^{N_2}_{v}=\frac{1}{N_2}\sum_{j_2=1}^{N_2}\delta_{\widetilde{\theta}^{j_2}_{v}}$.
The error estimate for this discretisation is given in the theorem below.  

\begin{theorem}\label{thm:EulerRate1}
Let Assumptions~\ref{as coefficients} and~\ref{ass exist and uniq} hold.
Assume also that $(s_l)_{l\geq 1}$ is a non-decreasing sequence of times, starting from $0$, such that the increments $(s_l-s_{l-1})_{l\geq 1}$ are positive and non-increasing, $\sum_{l\geq 1}(s_l-s_{l-1})^2<\infty$ and that, $\kappa$ is large enough so that
\begin{equation}\label{TimeStepRestrict}
\max_{l\geq 1}(s_l-s_{l-1})< \frac{\sigma^2\kappa-L}{2L\left(1+\frac{\sigma^2}{2}\Vert\nabla_aU\Vert^2_{Lip}\right)}.
\end{equation}
Then, for all $i$, $l$,
\begin{align*}
\mathbb E\left [\int_0^T |\theta^i_{s_l,t} - \widetilde \theta^i_{s_l,t}|^2\,dt \right ]
&\leq c\max_{1\leq l'\leq l}(s_{l'}-s_{l'-1})\left(1+\max_{0\leq s\leq s_l}\int_{0}^{T}\EE\left[\left|\theta^i_{s,t}\right|^2\right]\,dt\right).
\end{align*}

\end{theorem}
The proof of this theorem is in Appendix \ref{ssec:FullDiscrete}, where we will also briefly discuss the additional discretization along the time variable $t$.

%
%
%
\subsection{Generalisation error}

Recall $\nu_{S,\cdot}^{\sigma,N_1,N_2,\Delta s}$ denote distribution over parameter space induced by gradient algorithm when training with $N_1$ data samples, finite number of model parameters (the number of which is $\approx N_2 \times p \times n$, where $n$ is the number of grid points of $[0,T]$), the learning rate $\Delta s = \max_{0 < s_l < S} (s_l - s_{l-1})$, and training time $S$. In the next theorem we establish a bound for the generalisation error $J^{\mathcal M}(\nu_{S,\cdot}^{\sigma,N_1,N_2,\Delta s})$.

\begin{theorem} \label{th generalisation}
Let Assumptions \ref{as coefficients} and~\ref{ass exist and uniq} hold.	
Assume that $\sigma^2 \kappa$ is sufficiently large relative to $L$ and $T$.
Then there is $c>0$ independent of $\lambda$, $S$, $N_1$, $N_2$, $d$, $p$ and the time partition used in Theorem~\ref{thm:EulerRate1} such that
\[
\mathbb E\left[\Big|J^{0,\mathcal M}(\nu^{\star,\sigma})- J^{0,\mathcal M}(\nu_{S,\cdot}^{\sigma,N_1,N_2,\Delta s})\Big|^2\right] \leq c\left(e^{-\lambda S} + \frac1{N_1} + \frac1{N_2} + h \right)\,,
\]
where $h := \max_{0 < s_l < S} (s_l - s_{l-1})$. The generalisation error is given by

\[
J^{0,\mathcal M}(\nu_{S,\cdot}^{\sigma,N_1,N_2,\Delta s}) = J^{0,\mathcal M}(\nu_{S,\cdot}^{\sigma,N_1,N_2,\Delta s}) - J^{0,\mathcal M}(\nu^{\star,\sigma}) - \frac{\sigma^2}{2} \int_0^T\text{Ent}(\nu^{\star,\sigma}_t)\,dt +  \min_{\mu\in \mathcal V_2}J^{\sigma,\mathcal M}(\mu)\,,
\]
since $\min_{\mu\in \mathcal V_2}J^{\sigma,\mathcal M}(\mu) = J^{\sigma,\mathcal M}(\nu^{\ast,\sigma}) = J^{0,\mathcal M}(\nu^{\ast,\sigma}) + \frac{\sigma^2}2\int_0^T\text{Ent}(\nu^{\star,\sigma}_t)\,dt $.
\end{theorem}

Theorem \ref{th generalisation} tells us that the generalisation error consist of three errors: a) the numerical error of approximating an invariant measure with discrete time particle system, b) the relative entropy between the Gibbs measure $\gamma$ (a prior) and the $\nu^{\star,\sigma}$,  c) the minimum value of the cost function under population measure $J^{\sigma,\mathcal M}$.

The proof of Theorem~\ref{th generalisation} is postponed until Section~\ref{sec generalisation estimates}.

\subsection{Conditional generalisation error}

It is a common practice  to terminate the training when $J^{0,\mathcal M^{N_1}}$ is negligible and such models have been observed to generalise well. To link our results to such regime  we postulate the following assumption.

\begin{assumption}\label{ass epsilon loss}
	Fix $\varepsilon>0$ and $N_1>0$. 
	Assume that 
	 $J^{M_1}(\nu^{\star,\sigma,N_1}) \leq \varepsilon$.
\end{assumption}

\begin{theorem}\label{th conditional generalisation}
Let Assumptions \ref{as coefficients}, \ref{ass exist and uniq} and \ref{ass epsilon loss}
 hold. Then there is $c>0$ independent of $\lambda$, $S$, $N_1$, $N_2$, $d$, $p$ and the time partition used in Theorem~\ref{thm:EulerRate1} such that
 \[
 \mathbb E\left[\Big|J^{0,\mathcal M}(\nu_{S,\cdot}^{\sigma,N_1,N_2,\Delta s})\Big|^2\right] \leq  \varepsilon^2 + c\left(e^{-\lambda S} + \frac1{N_1} + \frac1{N_2} + h \right)\,.
 \]
\end{theorem}

\begin{proof}
We decompose the error as follows
\[
\begin{split}
& J^{0,\mathcal M}(\nu_{S,\cdot}^{\sigma,N_1,N_2,\Delta s}) \\
& = \left( J^{0,\mathcal M}(\nu_{S,\cdot}^{\sigma,N_1,N_2,\Delta s}) -
	J^{0,\mathcal M}(\nu^{\star,\sigma}) \right) + \left(J^{0,\mathcal M}(\nu^{\star,\sigma})  - J^{0,\mathcal M^{N_1}}(\nu^{\star,\sigma}) \right) +  J^{0,\mathcal M^{N_1}}(\nu^{\star,\sigma})\,.
\end{split}
\]
The bound on first term	follows from Theorem \ref{th generalisation}. Next there exists $c>0$ such that
\[
\mathbb E\left[ \left|J^{\mathcal M}(\nu_{s,\cdot})  - J^{\mathcal M^{N_1}}(\nu_{s,\cdot}) \right|^2\right]
=\mathbb E\left[ \left|\int \bar J(\nu_{s,\cdot},\xi,\zeta)( \mathcal M^{N_1}-\mathcal M)(d\xi,d\zeta)\right|^2\right]\leq
\frac{c}{N_1}.
\]
The proof is complete.
\end{proof}

\section{Proofs} \label{sec proofs}

\subsection{Outline of proof of Theorem~\ref{corollary with measure flow}}
\label{sec sketch of convergence proof}

Before we proceed to proofs of the main result in full generality  we present a sketch the the proof of Theorem~\ref{corollary with measure flow}. For brevity we take $f=0$.
Our aim is to solve this control problem using a (stochastic) gradient descent algorithm.
The goal is to find, for each $t\in [0,T]$ a vector field flow $(b_{s,t})_{s\geq 0}$ such that the measure flow $(\nu_{s,t})_{s \geq 0}$ given by
\begin{equation}
\label{eq flow in intro}
\partial_s \nu_{s,t} = \text{div} \bigg(\nu_{s,t} \,b_{s,t} + \frac{\sigma^2}{2}\nabla_a\nu_{s,t}\bigg)\,,\,\,\,s \geq 0\,,\,\, \nu_{0,t} = \nu^0_t \in \mathcal P_2(\mathbb R^p)\,,
\end{equation}
satisfies that $s\mapsto J^\sigma(\nu_{s,\cdot})$ is decreasing.
The aim is to compute $\frac{d}{ds}J(\nu_{s,\cdot})$, in terms of $b_{s,\cdot}$, and use this expression to choose $b_{s,\cdot}$ such that the derivative is negative.
Let us keep $(\xi,\zeta)$ fixed and use $X_{s,t}$ for the solution of~\eqref{eq process} when the control is given by $\nu_{s,\cdot}$.
Let $V_{s,t} := \frac{d}{ds}X_{s,t}$.
Let $B_{s,t} := b_{s,t} + \frac{\sigma^2}{2}\frac{\nabla_a  \nu_{s,t}}{\nu_{s,t}}$.
We will show (see Lemma~\ref{remark direct comp with flow 1}) that
\[
d V_{s,t}
= \left[(\nabla_x \Phi)(X_{s,t}, \nu_{s,t},\zeta)V_{s,t}
- \int (\nabla_a \phi_t)(X_{s,t},\nu_{s,t}, a,\zeta)\,B_{s,t}(a)\,\nu_{s,t}(da)\right]	\,dt \,.
\]
Since the equation is affine we can write its solution using an integrating factor
\[
V_{s,t}
 = -\int_0^t \int I(r,t;\nu_{s,\cdot})(\nabla_a \phi_r)(X_{s,r},a,\zeta)\,B_{s,r}(a)\,\nu_{s,r}(da)\,dr\,.	
\]
Further (see Lemma~\ref{remark direct comp with flow 2}) we can show that
\begin{equation*}
\begin{split}
\frac{d}{ds} \bar J^\sigma(\nu_{s,\cdot}, \xi,\zeta)  =
 - \int_0^T \int \bigg[  & (\nabla_x g)(X_{s,T},\zeta)
I(r,T;\xi,\nu_{s,\cdot})(\nabla_a \phi_r)(X_{s,r},a,\zeta)\\
&+\frac{\sigma^2}{2}\bigg(\frac{\nabla_a \nu_{s,t}(a)}{\nu_{s,t}(a)} + \nabla_a U(a)\bigg)\bigg]
\,B_{s,t}(a)\,\nu_{s,r}(da)\,dt\,.
\end{split}
\end{equation*}
Now we define
\begin{equation*}
P^{\xi,\zeta}_{s,r} := I^{\xi,\zeta}(r,T;\nu_{s,\cdot})(\nabla_x g)(X_{s,T},\zeta)
\end{equation*}
so that
\begin{equation*}
\frac{d}{ds} \bar J^\sigma(\nu_{s,\cdot}, \xi,\zeta)
= - \int_0^T \int \bigg[(\nabla_a \phi_r)(X_{s,r},a,\zeta)P_{s,r}\,
+\frac{\sigma^2}{2}\bigg(\frac{\nabla_a \nu_{s,t}(a)}{\nu_{s,t}(a)} + \nabla_a U(a)\bigg)\bigg] B_{s,r}(a)\nu_{s,r}(da)  \,dr \,.
\end{equation*}
After integrating over $\xi,\zeta$ w.r.t. $\mathcal M$ we get
\begin{equation*}
\begin{split}
\frac{d}{ds} J^\sigma(\nu_{s,\cdot}) = - \int_0^T \int \bigg[ & \int_{\mathbb R^d
\times \mathcal S}   (\nabla_a \phi_r)(X^{\xi,\zeta}_{s,r},a,\zeta)P^{\xi,\zeta}_{s,r}\,\mathcal M(d\xi, d\zeta) \\
&+\frac{\sigma^2}{2}\bigg(\frac{\nabla_a \nu_{s,t}(a)}{\nu_{s,t}(a)} + \nabla_a U(a)\bigg)
\bigg]
\, B_{s,r}(a)\,\nu_{s,r}(da)\,dr\,.			
\end{split}
\end{equation*}
At this point it is clear how to choose the flow to make this negative: we must take
\[
b_{s,r}(a) := \int_{\mathbb R^d
\times \mathcal S}   (\nabla_a \phi_r)(X^{\xi,\zeta}_{s,r},a,\zeta)P^{\xi,\zeta}_{s,r}\,\mathcal M(d\xi, d\zeta) +\frac{\sigma^2}{2} \nabla_a U(a)
\]
so that
\[
\begin{split}
& \frac{d}{ds} J^\sigma(\nu_{s,\cdot}) \\
& = -  \int_0^T \int \left|\int_{\mathbb R^d \times \mathcal S}  (\nabla_a \phi_r)(X^{\xi,\zeta}_{s,r},a,\zeta)P^{\xi,\zeta}_{s,r}  \, \mathcal M(d\xi,d\zeta)+\frac{\sigma^2}{2}\bigg(\frac{\nabla_a \nu_{s,t}(a)}{\nu_{s,t}(a)} + \nabla_a U(a)\bigg)\right|^2\,\nu_{s,r}(da)  \,dr \leq 0\,.
\end{split}
\]
Moreover, Theorem~\ref{thm conv to inv meas rate},  says that for $\sigma>0$ the $\nu^\star$ minimizing $J^\sigma$ exists and satisfies the following first order condition: for a.a. $t\in (0,T)$ we have
\[
\int_{\mathbb R^d \times \mathcal S}  \phi_t(X^{\xi,\zeta}_t(\nu^\star),a,\zeta)P^{\xi,\zeta}_t(\nu^\star)  \, \mathcal M(d\xi,d\zeta) + \frac{\sigma^2}{2}\log(\nu_t^\star(a)) + \frac{\sigma^2}{2}U(a)\,\,\,\text{is constant for a.a. $a\in \mathbb R^p$}\,,
\]
where $P^{\xi,\zeta}_t(\nu^\star) := I^{\xi,\zeta}(t,T;\nu^\ast)(\nabla_x g)(X^{\xi,\zeta}_T(\nu^\ast),\zeta)$.
Such first order condition is essentially another way of stating the necessary condition from Pontryagin optimality principle. 
Here we will provide a derivation from first principles for measure-valued control processes with entropy regularization. 
We note that a simple calculation using It\^o's formula  shows that the law in~\eqref{eq flow in intro} with the choice of $b_{s,r}$ made above is the law of
\[
d\theta_{s,t} = -\left(\int_{\mathbb R^d\times \mathcal S}(\nabla_a \phi_t)(X^{\xi,\zeta}_t(\mathcal L(\theta_{s,\cdot}), \theta_{s,t}, \zeta)\,P^{\xi,\zeta}_t(\mathcal L(\theta_{s,\cdot}))\,\mathcal M(d\xi,d\zeta) - \frac{\sigma^2}{2}U(\theta_{s,t})\right)\,ds+\sigma\,dB_s\
\]
i.e. $\mathcal L(\theta_{s,\cdot}) = \nu_{s,\cdot}$.

\subsection{Proof of Theorem~\ref{corollary with measure flow} and related optimality conditions}
\label{section pontryagin derivation}

In this section we derive Pontryagin's optimality principle for the relaxed control problem by proving Theorems~\ref{corollary with measure flow} and \ref{thm necessary cond linear}. 
Note that later we also prove the Pontryagin sufficient condition for optimality in Theorem~\ref{thm sufficient condition}. 
This is done purely for completeness as Theorem~\ref{thm sufficient condition} is not used anywhere in the analysis carried out in this work.

We will work with an additional control $(\mu_t)_{t \in [0,T]}$ and define $\nu^\varepsilon_t := \nu_t + \varepsilon(\mu_t - \nu_t)$.
In case $\sigma  \neq 0$ assume that $\mu_t$ are absolutely continuous w.r.t. the Lebesgue measure for all $t \in [0,T]$.
We will write $(X^{\xi,\zeta}_t)_{t\in[0,T]}$ for the solution of~\eqref{eq process} driven by $\nu$ and $(X^{\xi,\zeta,\varepsilon}_t)_{t\in[0,T]}$ for the solution of~\eqref{eq process} driven by $\nu^\varepsilon$ both with the data $(\xi,\zeta)$.
Moreover, let $V_0=0$ be fixed and
\begin{equation}
\label{eq V proc linear}
\begin{split}
dV^{\xi,\zeta}_t  = & \left[(\nabla_x \Phi_t)(X^{\xi,\zeta}_t, \nu_t,\zeta)V_t + \int \frac{\delta \Phi_t}{\delta m}(X^{\xi,\zeta}_t,\nu_t, a,\zeta)(\mu_t-\nu_t)(da)\right]	\,dt \,.	
\end{split}
\end{equation}
We observe that this is a linear equation.
Let
\[
V^{\xi,\zeta,\varepsilon}_t := \frac{X^{\xi,\zeta,\varepsilon}_t - X^{\xi,\zeta}_t}{\varepsilon} - V^{\xi,\zeta}_t
\,\,\, \text{i.e.}\,\,\,
X^{\xi,\zeta,\varepsilon}_t = X^{\xi,\zeta}_t + \varepsilon(V^{\xi,\zeta,\varepsilon}_t + V^{\xi,\zeta}_t)\,.
\]
\begin{lemma}
\label{lemma V process linear}
Under Assumption~\ref{as coefficients} we have
\[
\lim_{\varepsilon\searrow 0} \sup_{t\leq T} \bigg|\frac{X^{\xi,\zeta,\varepsilon}_t - X^{\xi,\zeta}_t}{\varepsilon} - V^{\xi,\zeta}_t\bigg|^2 = 0\,.
\]	
\end{lemma}
\begin{proof}
Since we will be working with $\xi,\zeta$ fixed we will omit them from the notation in the proof.
We note that
\[
\begin{split}
& \Phi_t(X^\varepsilon_t, \nu^\varepsilon_t) - \Phi_t(X_t,\nu_t)
= \Phi_t(X^\varepsilon_t, \nu^\varepsilon_t) - \Phi_t(X^\varepsilon_t, \nu_t) + \Phi_t(X^\varepsilon_t,\nu_t) - \Phi_t(X_t,\nu_t) \\
& = \varepsilon \int_0^1 (\nabla_x \Phi_t)(X_t + \lambda \varepsilon ( V^\varepsilon_t + V_t), \nu_t) (V^\varepsilon_t + V_t)\,d\lambda
+ \varepsilon\int_0^1 \int \frac{\delta \Phi_t}{\delta m}(X^\varepsilon_t, (1-\lambda)\nu^\varepsilon_t + \lambda \nu,a)(\mu_t-\nu_t)(da)\,d\lambda\,. 	
\end{split}
\]
Hence
\[
\begin{split}
& \frac1\varepsilon \left[\Phi_t(X^\varepsilon_t, \nu^\varepsilon_t) - \Phi_t(X_t,\nu_t)
- \varepsilon (\nabla_x \Phi_t)(X_t, \nu_t)V_t - \varepsilon \int \frac{\delta \Phi_t}{\delta m}(X_t,\nu_t, a)(\mu_t - \nu_t)(da)  \right]	\\
& =  \int_0^1 (\nabla_x \Phi_t)(X_t + \lambda \varepsilon ( V^\varepsilon_t + V_t), \nu_t) V^\varepsilon_t \,d\lambda
+  \int_0^1 \left[(\nabla_x \Phi_t)(X_t + \lambda \varepsilon ( V^\varepsilon_t + V_t), \nu_t) - (\nabla_x \Phi_t)(X_t,\nu_t) \right] V_t\,d\lambda \\
& + \int_0^1 \int \left[ \frac{\delta \Phi_t}{\delta m}(X^\varepsilon_t, (1-\lambda)\nu^\varepsilon_t + \lambda \nu,a) - \frac{\delta \Phi_t}{\delta m}(X_t,\nu_t,a) \right](\mu_t-\nu_t)(da)\,d\lambda =: I^{(0)}_t + I^{(1)}_t + I^{(2)}_t =: I_t.
\end{split}
\]
Note that
\[
dV^\varepsilon_t = \frac1\varepsilon[dX^\varepsilon_t - dX_t] - dV_t
\]
and so we then see that $
d|V^\varepsilon_t|^2 =  2V^\varepsilon_t I_t\,dt$.
Hence we have that
\[
\begin{split}
\sup_{s\leq t}|V^\varepsilon_s|^2  & \leq  c \int_0^t |I_s|^2\,ds \,.
\end{split}
\]
Let us now consider
\[
\begin{split}
\int_0^T |I^{(2)}_t|^2\,dt
 \leq &  2\int_0^T \left| \int_0^1 \int \left[ \frac{\delta \Phi_t}{\delta m}(X^\varepsilon_t, (1-\lambda)\nu^\varepsilon_t + \lambda \nu,a) - \frac{\delta \Phi_t}{\delta m}(X_t,\nu_t,a) \right]\mu_t(da)\,d\lambda \right|^2\,dt\\
&
+ 2\int_0^T \left| \int_0^1 \int \left[ \frac{\delta \Phi_t}{\delta m}(X^\varepsilon_t, (1-\lambda)\nu^\varepsilon_t + \lambda \nu,a) - \frac{\delta \Phi_t}{\delta m}(X_t,\nu_t,a) \right]\nu_t(da)\,d\lambda \right|^2\,dt \,.	
\end{split}
\]
Taking the terms separately and using the fact that we are working with probability measures we see that
\[
\begin{split}
& \left| \int_0^1 \int \left[ \frac{\delta \Phi_t}{\delta m}(X^\varepsilon_t, (1-\lambda)\nu^\varepsilon_t + \lambda \nu,a) - \frac{\delta \Phi_t}{\delta m}(X_t,\nu_t,a) \right]\mu_t(da)\,d\lambda \right| \\
& \leq  \int_0^1 \int \left| \frac{\delta \Phi_t}{\delta m}(X^\varepsilon_t, (1-\lambda)\nu^\varepsilon_t + \lambda \nu,a) - \frac{\delta \Phi_t}{\delta m}(X_t,\nu_t,a) \right|\mu_t(da)\,d\lambda \\
& = \int_0^1 \int \left| \phi_t(X^\varepsilon_t,a)  - \phi(X_t,a)
+ \int [ \phi(X_t,a') - \phi_t(X^\varepsilon_t,a')]\mu_t(da')   \right|\mu_t(da)\,d\lambda\\
& \leq \int_0^1 \int \left[ | \phi_t(X^\varepsilon_t,a)  - \phi_t(X_t,a) |
+ \left|\int | \phi_t(X_t,a') - \phi_t(X^\varepsilon_t,a')|\mu_t(da')\right|  \right]\mu_t(da)\,d\lambda  \leq L|X^\varepsilon_t  - X_t |\,.
\\ 	
\end{split}
\]
Hence
\[
\int_0^T |I^{(2)}_t|^2\,dt \leq 4L^2 \int_0^T |X^\varepsilon_t  - X_t |^2\,dt \,.
\]
Similarly
\[
\int_0^T |I^{(1)}_t|^2\,dt  \leq 4L^2 \int_0^T |X^\varepsilon_t  - X_t |^2\,dt \,.
\]
This and the assumption that the derivatives w.r.t. the spatial variable are bounded lead to
\[
\begin{split}
\sup_{s\leq t}|V^\varepsilon_s|^2  & \leq  c \int_0^t |V^\varepsilon_s|^2\,ds + \int_0^T |X^\varepsilon_s - X_s|^2 \,ds\,.	
\end{split}
\]
Finally note that
\[
\delta_\varepsilon := \int_0^T |X^\varepsilon_t  - X_t|^2\,dt \to 0 \,\,\,\text{as} \,\,\, \varepsilon \to \infty\,.
\]
So
\[
\sup_{s\leq t'}|V^\varepsilon_s|^2 \leq c\int_0^{t'} \sup_{s\leq t}|V^\varepsilon_s|^2 \,dt + \delta_\varepsilon
\]
and by Gronwall's lemma $\sup_{s\leq t}|V^\varepsilon_s|^2  \leq \delta_\varepsilon e^{cT} \to 0$ as $\varepsilon \to \infty$.
\end{proof}

Note that we effectively have $\frac{d}{d\varepsilon} X_t^{\xi,\zeta,\nu+\varepsilon(\mu-\nu)}\big|_{\varepsilon=0} = V^{\xi,\zeta}_t$ and moreover due to the affine structure of~\eqref{eq V proc linear} the solution can be expressed as
\begin{equation}
\label{eq proc V solved}
\begin{split}
\frac{d}{d\varepsilon} X_t^{\xi,\zeta,\nu+\varepsilon(\mu-\nu)}\big|_{\varepsilon=0} & = \int_0^t \int I^{\xi,\zeta}(r,t;\nu)\frac{\delta \Phi_r}{\delta m}(X^{\xi,\zeta,\nu}_r,\nu_r,a,\zeta) \,(\mu_r - \nu_r)(da)\,dr\,.
\end{split}	
\end{equation}
\begin{lemma}
\label{remark direct comp with flow 1}
Let $\sigma > 0$ be fixed and let Assumption~\ref{as coefficients} hold. 
Let $b$ be a permissible flow (c.f. Definition~\ref{def permissible flow})
with $\nu_{s,\cdot}\in \mathcal V_2$ the corresponding solution from Lemma~\ref{lemma vect field flow}.
Let $X^{\xi,\zeta}_{s,\cdot}$ be the solution to~\eqref{eq process} from data $(\xi,\zeta)$ with control $\nu_{s,\cdot} \in \mathcal V_2$.
Let $V^{\xi,\zeta}_{s,t} := \frac{d}{ds}X^{\xi,\zeta}_{s,t}$.
Then
\begin{equation}
\label{eq V proc solved 2}
V^{\xi,\zeta}_{s,t} = - \int_0^t I^{\xi,\zeta}(r,t;\nu_{s,\cdot}) \bigg(\nabla_a\frac{\delta \Phi_r}{\delta m}\bigg)(X^{\xi,\zeta}_{s,r},\nu_{s,r},a,\zeta) \cdot\bigg( b_{s,r}(a) + \frac{\sigma^2}{2}\nabla_a \log \nu_{s,r}(a)\bigg) \,\nu_{s,r}(da)\,dr\,.	
\end{equation}
and this can be written in differential form (w.r.t time $t$) as
\begin{equation}
\label{eq V proc flow}
\begin{split}
d V^{\xi,\zeta}_{s,t}
= & \bigg[(\nabla_x \Phi_t)(X^{\xi,\zeta}_{s,t}, \nu_{s,t},\zeta)V^{\xi,\zeta}_{s,t}\\
& - \int \bigg(\nabla_a\frac{\delta \Phi_t}{\delta m}\bigg)(X^{\xi,\zeta}_{s,t},\nu_{s,t}, a,\zeta)\cdot\bigg( b_{s,t}(a) + \frac{\sigma^2}{2}\nabla_a \log \nu_{s,t}(a)\bigg)\,\nu_{s,t}(da)\bigg]	\,dt \,.
\end{split}
\end{equation}
\end{lemma}

\begin{proof}
We will keep $(\xi,\zeta)$ fixed for the moment and hence omit it from the notation.
Let us now fix as $\nu_t = \nu_{s_0,t}$ and $\mu_t = \nu_{s_1,t}$ for all $t\in [0,T]$.
Define $\nu^\varepsilon_t := \nu_t + \varepsilon(\mu_t - \nu_t)$ and
$\mu^\varepsilon_t := \mu_t + \varepsilon(\mu_t - \nu_t)$, so that $\mu^\varepsilon_t - \nu^\varepsilon_t = \mu_t - \nu_t$.
From the Fundamental Theorem of Calculus
\[
\begin{split}
X_t(\mu) - X_t(\nu) & = \int_0^1 \lim_{\delta \searrow 0} \frac1\delta \Big(X_t(\nu + (\varepsilon+\delta)(\mu-\nu)) - X_t(\nu+\varepsilon(\mu-\nu)) \Big)\,d\varepsilon 	\\
& = \int_0^1 \lim_{\delta \searrow 0} \frac1\delta \Big(X_t(\nu^\varepsilon + \delta(\mu^\varepsilon-\nu^\varepsilon)) - X_t(\nu^\varepsilon) \Big)\,d\varepsilon \,.
\end{split}
\]
Due to Lemma~\ref{lemma V process linear} and~\eqref{eq proc V solved} we see that
\[
 \lim_{\delta \searrow 0} \frac1\delta \Big(X_t(\nu^\varepsilon + \delta(\mu^\varepsilon-\nu^\varepsilon)) - X_t(\nu^\varepsilon) \Big)
= \int_0^t \int I^{\xi,\zeta}(r,t;\nu^\varepsilon)\,\frac{\delta \Phi_r}{\delta m}(X_r(\nu^\varepsilon),\nu^\varepsilon_r,a)\,(\mu_r - \nu_r)(da)\,dr\,.
\]
Hence
\[
X_t(\mu) - X_t(\nu) = \int_0^1 \int_0^t \int I^{\xi,\zeta}(r,t;\nu^\varepsilon)\,\frac{\delta \Phi_r}{\delta m}(X_r(\nu^\varepsilon),\nu^\varepsilon_r,a)\,(\mu_r - \nu_r)(da)\,dr \, d\varepsilon \,.
\]
Now fix $s\geq 0$ and take $\mu_t = \nu_{s+h,t}$, $\nu_t = \nu_{s,t}$
and $\nu_{s,t}^{\varepsilon,h} = \nu_{s,t}+ \varepsilon(\nu_{s+h,t} - \nu_{s,t})$.
Note that $\nu_{s,t}^{\varepsilon,h} \to \nu_{s,t}$ as $h \searrow 0$ and hence 
\[
I^{\xi,\zeta}(r,t;\nu^{\varepsilon,h}_{s,\cdot})\,\frac{\delta \Phi_r}{\delta m}(X_r(\nu^{\varepsilon,h}_{s,\cdot}),\nu^{\varepsilon,h}_{s,r},a) 
\to 
I^{\xi,\zeta}(r,t;\nu_{s,\cdot})\,\frac{\delta \Phi_r}{\delta m}(X_r(\nu_{s,\cdot}),\nu_{s,r},a))
\]
as $h \searrow 0$. 
Moreover
\[
\begin{split}
\frac{d}{ds}X_t(\nu_{s,\cdot}) & = \lim_{h\searrow 0} \frac1h \Big(X_t(\nu_{s+h,\cdot}) - X_t(\nu_{s,\cdot})  \Big) \\
& = \lim_{h\searrow 0} \frac1h \int_0^1 \int_0^t \int I^{\xi,\zeta}(r,t;\nu^{\varepsilon,h}_{s,\cdot})\,\frac{\delta \Phi_r}{\delta m}(X_r(\nu^{\varepsilon,h}_{s,\cdot}),\nu^{\varepsilon,h}_{s,r},a)\,(\nu_{s+h,r} - \nu_{s,r})(da)\,dr \, d\varepsilon \\
& =  \int_0^1 \int_0^t \int  \lim_{h\searrow 0} \bigg[I^{\xi,\zeta}(r,t;\nu^{\varepsilon,h}_{s,\cdot})\,\frac{\delta \Phi_r}{\delta m}(X_r(\nu^{\varepsilon,h}_{s,\cdot}),\nu^{\varepsilon,h}_{s,r},a)\,\frac1h (\nu_{s+h,r} - \nu_{s,r})(a)\bigg]da\,dr \, d\varepsilon \,.
\end{split}
\]
Thus from Lemma~\ref{lemma vect field flow} which states that the control laws evolve according to a gradient flow we get 
\[
\begin{split}
& \frac{d}{ds}X_t(\nu_{s,\cdot}) \\
& = \int_0^t \int I^{\xi,\zeta}(r,t;\nu_{s,\cdot})\,\frac{\delta \Phi_r}{\delta m}(X_r(\nu_{s,\cdot}),\nu_{s,r},a)) \nabla_a \cdot \bigg( b_{s,r}(a) + \frac{\sigma^2}{2}\frac{\nabla_a \nu_{s,r}(a)}{\nu_{s,r}(a)}\bigg)\nu_{s,r}(a) \,da\,dr \\
& = -\int_0^t \int I^{\xi,\zeta}(r,t;\nu_{s,\cdot})\,\bigg(\nabla_a \frac{\delta \Phi_r}{\delta m}\bigg)(X_r(\nu_{s,\cdot}),\nu_{s,r},a)  \cdot\bigg( b_{s,r}(a) + \frac{\sigma^2}{2}\nabla_a \log \nu_{s,r}(a)\bigg) \nu_{s,r}(a)\,da\,dr\,,
\end{split}
\]
where the last equality is due to integration by parts.
\end{proof}

\begin{lemma}
\label{lemma derivative of entropy along flow}
Let $\sigma > 0$ be fixed and let Assumption~\ref{as coefficients} hold. 
Let $b$ be a permissible flow (c.f. Definition~\ref{def permissible flow})
with $\nu_{s,\cdot}\in \mathcal V_2$ the corresponding solution from Lemma~\ref{lemma vect field flow}.
Then
\[
\begin{split}
d \text{Ent}(\nu_{s,t}) 
& =  -\int\bigg( \nabla_a \log \nu_{s,t} + \nabla_a U \bigg)\cdot \left(  b_{s,t}+\frac{\sigma^2}{2} \nabla_a \log \nu_{s,t} \right)\,\nu_{s,t}(da)\,ds\,.
\end{split}
\]
	
\end{lemma}

\begin{proof}
The key part of the proof is done in~\cite[Proof of Proposition 2.4]{hu2019mean}.	
\end{proof}

\begin{lemma}
\label{lemma diff of J flat}
Under Assumption~\ref{as coefficients} the mapping $\nu \mapsto \bar J^0(\nu, \xi,\zeta)$ defined by~\eqref{eq objective bar J} satisfies
\[
\begin{split}
&\frac{d}{d\varepsilon} \bar J^0  \left((\nu_t + \varepsilon(\mu_t-\nu_t)_{t\in[0,T]}, \xi,\zeta\right)\bigg|_{\varepsilon=0} \\
& =   \int_0^T \left[\int f_t(X^{\xi,\zeta}_t,a,\zeta)\,(\mu_t-\nu_t)(da) + \int (\nabla_x f_t)(X^{\xi,\zeta}_t,a,\zeta) V^{\xi,\zeta}_t \, \nu_t(da) \right]\,dt + (\nabla_x g)(X^{\xi,\zeta}_T,\zeta)V^{\xi,\zeta}_T \,.	
\end{split}
\]	
\end{lemma}
\begin{proof}
We need to consider the difference quotient for $\bar J^0$ and to that end we consider
\[
\begin{split}
I_\varepsilon := &  \frac1\varepsilon \int_0^T \left[F(X^{\xi,\zeta,\varepsilon}_t, \nu^\varepsilon_t, \zeta) - F(X^{\xi,\zeta,\varepsilon}_t, \nu_t, \zeta) + F(X^{\xi,\zeta,\varepsilon}_t, \nu_t, \zeta) - F(X^{\xi,\zeta}_t, \nu_t, \zeta) \right]\,dt \\
& = \int_0^T \int_0^1 \int \left[f(X^{\xi,\zeta,\varepsilon}_t, a, \zeta) - \int f(X^{\xi,\zeta,\varepsilon}_t, a', \zeta)\,\nu_t(da')  \right](\mu_t - \nu_t)(da)	\,d\lambda\,dt\\
& \qquad + \int_0^T \int_0^1 \int (\nabla_x f)(X^{\xi,\zeta}_t + \lambda \varepsilon (V^{\xi,\zeta,\varepsilon}_t + V^{\xi,\zeta}_t),a, \zeta)(V^{\xi,\zeta,\varepsilon}_t - V^{\xi,\zeta}_t)\,\nu_t(da)\,d\lambda\,dt\\
& = \int_0^T \left[ \int f(X^{\xi,\zeta,\varepsilon}_t, a, \zeta)(\mu_t - \nu_t)(da) 	\right]\,dt\\
& \qquad + \int_0^T \int_0^1\int (\nabla_x f)(X^{\xi,\zeta}_t + \lambda \varepsilon (V^{\xi,\zeta,\varepsilon}_t + V^{\xi,\zeta}_t),a, \zeta)(V^{\xi,\zeta,\varepsilon}_t - V^{\xi,\zeta}_t)\,\nu_t(da)\,d\lambda\,dt\,.
\end{split}
\]
Using Lebesgue's dominated convergence theorem and Lemma~\ref{lemma V process linear} we get
\[
\lim_{\varepsilon \to 0} I_\varepsilon = \int_0^T \int f_t(X^{\xi,\zeta}_t, a,\zeta)(\mu_t-\nu_t)(da)\,dt
+ \int_0^T \int (\nabla_x f_t)(X^{\xi,\zeta}_t,a,\zeta) V^{\xi,\zeta}_t \,\nu_t(da)\,dt  \,.
\]
The term involving $g$ can be treated using the differentiability assumption and Lemma~\ref{lemma V process linear}.
\end{proof}

\begin{lemma}
\label{remark direct comp with flow 2}
Let $\sigma > 0$ be fixed and let Assumption~\ref{as coefficients} hold. 
Let $b$ be a permissible flow (c.f. Definition~\ref{def permissible flow})
with $\nu_{s,\cdot}\in \mathcal V_2$ the corresponding solution from Lemma~\ref{lemma vect field flow}.
Let $X^{\xi,\zeta}_{s,\cdot}$ be the solution to~\eqref{eq process} from data $(\xi,\zeta)$ with control $\nu_{s,\cdot} \in \mathcal V_2$.
Let $B_{s,t} := b_{s,t} + \frac{\sigma^2}{2}\nabla_a \log \nu_{s,t}$.
Then
\[
\begin{split}
\frac{d}{ds} & \bar J^0(\nu_{s,\cdot}, \xi,\zeta)  =  -\int_0^T \int  (\nabla_a f_t)(X_{s,t}, a, \zeta) B_{s,t}(a)\, \nu_{s,t}(da) \\
& - \int_0^T (\nabla_x F)(X_{s,t},\nu_{s,t},\zeta) \int_0^t \int \bigg(\nabla_a \frac{\delta \Phi_r}{\delta m}\bigg)(X_r(\nu_{s,\cdot}),\nu_{s,r},a)  B_{s,r}(a)\,\nu_{s,r}(da)\,dr \,dt  \\
& - (\nabla_x g)(X_{s,T},\zeta) \int_0^T \int \bigg(\nabla_a \frac{\delta \Phi_r}{\delta m}\bigg)(X_r(\nu_{s,\cdot}),\nu_{s,r},a)  B_{s,t}(a)\,\nu_{s,r}(da)\,dt\,.
\end{split}
\]	
\end{lemma}

\begin{proof}
We will keep $(\xi,\zeta)$ fixed for the moment and hence omit it from the notation.
Let us now fix as $\nu_t = \nu_{s_0,t}$ and $\mu_t = \nu_{s_1,t}$ for all $t\in [0,T]$.
Define $\nu^\varepsilon_t := \nu_t + \varepsilon(\mu_t - \nu_t)$ and
$\mu^\varepsilon_t := \mu_t + \varepsilon(\mu_t - \nu_t)$, so that $\mu^\varepsilon_t - \nu^\varepsilon_t = \mu_t - \nu_t$.
From the Fundamental Theorem of Calculus
\[
\begin{split}
\bar J^0(\mu) - \bar J^0(\nu) & = \int_0^1 \lim_{\delta \searrow 0} \frac1\delta \Big(\bar J^0(\nu + (\varepsilon+\delta)(\mu-\nu)) - \bar J^0(\nu+\varepsilon(\mu-\nu)) \Big)\,d\varepsilon 	\\
& = \int_0^1 \lim_{\delta \searrow 0} \frac1\delta \Big(\bar J^0(\nu^\varepsilon + \delta(\mu^\varepsilon-\nu^\varepsilon)) - \bar J^0(\nu^\varepsilon) \Big)\,d\varepsilon \,.
\end{split}
\]
Due to Lemma~\ref{lemma diff of J flat}
we see that
\[
\begin{split}
&\bar J^0(\mu) - \bar J^0(\nu)\\
& = \int_0^1 \bigg(\int_0^T \left[\int f_t(X_t(\nu^\varepsilon),a)\,(\mu_t-\nu_t)(da) + (\nabla_x F_t)(X_t(\nu^\varepsilon)) V^\varepsilon_t \right]\,dt + (\nabla_x g)(X_T(\nu^\varepsilon))V^\varepsilon_T \bigg)\, d\varepsilon \,.
\end{split}
\]
Now fix $s\geq 0$ and take $\mu_t = \nu_{s+h,t}$, $\nu_t = \nu_{s,t}$
and $\nu_{s,t}^{\varepsilon,h} = \nu_{s,t}+ \varepsilon(\nu_{s+h,t} - \nu_{s,t})$ and note that $\nu_{s,t}^{\varepsilon,h} \to \nu_{s,t}$ as $h \searrow 0$.
Then
\[
\begin{split}
& \frac{d}{ds}\bar J^0(\nu_{s,\cdot})  = \lim_{h\searrow 0} \frac1h \Big(\bar J^0(\nu_{s+h,\cdot}) - \bar J^0(\nu_{s,\cdot})  \Big) \\
& =  \int_0^1 \bigg(\int_0^T \int  \lim_{h\searrow 0} \bigg[f_t(X_t(\nu^{\varepsilon,h}_{s,.}),a)\,\frac1h (\nu_{s+h,t} - \nu_{s,t})\bigg](da)\,dt  \\
& \qquad + \int_0^T \int_0^t \int \lim_{h\searrow 0} \bigg[(\nabla_xF_t)(X_t(\nu^{\varepsilon,h}_{s,.})I(r,t;\nu^{\varepsilon,h}_{s,\cdot})\frac{\delta \Phi_r}{\delta m}(X_r(\nu^{\varepsilon,h}_{s,.}),\nu^{\varepsilon,h}_{s,r},a)\frac1h (\nu_{s+h,r} - \nu_{s,r})\bigg](da)\,dr\,dt \\
& \qquad + \int_0^T \int \lim_{h\searrow 0} \bigg[(\nabla_x g)(X_T(\nu^{\varepsilon,h}_{s,.})I(r,T;\nu^{\varepsilon,h}_{s,\cdot})\frac{\delta \Phi_r}{\delta m}(X_r(\nu^{\varepsilon,h}_{s,.}),\nu^{\varepsilon,h}_{s,r},a)\frac1h (\nu_{s+h,r} - \nu_{s,r})\bigg](da)\,dr  \bigg) \, d\varepsilon
\end{split}
\]
Thus from Lemma~\ref{lemma vect field flow} which states that the control laws evolve according to a gradient flow we get 
\[
\begin{split}
& \int  \lim_{h\searrow 0} \bigg[f_t(X_t(\nu^{\varepsilon,h}_{s,.}),a)\,\frac1h (\nu_{s+h,t} - \nu_{s,t})\bigg](da) \\
& = \int  f_t(X_t(\nu_{s,.}),a)\,\nabla_a \cdot \bigg( b_{s,r}(a) + \frac{\sigma^2}{2}\frac{\nabla_a \nu_{s,r}(a)}{\nu_{s,r}(a)}\bigg)\nu_{s,r}(a)\,da\\
& = - \int (\nabla_a f_t)(X_t(\nu_{s,.}),a) \cdot B_{s,t}\nu_{s,r}(a)\,da \,,
\end{split}
\]
where the last equality is due to integration by parts.
The integrands in the other two integrals are treated similarly and so the result follows.
\end{proof}

\begin{proof}[Proof of Theorem~\ref{corollary with measure flow}]
Recall that for each $(\xi,\zeta)$ and each $s\geq 0$ we have that $X^{\xi,\zeta}_{s,\cdot}$ is the forward process arising in~\eqref{eq process} with control $\nu_{s,\cdot} \in \mathcal V_2$.
From Lemmas~\ref{lemma derivative of entropy along flow} and~\ref{remark direct comp with flow 2} we have that
\begin{equation}
\label{eq min with flow calc 1}
\begin{split}
& \frac{d}{ds} \bar J^\sigma(\nu_{s,\cdot}, \xi,\zeta)  =  -\int_0^T \int \bigg[ (\nabla_a f)(X_{s,t}, a, \zeta) + \frac{\sigma^2}{2}\frac{\nabla_a \nu_{s,t}(a)}{\nu_{s,t}(a)} + \nabla_a U(a) \bigg]B_{s,t}(a)\, \nu_{s,t}(da)\,dt \\
& - \int_0^T (\nabla_x F)(X_{s,t},\nu_{s,t},\zeta) \int_0^t \int I(r,t,\nu_{s,\cdot})\bigg(\nabla_a \frac{\delta \Phi_r}{\delta m}\bigg)(X_{s,r},a)\,B_{s,r}(a)\nu_{s,r}(da)\,dr  \,dt  \\
& - (\nabla_x g)(X_{s,T},\zeta) \int_0^T \int I(t,T,\nu_{s,\cdot})\bigg(\nabla_a \frac{\delta \Phi_t}{\delta m}\bigg)(X_{s,t},a)\,B_{s,t}(a)\nu_{s,t}(da)\,dt\,.
\end{split}
\end{equation}
We now perform a change of order of integration in the triangular region
\[
\{(r,t)\in \mathbb R^2: 0\leq t \leq T \,,\,\, 0\leq r \leq t\} = \{(r,t)\in \mathbb R^2 :0\leq r \leq T\,,\,\, r \leq t \leq T \}
\]
which transforms~\eqref{eq min with flow calc 1} into
\begin{equation}
\label{eq min with flow calc 2}
\begin{split}
&\frac{d}{ds} \bar J^\sigma(\nu_{s,\cdot}, \xi,\zeta)
=  -\int_0^T \int \bigg[ (\nabla_a f_t)(X_{s,t}, a, \zeta) + \frac{\sigma^2}{2}\frac{\nabla_a \nu_{s,t}(a)}{\nu_{s,t}(a)} + \nabla_a U(a) \bigg]B_{s,t}(a)\, \nu_{s,t}(da)\,dt \\
& - \int_0^T \int_r^T \int I(r,t,\nu_{s,\cdot})\bigg(\nabla_a \frac{\delta \Phi_r}{\delta m}\bigg)(X_{s,r},a)(\nabla_x F)(X_{s,t},\nu_{s,t},\zeta) \,B_{s,r}(a)\,\nu_{s,r}(da)\,dt  \,dr  \\
& -  \int_0^T \int I(t,T,\nu_{s,\cdot})\bigg(\nabla_a \frac{\delta \Phi_t}{\delta m}\bigg)(X_{s,t},a)(\nabla_x g)(X_{s,T},\zeta)\,B_{s,t}(a)\nu_{s,r}(da)\,dt\,.
\end{split}
\end{equation}
Now we recall that for each $(\xi,\zeta)$ and each $s\geq 0$ we have that $P^{\xi,\zeta}_{s,\cdot}$ is the backward process arising in~\eqref{eq adjoint proc} from the  forward process $X^{\xi,\zeta}_{s,\cdot}$ and the control $\nu_{s,\cdot} \in \mathcal V_2$.
We can see that since~\eqref{eq adjoint proc} is affine
\begin{equation*}
P_{s,r} = (\nabla_x g)(X_{s,T},\zeta)I(r,T,\nu_{s,\cdot}) + \int_r^T 	(\nabla_x F)(X_{s,t},\nu_{s,t},\zeta)I(r,T,\nu_{s,\cdot})\,dt
\end{equation*}
so that~\eqref{eq min with flow calc 2} can be written as
\begin{equation}
\label{eq min with flow calc 3}
\begin{split}
\frac{d}{ds} \bar J^\sigma(\nu_{s,\cdot}, \xi,\zeta)
= & -\int_0^T \int \bigg[ (\nabla_a f_t)(X_{s,t}, a, \zeta) + \frac{\sigma^2}{2}\frac{\nabla_a \nu_{s,t}(a)}{\nu_{s,t}(a)} + \nabla_a U(a) \bigg]B_{s,t}(a)\, \nu_{s,t}(da)\,dt \\
& - \int_0^T \int (\nabla_a \phi_r)(X_{s,r},a,\zeta)P_{s,r}\, B_{s,r}(a)\nu_{s,r}(da)  \,dr \,.
\end{split}
\end{equation}
Recalling the Hamiltonian defined in~\eqref{eq hamiltonian}  completes the proof.
\end{proof}

\begin{lemma}
\label{lemma der of J as hamiltonian flat}
Under Assumption~\ref{as coefficients} we have that
\begin{equation*}
\frac{d}{d\varepsilon} \bar J^0\left((\nu_t + \varepsilon(\mu_t-\nu_t)_{t\in[0,T]},\xi,\zeta\right)\bigg|_{\varepsilon=0}
=\int_0^T \left[\int \frac{\delta H^0_t}{\delta m}(X^{\xi,\zeta}_t,P^{\xi,\zeta}_t,\nu_t,a,\zeta) (\mu_t-\nu_t)(da)   \right]\,dt \,.	
\end{equation*}	
\end{lemma}

\begin{proof}[Proof of Lemma~\ref{lemma der of J as hamiltonian flat}]
We first observe (omitting the superscript $\nu,\xi,\zeta$ from the adjoint process as these are fixed) that due to~\eqref{eq adjoint proc}, ~\eqref{eq V proc linear} and the fact that $V_0 = 0$, we have
\[
\begin{split}
& P_T V_T  =  \int_0^T P_t \, dV_t + \int_0^T V_t \,dP_t  \\	
& =  \int_0^T P_t \left[\int (\nabla_x \phi_t)(X_t,a, \zeta)\,\nu_t(da) V_t + \int \phi_t(X_t,a, \zeta)\,(\mu_t  - \nu_t)(da)\right]\,dt - \int_0^T V_t  (\nabla_x h_t)(X_t,P_t,a,\zeta)\,\nu_t(a)\,dt\,. \\
\end{split}
\]
Hence due to~\eqref{eq hamiltonian} we get
\[
\begin{split}
 P_T V_T  = & \int_0^T P_t \left[\int (\nabla_x \phi_t)(X_t,a, \zeta)\,\nu_t(da) V_t + \int \phi(X_t,a, \zeta)\,(\mu_t-\nu_t) (da)\right]\,dt \\
& - \int_0^T \int   \left[  P_t (\nabla_x \phi_t)(X_t,a, \zeta)  V_t  +  V_t (\nabla_x f_t)(X_t,a, \zeta)  \right]\,\nu_t(a)\,dt \\
=  & \int_0^T  \bigg[\int P_t \phi_t(X_t,a, \zeta)\,(\mu_t-\nu_t)(da)
 - \int V_t (\nabla_x f_t)(X_t,a, \zeta)\,\nu_t (da)\bigg]\,dt\,.
\end{split}
\]
From this and Lemma~\ref{lemma diff of J flat} and noting that $(\nabla_x g)(X_T, \zeta)V_T = P_T V_T$ we see that
\[
\begin{split}
 \frac{d}{d\varepsilon} & \bar J^0\left((\nu_t + \varepsilon(\mu_t-\nu_t)_{t\in[0,T]},\xi,\zeta\right)\bigg|_{\varepsilon=0} \\
= &  \int_0^T \left[\int f_t(X_t,a, \zeta)\,(\mu_t-\nu_t)(da) + \int (\nabla_x f_t)(X_t,a, \zeta) V_t \, \nu_t(da) \right]\,dt \\
& + \int_0^T  \bigg[\int P_t \phi_t(X_t,a, \zeta)\,(\mu_t-\nu_t)(da) - \int V_t (\nabla_x f_t)(X_t,a, \zeta)\,\nu_t (da)\bigg]\,dt  \\
 = & \int_0^T \int   h_t(X_t,P_t,a, \zeta) 
(\mu_t-\nu_t)(da)\,dt\,.
\end{split}
\]
We can conclude the proof by properties of the linear derivative.
\end{proof}

\begin{proof}[Proof of Theorem~\ref{thm necessary cond linear}.]
Let $(\mu_t)_{t\in[0,T]}$  be an arbitrary relaxed control.
Since $(\nu_t)_{t\in[0,T]}$ is optimal we know that 	
\[
J^\sigma\left(\nu_t + \varepsilon(\mu_t - \nu_t))_{t\in[0,T]}\right) \geq J^\sigma(\nu)\,\,\,\text{ for any $\varepsilon > 0$.}
\]
From this and Lemma~\ref{lemma der of J as hamiltonian flat} and Lemma~\ref{lemma diff of Ent flat} we get, after integrating over $(\xi, \zeta) \in \mathbb R^d \times \mathcal S$, that
\[
\begin{split}
0 & \leq \limsup_{\varepsilon \to 0} \frac1\varepsilon\left(J^\sigma(\nu_t + \varepsilon(\mu_t - \nu_t))_{t\in[0,T]} - J^\sigma(\nu)\right)\\
& \leq \int_0^T  \int \bigg( \int_{\mathbb R^d \times \mathcal S}  \frac{\delta H^0_t}{\delta m}(X^{\xi,\zeta}_t, P^{\xi,\zeta}_t, \nu_t, a, \zeta) \,\mathcal M(d\xi,d\zeta) + \frac{\sigma^2}{2}\left(\log \nu_t(a) - U(a)\right)\bigg)\, (\mu_t-\nu_t)(da) \,dt\,.	
\end{split}
\]
Now assume there is $S \in \mathcal B([0,T])$, with strictly positive  Lebesgue measure, such that
\[
\int_0^T \mathds{1}_S \int \bigg(\int_{\mathbb R^d \times \mathcal S} \frac{\delta H^0_t}{\delta m}(X^{\xi,\zeta}_t, P^{\xi,\zeta}_t, \nu_t,a,\zeta)\, \mathcal M(d\xi,d\zeta)+ \frac{\sigma^2}{2}\left(\log \nu_t(a) - U(a)\right)\bigg)\, (\mu_t-\nu_t)(da)\,dt < 0
\]
Define $\tilde \mu_t := \mu_t \mathds{1}_S + \nu_t \mathds{1}_{S^c}$.
Then by the same argument as above
\[
\begin{split}
0 & \leq  \int_0^T \int \bigg( \int_{\mathbb R^d \times \mathcal S}\frac{\delta H^0_t}{\delta m}(X^{\xi,\zeta}_t, P^{\xi,\zeta}_t, \nu_t,a,\zeta) \, \mathcal M(d\xi,d\zeta) + \frac{\sigma^2}{2}\left(\log \nu_t(a) - U(a)\right)\bigg) \, (\tilde\mu_t-\nu_t)(da)\,dt \\
&\, = \mathbb E\int_0^T \mathds{1}_S \bigg( \int_{\mathbb R^d \times \mathcal S}\int \frac{\delta H^0}{\delta m_t}(X^{\xi,\zeta}_t, P^{\xi,\zeta}_t, \nu_t,a,\zeta) \, \mathcal M(d\xi,d\zeta)+ \frac{\sigma^2}{2}\left(\log \nu_t(a) - U(a)\right)\bigg) \, (\mu_t-\nu_t)(da)\,dt\\
&\, < 0	
\end{split}
\]
leading to a contradiction. This proves i).

From i), properties of linear derivatives and Lemma~\ref{lemma diff of Ent flat} we have 
\[
\begin{split}
0 & \leq \int \bigg( \int_{\mathbb R^d \times \mathcal S}  \frac{\delta H^0_t}{\delta m}(X^{\xi,\zeta}_t, P^{\xi,\zeta}_t, \nu_t, a, \zeta) \,\mathcal M(d\xi,d\zeta) + \frac{\sigma^2}{2}\left(\log \nu_t(a) - U(a)\right)\bigg)\, (\mu_t-\nu_t)(da)\\	
& \leq \limsup_{\varepsilon \searrow 0} \int_{\mathbb R^d\times \mathcal S} \left[H^\sigma(X^{\xi,\zeta}_t, P^{\xi,\zeta}_t, \nu_t^\varepsilon, \zeta) - H^\sigma(X^{\xi,\zeta}_t, P^{\xi,\zeta}_t, \nu_t, \zeta) \right] \,\mathcal M(d\xi,d\zeta)\,.
\end{split}
\]
From this ii) follows.
\end{proof}

\begin{theorem}[Sufficient condition for optimality]
\label{thm sufficient condition}
Fix $\sigma \geq 0$.
Assume that $g$ and $h$ are continuously differentiable in the $x$ variable.
Assume that $(\nu_t)_{t\in [0,T]}$, $X^{\xi,\zeta}$, $P^{\xi,\zeta}$ are a solution to~\eqref{eq pontryagin system}.
Finally assume that
\begin{enumerate}[i)]
\item the map $x\mapsto g(x,\zeta)$ is convex for every $\zeta \in \mathcal S$ and
\item the map $(x,m) \mapsto H^\sigma_t(x, P^{\xi,\zeta}_t, m, \zeta)$ satisfies that for all $t\in [0,T]$, $\xi \in \mathbb R^d$, $\zeta \in \mathcal S$, $x,x'\in \mathbb R^d$ and all $m,m' \in \mathcal P_2(\mathbb R^p)$ (absolutely continuous w.r.t. the Lebesgue measure if $\sigma >0$) it holds that
\[
\begin{split}
& H^\sigma_t(x,P^{\xi,\zeta}_t,m,\zeta) - H^\sigma_t(x',P^{\xi,\zeta}_t,m',\zeta) \\
& \leq (\nabla_x H^0_t)(x,P^{\xi,\zeta}_t,m,\zeta)(x-x') + \int \bigg( \frac{\delta H^0_t}{\delta m}(x, P^{\xi,\zeta}_t, m, a, \zeta) + \frac{\sigma^2}2\log m(a)- \frac{\sigma^2}2U(a)\bigg)(m - m')(da)\,.	
\end{split}
\]  	
\end{enumerate}
Then the relaxed control $(\nu_t)_{t\in [0,T]}$ is an optimal control.
\end{theorem}

\begin{proof}[Proof of Theorem~\ref{thm sufficient condition}.]
Let $(\tilde \nu_t)_{t\in [0,T]}$ be another control with the associated family of
forward and backward processes $\tilde X^{\xi,\zeta}$, $\tilde P^{\xi,\zeta}$, $(\xi,\zeta) \in \mathbb R^d \times \mathcal S$.
Of course $X^{\xi,\zeta}_0 = \tilde X^{\xi, \zeta}_0$.
For now, we have $(\xi, \zeta)$ fixed and so write $X=X^{\xi,\zeta}$, $P=P^{\xi,\zeta}$.

First, we note that due to convexity of $x\mapsto g(x,\zeta)$ for every $\zeta \in \mathcal S$, we have
\[
\begin{split}
 & g(X_T,\zeta) - g(\tilde X_T,\zeta) \,
\leq  (\nabla_x g)(X_T,\zeta)(X_T - \tilde X_T) =  P_T (X_T - \tilde X_T) \\
  & = \int_0^T (X_t - \tilde X_t)\,dP_t + \int_0^T P_t(dX_t - d\tilde X_t) \\
  & = -\int_0^T (X_t - \tilde X_t)(\nabla_x H^0_t)(X_t,P_t,\nu_t,\zeta)\,dt
+ \int_0^T P_t\left(\Phi_t(X_t,\nu_t,\zeta) - \Phi_t(\tilde X_t,\tilde \nu_t,\zeta)\right)\,dt\,.
\end{split}
\]
Moreover, since $F(x,\nu,\zeta) + \frac{\sigma^2}{2}\text{Ent}(\nu) = H^\sigma_t(x,p,\nu,\zeta) - \Phi(x,\nu)\,p$  we have
\[
\begin{split}
 \int_0^T &  \left[F(X_t,\nu_t,\zeta) - F(\tilde X_t,\tilde \nu_t,\zeta) +\frac{\sigma^2}{2}\text{Ent}(\nu_t) - \frac{\sigma^2}{2}\text{Ent}(\tilde \nu_t) \right]\,dt 	\\
 = &  \int_0^T \Big[H^\sigma_t(X_t,P_t,\nu_t,\zeta) - \Phi_t(X_t,\nu_t,\zeta)Y_t
 - H^\sigma_t(\tilde X_t,P_t,\tilde \nu_t,\zeta) + \Phi_t(\tilde X_t,\tilde \nu_t,\zeta)P_t  \Big]\,dt \,. 	\\
\end{split}
\]
Hence
\[
\begin{split}
&\bar J^\sigma(\nu,\xi,\zeta) - \bar J^\sigma(\tilde \nu,\xi,\zeta) \\
&\leq -\int_0^T (X_t - \tilde X_t)(\nabla_x H^\sigma)(X_t,P_t,\nu_t,\zeta)\,dt +  \int_0^T \left[H^\sigma(X_t,P_t,\nu_t,\zeta) - H^\sigma(\tilde X_t,P_t,\tilde \nu_t,\zeta)\right]\,dt	\,.
\end{split}
\]	
We are assuming that $(x,\mu) \mapsto H(x,P_t,\mu,\zeta)$ is jointly convex in the sense of flat derivatives and so we have
\[
\begin{split}
& H^\sigma(X_t,P_t,\nu_t,\zeta) - H^\sigma(\tilde X_t,P_t,\tilde \nu_t,\zeta) \\
& \leq (\nabla_x H^\sigma_t)(X_t,P_t,\nu_t,\zeta)(X_t-\tilde X_t) + \int \bigg(\frac{\delta H^0_t}{\delta m}(X_t, P_t, \nu_t, a,\zeta) +\frac{\sigma^2}2\log \nu_t(a) - \frac{\sigma^2}2U(a)\bigg) (\nu_t - \tilde \nu_t)(da)\,.	
\end{split}
\]
Integrating over $(\xi,\zeta)$ w.r.t. $\mathcal M$ we thus have
\[
\begin{split}
& J^\sigma(\nu) - J^\sigma(\tilde \nu) \\
& \leq  \int_0^T \int \bigg(\int_{\mathbb R^d \times \mathcal S}  \frac{\delta H^0_t}{\delta m}(X^{\xi,\zeta}_t, P^{\xi,\zeta}_t, \nu_t, a,\zeta) \,\mathcal M(d\xi,d\zeta)+\frac{\sigma^2}2 \log\nu_t(a)-\frac{\sigma^2}2U(a)\Bigg)\,  (\nu_t - \tilde \nu_t)(da)	\,.
\end{split}
\]
Moreover $\nu_t = \argmin_{\mu} \int_{\mathbb R^d \times \mathcal S} H^\sigma_t(X^{\xi,\zeta}_t,P^{\xi,\zeta}_t,\mu,\zeta)\, \mathcal M(d\xi,d\zeta)$ implies that
\[
\int \bigg( \int_{\mathbb R^d \times \mathcal S}   \frac{\delta H^0_t}{\delta m}(X^{\xi,\zeta}_t, P^{\xi,\zeta}_t,\nu_t, a) \, \mathcal M(d\xi,d\zeta) + \frac{\sigma^2}{2}\log \nu_t(a) - \frac{\sigma^2}{2}U(a) \bigg)\, (\nu_t - \tilde \nu_t)(da) \leq 0\,.
\]
Hence $J(\nu) - J(\tilde \nu) \leq 0$ and $\nu$ is an optimal control.
We note that if either $g$ or $H^\sigma$ are strictly convex then $\nu$ is the optimal control.
A particular case is whenever $\sigma > 0$.
\end{proof}

\subsection{Existence, uniqueness, convergence to the invariant measure}
\label{sec exist uniq inv}

The reader will recall the space of relaxed controls $\mathcal V_2$ given by~\eqref{eq def Vq} and the integrated Wasserstein metric given by~\eqref{eq def wasserstein qT}.
The Mean-field Langevin system induces an $s$-time marginal law in the space $\mathcal V_2$ i.e. for each $s\geq 0$ we have $\mathcal{L}(\theta_{s,\cdot})  \in \mathcal V_2$
and moreover the map $s\mapsto \mathcal{L}(\theta_{s,\cdot})$ is continuous in the $\mathcal W^T_2$
metric on $\mathcal V_2$.
Hence, as in~\cite{hu2019meanode}, we have that
the map $I \ni s\mapsto \mathcal{L}(\theta_{s,\cdot}) \in \mathcal V_2$ belongs to the space
\[
\mathcal C(I,\mathcal V_2):=\left\{ \nu =(\nu_{s,\cdot})_{s\in I}: \nu_{s,\cdot} \in \mathcal V_2 \text{ and } \lim_{s'\rightarrow s}  \mathcal W^T_2(\nu_{s',\cdot},\nu_{s,\cdot})=0\, \,\,\, \forall s\in I\right\}\,.
\]
The existence of a unique solution to the system~\eqref{eq mfsgd}-\eqref{eq mfsgd 2} will be established by a fixed point argument in $\mathcal C(I,\mathcal V_2)$.

\begin{lemma}
\label{lemma odes}
Let Assumption~\ref{as coefficients} hold and
let $\mu \in \mathcal V_2$.	
\begin{enumerate}[i)]
\item The system
\begin{equation}
\label{eq odesystem}
\begin{cases}
	X^{\xi,\zeta}_{t}(\mu) &= \xi + \int_0^t  \Phi_r(X^{\xi,\zeta}_r(\mu),\mu_r,\zeta)\,dr
 \,,\,\,\, t\in [0,T]\,, \\
dP^{\xi,\zeta}_t(\mu) & = - (\nabla_x H_t)(X^{\xi,\zeta}_t(\mu),P^{\xi,\zeta}_t(\mu),\mu_t,\zeta)\,dt\,,\,\,\,\,
P^{\xi,\zeta}_T(\mu)=  (\nabla_x  g)(X^{\xi,\zeta}_T(\mu),\zeta)\,.
\end{cases}
\end{equation}
has a unique solution.
\item Assume further point iii) of Assumption~\ref{ass exist and uniq} holds.
Then
\[
\sup_{t\in[0,T]}\sup_{\mathcal M\in\mathcal P(\mathbb R^d\times\mathcal S)}\sup_{\mathcal \nu\in \mathcal V_2} \int_{\mathbb R^d\times\mathcal S} \big[|X_t^{\xi,\zeta}(\nu)|^2+|P_t^{\xi,\zeta}(\nu)|^2 \big]\,\mathcal M(d\xi,d\zeta) < \infty\,.
\]
\end{enumerate}
\end{lemma}
\begin{proof}
Assumption~\ref{as coefficients} implies that the functions $x \mapsto \Phi(x,\mu_t,\zeta)$
and $(x,p) \mapsto (\nabla_x H_t)(x,p,\mu_t,\zeta)$
are Lipschitz continuous and hence unique solution $(P^{\xi,\zeta}_t,X^{\xi,\zeta}_t)_{t\in[0,T]}$ exists for all $\xi,\zeta$.	 This proves point i). Point ii) follows from Lemma~\ref{lem:UniformBounds_Continuity} under our assumptions.
\end{proof}

Recall that~\eqref{eq fat h def} defines
\begin{equation*}
\mathbf h_t(a, \mu, \mathcal M) = \int_{\mathbb R^d \times \mathcal S} h_t(X^{\xi,\zeta}_t(\mu),P^{\xi,\zeta}_t(\mu),a,\zeta)\mathcal M(d\xi, d\zeta)\,,	
\end{equation*}
for $\mu \in \mathcal V_2$, $\mathcal M \in \mathcal P_2(\mathbb R^d\times\mathcal S)$,
where $(X^{\xi,\zeta}(\mu),P^{\xi,\zeta}(\mu))$ is the unique solution to~\eqref{eq odesystem} given by Lemma~\ref{lemma odes}.

\begin{lemma}
\label{lem hamilton lipschitz}
Let Assumptions \ref{as coefficients} and \ref{ass exist and uniq} hold. Then for any $\mathcal M \in \mathcal P_2(\mathbb R^d\times\mathcal S)$ there exists $L>0$ such that for all $a,a' \in \mathbb R^p$ and $\mu,\mu' \in \mathcal V_2$
\begin{equation} \label{con lipschitz}
| (\nabla_a \bold h_t)(a,\mu,\mathcal M) - (\nabla_a \bold h_t)(a',\mu',\mathcal M) | \leq L \left( |a-a'| + \mathcal W^T_1(\mu,\mu')\right)\,.
\end{equation}	
\end{lemma}
\begin{proof}
Due to Theorem~\ref{thm:RegularityHamiltonian} we know that there is $L'>0$ such that
\[
\begin{aligned}
|(\nabla_a \bH_t)(a,\mu,\DataM) & - (\nabla_a \bH_t)(a',\mu',\DataM) |\\
& \leq L' \left(1+\sup_{t\in[0,T]}\sup_{\mu \in\mathcal V_2}\int_{\DataS}|P^{\xi,\zeta}_t(\mu)|\,\Mm(d\xi,d\zeta)\right) \left( |a-a'| + \mathcal W^T_1(\mu,\mu')\right).
\end{aligned}
\]
This, together with point ii) of Lemma~\ref{lemma odes} provides the desired $L$.
\end{proof}

\begin{lemma}[Existence and uniqueness]\label{thm:WellposednessMeanField}
\label{lemma existence and uniqueness}
Let Assumptions \ref{as coefficients} and \ref{ass exist and uniq} hold.
Then there is a unique solution to~\eqref{eq mfsgd}-\eqref{eq mfsgd 2} for any $s\in I$.
Moreover
\begin{equation}\label{eq:UnifBound}
\begin{aligned}
&\int_{0}^{T}\mathbb E[|\theta_{s,t}|^2]\,dt
	\leq e^{(K-\frac{\sigma^2}{2}\kappa)s}\int_0^T \mathbb E [|\theta^0_{t}|^2]\,dt
+\int_{0}^{s}e^{(K-\frac{\sigma^2}{2}\kappa) (s-v)}(\sigma^{2}T
	 + \frac{\sigma^2T}{4\kappa } |(\nabla_x U)(0)|^2 +K)\,dv.
\end{aligned}
\end{equation}
where $K$ is a finite constant depending on $\phi$, $\Lagrang$, $\nabla_xg$, and $\DataM$.
\end{lemma}

\begin{proof}
Consider $\mu \in C(I,\mathcal V_2)$.
For each $\mu_s \in \mathcal V_2$, $s\geq 0$ we obtain unique solution to~\eqref{eq odesystem} which we denote $(X^{\xi,\zeta}_{s,\cdot}(\mu_{s,\cdot}), (P^{\xi,\zeta}_{s,\cdot}(\mu_{s,\cdot}))$.
Moreover, for each $t\in [0,T]$ the SDE
\[
d\theta_{s,t}(\mu) = - \bigg((\nabla_a \mathbf h_t)(\theta_{s,t}(\mu), \mu_{s,t},\mathcal M) + \frac{\sigma^2}{2}(\nabla_a U)(\theta_{s,t}(\mu))\bigg)\,ds + \sigma\,dB_s
\]
has unique strong solution and for each $s\in I$ we denote the measure in
$\mathcal V_2$ induced by
$\theta_{s,\cdot}$ as $\mathcal L(\theta(\mu_{s,\cdot}))$.
Consider now the map $\Psi$ given by $\mathcal C(I,\mathcal V_2) \ni \mu \mapsto \mathcal \{ \mathcal L(\theta(\mu)_{s,\cdot}) : s\in I\}$.

%


\paragraph{Step 1.} We need to show that $\{ \mathcal L(\theta(\mu_{s,\cdot})) : s\in I\} \in \mathcal C(I,\mathcal V_2)$.
This amounts to showing that we have the appropriate integrability and continuity i.e. that there is $K>0, \lambda > 0$  such that
\[
	\int_{0}^{T}\mathbb E[|\theta_{s,t}(\mu)|^2]\,dt
	\leq e^{(K-\sigma^2\kappa) s}\left( \int_0^T \mathbb E [|\theta^{0}_{t}|^2]\,dt  	+\int_{0}^{s}e^{\lambda v}(\sigma^{2}T
	 + \frac{\sigma^2T}{4\kappa } |(\nabla_x U)(0)|^2 +K)\,dv\right)\,.
\]
and  that
\[
\lim_{s'\rightarrow s} \int_0^T \mathbb E[|\theta_{s',t}(\mu) - \theta_{s,t}(\mu)|^2]dt = 0\,.
\]
To show the integrability observe that Assumption~\ref{ass exist and uniq} point ii) together with the Young's inequality: $\forall a,y\in \mathbb R^p$ $|a y|\leq \frac{\kappa}{2} |a|^2 + (2 \kappa)^{-1}|y|^2$ imply that
\begin{equation*}
 (\nabla_a U)(a)\,a \geq \frac{\kappa}{2} |a|^2  - \frac{1}{2\kappa}| (\nabla_a U)(0)|^2\,.	
\end{equation*}
Hence, for any $\lambda>0$, we have
\begin{align*}
&\int_{0}^{T}\mathbb E[e^{\lambda s}|\theta_{s,t}(\mu)|^2]\,dt
=\int_0^T\mathbb E [|\theta_{0,t}|^2]\,dt  + \lambda \int_{0}^{T}\int_{0}^{s} \mathbb E \left[e^{\lambda v}|\theta_{v,t}|^2\right]\,dv\,dt+\sigma^{2}T\int_{0}^{s}e^{\lambda v}\,dv\\
&\quad -2\int_{0}^{T}\int_{0}^{s}\mathbb E\left[e^{\lambda v}\left(\frac{\sigma^{2}}{2}\nabla_a U(\theta_{v,t})+(\nabla_a \bold h_t)(\theta_{v,t}(\mu),\mu_{v,\cdot},\mathcal M)\right) \theta_{v,t}(\mu)\right]\,dv\,dt.
\end{align*}
From the definition of the Hamiltonian and Lemma~\ref{lemma odes}, we have
\[
|(\nabla_a \bold h_t)(\theta_{v,t},\mu_{v,\cdot},\mathcal M)|\leq c\left(1+|\theta_{v,t}|+\int_{0}^{T}\mathbb E[|\theta_{v,t}|]\,dt\right)\,.
\]
Hence
\[
\int_0^T\mathbb E\left[\left|\left( \nabla_a \bold h_t)(\theta_{v,t}(\mu),\mu_{v,\cdot},\mathcal M)\right)\theta_{v,t}\right|\right]\,dt\leq c\left(1+\int_{0}^{T} \mathbb E [|\theta_{v,t}(\mu)|^2]\,dt\right)\,.
\]
Therefore,
\[
\begin{split}
	&\int_{0}^{T}\mathbb E[e^{\lambda s}|\theta_{s,t}(\mu)|^2]\,dt
	\leq \int_0^T \mathbb E [|\theta^0_t|^2]\,dt  + \left(\lambda-\frac{\sigma^2\kappa}{2} + c\right)\int_{0}^{T}\int_{0}^{s} \mathbb E \left[e^{\lambda s}|\theta_{v,t}(\mu)|^2\right]\,dv\,dt \\
	&+\int_{0}^{s}e^{\lambda v}\left(\sigma^{2}T
	 + \frac{\sigma^2T}{4\kappa } |(\nabla_x U)(0)|^2 + c\right)\,dv.
\end{split}
\]
Take $\lambda=\frac{\sigma^2\kappa}{2} - c$ to conclude the required integrability property.
To establish the continuity property note that for $s'\geq s$ we have
\[
|\theta_{s',t}(\mu) - \theta_{s,t}(\mu)|\leq L\int_s^{s'}\left(1+\int_0^T\int |a|\mu_{r,t}(da)\,dt\right)\,dr+\sigma|B_s-B_{s'}|\,.
\]
This, together with Lebesgue's theorem on dominated convergence
enables us to establish the continuity property.

\paragraph{Step 2.}
Take  $\lambda=\frac{\sigma^2}{2}\kappa - \frac{3}{2}L$.
Fix $t\in[0,T]$.
From It\^o's formula we get
\[
\begin{split}
	  d\Big(e^{2\lambda s} & |\theta_{s,t}(\mu) - \theta_{s,t}(\mu')|^2  \Big)
	=  \, 2 e^{2\lambda s}
	\biggl( \lambda  |\theta_{s,t}(\mu) - \theta_{s,t}(\mu')|	^2  \\
	& - \int  \frac{\sigma^2}{2} (\theta_{s,t}(\mu) - \theta_{s,t}(\mu')) (  (\nabla_a U)(\theta_{s,t}(\mu)) -  ((\nabla_a U)(\theta_{s,t}(\mu')) ) \mathcal M(d\xi,d\zeta)\\
	& 	-
  (\theta_{s,t}(\mu) - \theta_{s,t}(\mu'))
	( \nabla_a \bold h_t)(\theta_{s,t}(\mu),\mu_{s,\cdot},\mathcal M) - (\nabla_a \bold h_t)(\theta_{s,t}(\mu'),\mu'_{s,\cdot},\mathcal M))\biggr) ds \,.
\end{split}
\]
Assumption \ref{ass exist and uniq} point ii), along with Young's inequality:
 $\forall x,y$, $L>0$, $|xy|\leq L |x|^2 + (4 L)^{-1}|y|^2$  yields that
\begin{equation*}
\begin{split}
	  d(e^{2\lambda s} & |\theta_{s,t}(\mu)  - \theta_{s,t}(\mu')|^2  )\\
& \leq 2\bigg(\lambda - \frac{\sigma^2}{2}\kappa\bigg) e^{2\lambda s}|\theta_{s,t}(\mu) - \theta_{s,t}(\mu')|^2
 ds \\
& \qquad  + 2e^{2 \lambda s} \left( L|\theta_{s,t}(\mu) - \theta_{s,t}(\mu')|^2
+\frac{L}{2} \left( |\theta_{s,t}(\mu) - \theta_{s,t}(\mu')|^2  + \frac{L}{2} \mathcal W^T_2(\mu_{s,\cdot},\mu'_{s,\cdot})^2  \right) \right)ds \\
& \leq 2 e^{2\lambda s}\left[\bigg(\lambda - \Big(\frac{\sigma^2}{2}\kappa - \frac{3}{2}L\Big)\bigg) |\theta_{s,t}(\mu) - \theta_{s,t}(\mu')|^2 + \frac{L}{2} \mathcal W^T_2(\mu_{s,\cdot},\mu'_{s,\cdot})^2  \right] ds	\,.
\end{split}
\end{equation*}
Recall that for any $s\in I$ we have
\begin{equation}
\label{eq wass est}
\mathcal W^T_2(\mathcal L( \theta_{s,\cdot}(\mu)),\mathcal L( \theta_{s,\cdot}(\nu))^2\leq\int_0^T \mathbb E \left[ |\theta_{s,t}(\mu)-\theta_{s,t}(\nu)|^2 \right] dt\,.	
\end{equation}
Recall that for all $t\in[0,T]$, $\theta_{0,t}(\mu)=\theta_{0,t}(\nu)$.
Fix $S>0$ and note that
\begin{equation} \label{eq theta inequlity}
\begin{split}
&  e^{2\lambda S}\mathcal W^T_2(\mathcal L( \theta_{S,\cdot}(\mu)),\mathcal L( \theta_{S,\cdot}(\nu)))^2
\leq \frac{L T}{2}\int_0^S e^{2\lambda s} \mathcal W^T_2(\mu_{s,\cdot},\nu_{s,\cdot})^2ds\,.
\end{split}
\end{equation}
\paragraph{Step 3.} Let $\Psi^k$ denote the $k$-th composition of the mapping $\Psi$ with itself. Then, for any integer $k > 1$,
\[
\sup_{s\in[0,S]} \mathcal W^T_2(\Phi^k(\mu_{s,\cdot}),\Phi^k(\nu_{s,\cdot}))^2 \leq
 e^{-2\lambda S} \left(\frac{L T }{2} \right)^k \frac{S^k}{k!} \sup_{s\in[0,S]} \mathcal W^T_2(\mu_{s,\cdot},\nu_{s,\cdot})^2\,.
\]	
Hence, for any $S\in I$ there is $k$, such that $\Phi^k$ is a contraction and then Banach fixed point theorem gives existence of the unique solution on $[0,S]$. The estimate \eqref{eq:UnifBound} follows simply from  Step 1. The proof is complete.
\end{proof}

\begin{remark}
As a by-product of the proof of the above result we obtained the rate of convergence of the Picard iteration algorithm on  $\mathcal C(I,\mathcal V_2)$ for solving~\eqref{eq mfsgd}-\eqref{eq mfsgd 2}.
Such an algorithm can be combined with particle approximations in the spirit of~\cite {szpruch2017iterative}.
\end{remark}

\begin{proof}[Proof of Theorem~\ref{thm conv to inv meas rate}]
Lemma~\ref{lemma existence and uniqueness} provides the unique solution to~\eqref{eq mfsgd}-\eqref{eq mfsgd 2}.
The existence of the invariant measures follows by the similar argument as in \cite[Prop 2.4]{hu2019mean} and \cite[Prop 3.9]{hu2019meanode}.
We present the proof for completeness.
Since $\mathcal M$ is fixed here, we omit it from the notation in $\mathbf h_t$ defined in~\eqref{eq fat h def}.
Moreover, the required integrability and regularity for solutions to the Kolmogorov--Fokker--Planck equations are exactly those proved in~\cite[Proposition 5.2 and Lemmas 6.1-6.5]{hu2019mean} and we do not state them here.

\paragraph{Step 1.}
First we show that $\argmin_{\nu \in \mathcal V_2} J^{\sigma}(\nu) \neq \emptyset$. Denote
\[
\mathcal K :=\left\{ \nu \in \mathcal V: \,\,\,\frac{\sigma^2}{2} \text{Ent}(\nu) \leq  J^{\sigma}(\bar \nu) - \inf_{\nu} J^{\sigma}(\nu)\right\}\,.
\]
Note that since $\text{Ent}(\nu)$ is weakly lower-semicontinuous so is $J^{\sigma}$. Further, as sub-level set of relative entropy, $\mathcal K$ is weakly compact, \cite[Lem 1.4.3]{dupuis2011weak}. Further note that for $\nu \notin \mathcal K$
\[
J^{\sigma}(\nu) \geq \frac{\sigma^2}{2} \text{Ent}(\nu)  +\inf_{\nu} J^{\sigma}(\nu) > J^{\sigma}(\bar \nu)\,.
\]
Hence $\{\nu \in \mathcal V_2:\,\, J^{\sigma}(\nu) \leq J^{\sigma}(\bar \nu) \}\subset \mathcal K$, and so
$\inf_{\nu \in \mathcal V_2  }J^{\sigma}(\nu) = \inf_{\nu \in \mathcal K  }J^{\sigma}(\nu) $. A weakly lower continuous function achieves a global minimum on a weakly compact set and hence $\argmin_{\nu\in\mathcal V_2} J(^{\sigma}\nu) \neq \emptyset$.
This proves point i).

\paragraph{Step 2.}
Let $\nu^{\star} \in \argmin_{\nu \in \mathcal V_2}J(\nu)$. Necessarily $J(\nu^{\star})\leq J(\bar \nu) < \infty $. This together with the fact that $J$ is bounded from below means that $\int_0^T \text{Ent}(\nu_t)dt<\infty$ .
In particular,
$\nu^\star$ is absolutely continuous with respect to the Lebesgue measure.
Theorem \ref{thm necessary cond linear} i) tells us that for any
$\mu\in \mathcal V_2$,
 \[
\int \int_{\mathbb R^d \times \mathcal S}  \frac{\delta H^\sigma_t}{\delta m}(X^{\xi,\zeta}_t(\nu^\star), P^{\xi,\zeta}_t(
\nu^\star), \nu^\star_t,a,\zeta) \,\mathcal M(d\xi,d\zeta)\,(\mu_t-\nu^\star_t)(da)
\geq 0 \,\,\, \text{for a.a. $t\in (0,T)$}\,.
\]
From this and Lemma \ref{lem constant} we deduce that for a.a. $t\in [0,T]$ we have that
\[
a\mapsto \Gamma_t(a) := \int_{\mathbb R^d \times \mathcal S}  \frac{\delta H^\sigma}{\delta m}(X^{\star,\xi,\zeta}_t, P^{\star,\xi,\zeta}_t,  \nu^\star_t,a,\zeta) \,\mathcal M(d\xi,d\zeta)
\]
stays constant in $a$ and  $\Gamma_t(a)=\int \Gamma(a')\,\nu_t^{\star}(da'):=\Gamma_t$ for $\nu_t^{\star}$-a.a. $a\in \mathbb R^p.$
From the fact that $\nu^\ast$ is absolutely continuous w.r.t. Lebesgue measure and from~\eqref{eq hamiltonian} we see that
\[
a\mapsto \Gamma_t(a) = \mathbf h_t(a,\nu^\star) + \frac{\sigma^2}{2}\log(\nu^\star(a)) + \frac{\sigma^2}{2}U(a)\,
\]
is constant in $a$, possibly $-\infty$.
We know $\nu^\star(a) \geq 0$ and on $S:=\{(t,a)\in \mathbb R^p \times [0,T]  : \nu_t^\star(a) > 0\}$
the probability measure $\nu_t^\star$ satisfies the equation
\[
\nu^{\star}_t(a) = e^{-\frac{2}{\sigma^2}\Gamma_t}e^{-U(a) -\frac{2}{\sigma^2} \mathbf h_t(a,\nu^\star)}\,.
\]
Since $\nu^\star$ is a probability measure
\[
1 = \int \nu^\star(a)\,da = e^{-\frac{2}{\sigma^2}\Gamma_t}
\int e^{-U(a) -\frac{2}{\sigma^2} \mathbf h_t(a,\nu^\star)}\,da\,.
\]
Due Lemma~\ref{lemma odes}, part ii) we see that $a\mapsto h_t(a,\nu^\star)$ has at most linear growth.
This and Assumption~\ref{ass exist and uniq} point ii) implies that $0 < \int e^{-U(a) -\frac{2}{\sigma^2} \mathbf h_t(a,\nu^\star)}\,da < \infty$.
This implies that for a.a. $t\in [0,T]$ we have $\Gamma_t \neq -\infty$.
Hence $\nu_t^\star(a) > 0$ for Lebesgue a.a $a\in (0,T)\times \mathbb R^p$ and $\nu^\star$ is equivalent to the Lebesgue measure on $[0,T]\times \mathbb R^p$.
\paragraph{Step 3.}
Let $\mathcal L(\theta^0_t) := \nu^\star$ and consider $\theta_{s,t}$ given by~\eqref{eq mfsgd}-\eqref{eq mfsgd 2}.
Since $\nu^\star \in \argmin_{\nu\in\mathcal V_2} J(\nu)$ we know that
\[
0 \leq \lim_{s\searrow 0}\frac{J(\mathcal L(\theta_{s,\cdot})) - J(\nu^\star)}{s} = \frac{d}{ds}J(\mathcal L(\theta_{s,\cdot}))\bigg|_{s=0}\,.
\]
Moreover, from Theorem~\ref{corollary with measure flow} we have that
\[
\frac{d}{ds}J(\mathcal L(\theta_{s,\cdot}))\bigg|_{s=0} = - \int_0^T \bigg|(\nabla_a \mathbf h_t)(a,\nu^\star) +\frac{\sigma^2}{2}\frac{\nabla_a \nu_t^\star}{\nu_t^\star}(a)+\frac{\sigma^2}{2}\nabla_a U(a)\bigg|^2\nu^\star_t(a)\,da\,dt \leq 0\,.
\]
Hence for Lebesgue $a.a.$ $(t,a)\in (0,T)\times \mathbb R^p$ we have
\[
0 = (\nabla_a \mathbf h_t)(a,\nu^\star) +\frac{\sigma^2}{2}\frac{\nabla_a \nu_t^\star}{\nu_t^\star}(a)+\frac{\sigma^2}{2}\nabla_a U(a)\,.
\]
Multiplying by $\nu^\star(a)>0$ and applying the divergence operator we see that for a.a. $t\in [0,T]$ the function $a\mapsto \nu_t^\star(a)$ solves the stationary Fokker--Planck equation
\[
\nabla \cdot \bigg[\bigg((\nabla_a \mathbf h_t)(\cdot,\nu^\star_t) +\frac{\sigma^2}{2}(\nabla_a U)\bigg)\nu^\star_t   \bigg] + \frac{\sigma^2}{2}\Delta_a \nu^\star_t  = 0\,\,\,\text{on}\,\,\, \mathbb R^p\,.
\]
From this we can conclude that $\nu^\star_t$ is an invariant measure for~\eqref{eq mfsgd}-\eqref{eq mfsgd 2}.
This concludes the proof of point ii).

\paragraph{Step 4.}
Take two solutions to~\eqref{eq mfsgd}-\eqref{eq mfsgd 2} denoted by $(\theta,X,P)$ and $(\theta',X',P,)$ with initial distributions  $\mathcal L( (\theta_{0,t})_{t\in[0,T]})$ and $\mathcal L( (\theta'_{0,t})_{t\in[0,T]})$, respectively.
Take $\lambda = \sigma^2\kappa - 4L > 0$.
Due to~\eqref{eq wass est}, Lemma~\ref{lem hamilton lipschitz} and Assumption~\ref{ass exist and uniq} point ii) we have
\begin{equation*}
\begin{split}
&  d\bigg(e^{\lambda s} \mathcal W^T_2 (\mathcal L(\theta_{s,\cdot}),  \mathcal L(\theta_{s,\cdot}))^2 \bigg) \leq d\bigg(e^{\lambda s} \mathbb E \int_{0}^T|\theta_{s,t}-\theta'_{s,t}|^2dt\bigg)  =  \, e^{\lambda s} \mathbb E \int_{0}^T
	\Biggl[ \lambda |\theta_{s,t} - \theta'_{s,t}|^2 \\
& - 2(\theta_{s,t} - \theta'_{s,t}) \biggl(
	(\nabla_a \bold h_t)(\theta_{s,t}, \mathcal L(\theta_{s,\cdot}))
 -(\nabla_a \bold h_t)(\theta'_{s,t},\mathcal L(\theta'_{s,\cdot}))
 + \frac{\sigma^2}{2} \Big( (\nabla_a U)(\theta_{s,t})
  - (\nabla_a U)(\theta_{s,t}') \Big) \biggr) \,dt \,ds \\
& \leq  \,  e^{\lambda s}
	\Bigl( -L \int_{0}^T\mathbb E|\theta_{s,t} - \theta'_{s,t}|^2 dt
	 + L \mathcal W^T_2 (\mathcal L(\theta_{s,\cdot}),  \mathcal L(\theta'_{s,\cdot}))^2 \Bigr) ds \leq 0\,,
\end{split}
\end{equation*}
where we used~\eqref{eq wass est} again to obtain the last inequality.
Integrating this leads to
\[
\mathcal W^T_2(\mathcal L( \theta_{s,\cdot}),\mathcal L(\theta'_{s,\cdot}))^2 \leq e^{-\lambda s}   {\cal W}^T_2(\mathcal L( \theta_{0,\cdot}),\mathcal L(\theta'_{0,\cdot}))^2\,.
\]
Take $\theta'$ to be solution to~\eqref{eq mfsgd}-\eqref{eq mfsgd 2} with initial condition
$\mathcal L(\theta'_{0,\cdot})=\nu^\star_{\cdot}$
and $\theta$ to be solution to \eqref{eq mfsgd}-\eqref{eq mfsgd 2} with an arbitrary initial condition $(\theta^0_t)_{t\in [0,T]}$.
We see that we have the convergence claimed in~\eqref{eq inv meas conv rate}.
Moreover
if $\mu^\star$ is another invariant measure and we can start~\eqref{eq mfsgd}-\eqref{eq mfsgd 2} with $\mathcal L(\theta_{0,\cdot})=\nu^\star_{\cdot}$ and $\mathcal L({\theta'}_{0,\cdot})=\mu^\star_{\cdot}$
to see that for any $s\geq 0$ we have $\mathcal W^T_2(\mu^\star,\nu^\star)^2 \leq e^{-\lambda s}   {\cal W}^T_2(\mu^\star,\nu^\star)^2$.
This is a contradiction unless $\mu^\star = \nu^\star$.
This proves point iii).
\end{proof}

\subsection{Particle approximation and propagation of chaos}
\label{section chaos and euler}
In this section, we study particle system that corresponds to \eqref{eq mfsgd}-\eqref{eq mfsgd 2}.
Define $\DataM^{N_1}=\frac{1}{N_1}\sum_{j=1}^{N_1}\delta_{{(\xi^{j},\zeta^{j})}}$.
For convenience, we introduce ${\bH}^{\sigma}_t$ defined on $\er^{\ControlS}\times\Vv_2\times\Pp(\er^d)$ by
\[
\nabla_a{\bH}^{\sigma}_{t}(a,{{\nu}},\DataM)={\frac{\sigma^2}{2}}\nabla_aU(a)+\nabla_a{\bH}_{t}(a,{{\nu}},\DataM)\,.
\]

 \begin{lemma}\label{lem:UnifParticleMoment} Let Assumptions \ref{as coefficients} and \ref{ass exist and uniq} hold. Then, for any finite time horizon $0<S<\infty$, there exists a unique solution $(\theta^1_{s,t},...,\theta^{N_2}_{s,t})$ to the $\mathbb R^{p N_2}$ dimensional system of SDEs \eqref{eq:ParticleApprox-bis}.
 Moreover, for all $0\leq s\leq S$, $\int_0^T\EE[|\theta^{i}_{s,t}|^2]\,dt<\infty$.
\end{lemma}

\begin{proof}
Under Assumptions~\ref{as coefficients} and~\ref{ass exist and uniq} the coefficients of $\mathbb R^{p N_2}$-dimensional system of SDEs \eqref{eq:ParticleApprox-bis} are globally Lipschitz continuous and therefore existence and uniqueness and integrability results can be derived by adapted classical techniques for stochastic differential equations, see e.g.~\cite[Chapter 5]{karatzas2012brownian}.
\end{proof}

The rate of convergence between \eqref{eq mfsgd} and \eqref{eq:ParticleApprox-bis} is given by the following theorem.

\begin{theorem}\label{prop:PropagationChaosBis}
Let Assumptions~\ref{as coefficients} and~\ref{ass exist and uniq} hold.
Fix $\lambda = \frac{\sigma^2 \kappa}{2 } -\frac{L}{2}(3 + T) + \frac{1}{2} $. Define $(\theta^{i,\infty})_{i=1}^{N_2}$ consisting of $N_2$ independent copies of \eqref{eq mfsgd}, given by
\begin{equation}\label{eq:McKeanVlasovCopies}
\begin{aligned}
&\theta^{i,\infty}_{s,t}=\theta^{i}_{0,t} - \int_0^s\nabla_a {\bH}^{\sigma}_{t}(\theta^{i,\infty}_{v,t},\Ll(\theta^{i,\infty}_{v,.}),\DataM)\,dv+\sigma \ControlBrownian^i_s,\,0\leq s\leq S,\,0\leq t\leq T.
%
%
\end{aligned}
\end{equation}
Then there exists $c$, independent of $S,N_1,N_2,d,p$, such that, for all $i=1,\ldots,N_2$ we have
\begin{align*}
&\int_0^T\EE\left[\left|\theta^{1}_{s,t}-\theta^{i,\infty}_{s,t}\right|^2\right]\,dt \leq \frac{c}{\lambda}(1-e^{- \lambda s})  \left(\frac{1}{N_1}+\frac{1}{N_2}\right)\,.
\end{align*}
\end{theorem}
\begin{proof}
\textbf{Step 1.}
Noticing that the particle system \eqref{eq:ParticleApprox-bis} is exchangeable, it is sufficient to prove our claim for $i=1$.
By uniqueness of solutions to~\eqref{eq mfsgd}, $\Ll(\theta_{v,.})=\Ll(\theta^{i,\infty}_{v,.})$.
We also define
$
{{\nu}}^{N_2,\infty}_{s,t}=\frac{1}{N_2}\sum_{j=1}^{N_2}\delta_{\theta^{j,\infty}_{s,t}}
$.
Furthermore, for any $i$,
\[
\theta^{i}_{s,t}-\theta^{i,\infty}_{s,t} = -\left( \int_0^s\nabla_a {\bH}^{\sigma}_{t}(\theta^{i}_{v,t},{{\nu}}^{N_2}_{v,.},\DataM^{N_1})\,dv -\int_0^s\nabla_a {\bH}^{\sigma}_{t}(\theta^{i,\infty}_{v,t},\Ll(\theta_{v,.}),\DataM)\,dv
\right)\,.
\]
For $\lambda>0$, to be chosen later on, we have
\begin{equation}
\label{PropaChaos:proofst1}
\begin{aligned}
&\int_0^T\EE\left[e^{\lambda s}\left|\theta^{1}_{s,t}-\theta^{1,\infty}_{s,t}\right|^2\right]\,dt
=\lambda \int_0^T\int_0^s\EE\left[e^{\lambda v}\left|\theta^{1}_{v,t}-\theta^{1,\infty}_{v,t}\right|^2\right]\,dv\,dt\\
&\quad - 2 \left( \int_0^T\int_0^s e^{\lambda v}\EE\left[\left(\nabla_a {\bH}^{\sigma}_{t}(\theta^{1}_{v,t},{{\nu}}^{N_2}_{v,\cdot},\DataM^{N_1}) - \nabla_a {\bH}^{\sigma}_{t}(\theta^{1,\infty}_{v,t},\Ll(\theta_{v,\cdot}),\DataM)
\right)\left(\theta^{1}_{v,t}-\theta^{1,\infty}_{v,t}\right)\right]\,dv
dt \right)\,.
\end{aligned}
\end{equation}
Observe that
\begin{align*}
\EE & \left[\left(\nabla_a {\bH}^{\sigma}_{t}(\theta^{1,\infty}_{v,t},\Ll(\theta_{v,\cdot}),\DataM)
-\nabla_a {\bH}^{\sigma}_{t}(\theta^{1}_{v,t},{{\nu}}^{N_2}_{v,\cdot},\DataM^{N_1})\right)\left(\theta^{1}_{v,t}-\theta^{1,\infty}_{v,t}\right)\right]\\
= &\EE\left[\frac{\sigma^2}{2}\left( U(\theta^{1}_{v,t}) - U(\theta^{1,\infty}_{v,t})  \right)\left(\theta^{1}_{v,t}-\theta^{1,\infty}_{v,t}\right)\right]\\
&+\EE\left[\left(\nabla_a {\bH}_{t}(\theta^{1,\infty}_{v,t},\Ll(\theta_{v,\cdot}),\DataM)
-\nabla_a {\bH}_{t}(\theta^{1,\infty}_{v,t},{{\nu}}^{N_2,\infty}_{v,\cdot},\DataM^{N_1})\right)
\left(\theta^{1}_{v,t}-\theta^{1,\infty}_{v,t}\right)\right]\\
&+\EE\left[\left(\nabla_a {\bH}_{t}(\theta^{1,\infty}_{v,t},{{\nu}}^{N_2,\infty}_{v,\cdot},\DataM^{N_1})
-\nabla_a {\bH}_{t}(\theta^{1}_{v,t},{{\nu}}^{N_2}_{v,\cdot},\DataM^{N_1})\right)\left(\theta^{1}_{v,t}-\theta^{1,\infty}_{v,t}\right)\right]\,.
\end{align*}
By Lemma \eqref{lem hamilton lipschitz}
\begin{align*}
&\left|\nabla_a {\bH}_{t}(\theta^{1,\infty}_{v,t},{{\nu}}^{N_2,\infty}_{v,\cdot},\DataM^{N_1})
-\nabla_a {\bH}_{t}(\theta^{1}_{v,t},{{\nu}}^{N_2}_{v,\cdot},\DataM^{N_1})\right|
%
&\leq L\left(|\theta^{1}_{v,t}-\theta^{1,\infty}_{v,t}|+\frac{1}{N_2}\sum_{j=1}^{N_2}
\mathbb E \left[ \int_0^T|\theta^{j}_{v,r}-\theta^{j,\infty}_{v,r}|\,dr \right]\right)\,.
\end{align*}
Therefore, by Young's inequality and the exchangeability of $(\theta^{j}_{s,\cdot})_{j=1}^{N_2}$,
\begin{align*}
&\EE\left[\left(\nabla_a {\bH}_{t}(\theta^{1,\infty}_{v,t},\nu^{N_2,\infty}_{v,\cdot},\DataM^{N_1})
-\nabla_a {\bH}_{t}(\theta^{1}_{v,t},\nu^{N_2}_{v,\cdot},\DataM^{N_1})\right)\left(\theta^{1}_{v,t}-\theta^{1,\infty}_{v,t}\right)\right]\\
&\leq L(1 + \frac{T}{2})  \EE\left[|\theta^{1}_{v,t}-\theta^{1,\infty}_{v,t}|^2\right]+\frac{L}{2}\int_0^T\EE\left[|\theta^{1}_{v,r}-\theta^{1,\infty}_{v,r}|^2\right]\,dt\,.
\end{align*}
Using Young's inequality again, we have
\begin{align*}
&\EE\left[\left(\nabla_a {\bH}_{t}(\theta^{1,\infty}_{v,t},\Ll(\theta_{v,\cdot}),\DataM)
-\nabla_a {\bH}_{t}(\theta^{1,\infty}_{v,t},\nu^{N_2,\infty}_{v,\cdot},\DataM^{N_1})\right)
\left(\theta^{1}_{v,t}-\theta^{1,\infty}_{v,t}\right)\right]\\
&\leq \frac{1}{2}\EE\left[\left|\nabla_a {\bH}_{t}(\theta^{1,\infty}_{v,t},\Ll(\theta_{v,\cdot}),\DataM)
-\nabla_a {\bH}_{t}(\theta^{1,\infty}_{v,t},\nu^{N_2,\infty}_{v,\cdot},\DataM^{N_1})\right|^2\right]
+\frac{1}{2}\EE\left[\left|\theta^{1}_{v,t}-\theta^{1,\infty}_{v,t}\right|^2\right]
\end{align*}
Coming back to \eqref{PropaChaos:proofst1} and employing Assumption~\ref{ass exist and uniq}-ii), we get
\begin{equation}\label{PropaChaos:proofst2}
\begin{aligned}
&\int_0^T\EE\left[e^{\lambda s}\left|\theta^{1}_{s,t}-\theta^{1,\infty}_{s,t}\right|^2\right]\,dt\leq\int_0^T\int_0^s e^{\lambda v}\EE\left[2(\lambda - (\frac{\sigma^2 \kappa}{2 } -\frac{L}{2}(3 + T) + \frac{1}{2} )\left|\theta^{1}_{v,t}-\theta^{1,\infty}_{v,t}\right|^2)\right]\,dv\,dt\\
&\quad+2\,\int_0^T\int_0^s e^{\lambda v}
\EE\left[\left|\nabla_a {\bH}_{t}(\theta^{1,\infty}_{v,t},{{\nu}}_{v,\cdot},\DataM)
-\nabla_a {\bH}_{t}(\theta^{1,\infty}_{v,t},\nu^{N_2,\infty}_{v,\cdot}, \DataM^{N_1})\right|^2\right]\,dv\,dt\\
\end{aligned}
\end{equation}
\paragraph{Step 2.}
The statistical errors will be derived from
\[
\begin{split}
I &:=\int_{0}^{T}\EE\left[\left|\nabla_a {\bH}_{t}(\theta^{1,\infty}_{v,t},{{\nu}}_{v,\cdot},\DataM)
-\nabla_a {\bH}_{t}(\theta^{1,\infty}_{v,t},\nu^{N_2,\infty}_{v,\cdot}, \DataM^{N_1})\right|^2\right]\,dt \\
&\leq 2 \int_{0}^{T}\EE\left[\left|\nabla_a {\bH}_{t}(\theta^{1,\infty}_{v,t},\nu^{N_2,\infty}_{v,\cdot},\DataM)-\nabla_a {\bH}_{t}(\theta^{1,\infty}_{v,t},\nu^{N_2,\infty}_{v,\cdot},\DataM^{N_1})\right|^2\right] \,dt\\
&\quad +2 \int_{0}^{T} \EE\left[\left|\nabla_a {\bH}_{t}(\theta^{1,\infty}_{v,t},{{\nu}}_{v,\cdot},\DataM)-\nabla_a {\bH}_{t}(\theta^{1,\infty}_{v,t},\nu^{N_2,\infty}_{v,\cdot},\DataM) \right|^2\right]\,dt:=  I_1  +  I_2 \,.
\end{split}
\]
Recall that
\[
\nabla_a {\bH}_{t}(\theta^{i,\infty}_{v,t},\nu^{N_2}_{v,\cdot},\mathcal M^{N_1})=
\frac{1}{N_1}\sum_{j=1}^{N_1}\nabla_a
h_{t}(X^{\xi^{j},\zeta^{j}}_{t}(\nu^{N_2}_{v,\cdot}),\theta^i_{v,t},P^{\xi^{j},\zeta^{j}}_{t}(\nu^{N_2}_{v,\cdot}),{\zeta}^{j}) \,.
\]
Note that $\mathbb E \left[ \nabla_a {\bH}_{t}(\theta^{1,\infty}_{v,t},{{\nu}}^{N_2,\infty}_{v,\cdot},\DataM^{N_1})\middle| \theta^{1,\infty}_{v,t},\nu^{N_2,\infty}_{v,\cdot}   \right] = \nabla_a {\bH}_{t}(\theta^{1,\infty}_{v,t},
{{\nu}}^{N_2,\infty}_{v,\cdot},\mathcal M^{N_1}) $.
Hence, from Assumption \ref{as coefficients} and Lemmas \ref{thm:WellposednessMeanField} and \ref{lem:UniformBounds_Continuity}
\begin{align*}
& I_1 = \frac{2}{N_1}
\int_0^T \mathbb E\left[ \mathbb Var\left[
 \nabla_a {\bH}_{t}(\theta^{1,\infty}_{v,t},\nu^{N_2,\infty}_{v,\cdot},\DataM^{N_1}) \middle| \theta^{1,\infty}_{v,t},\nu^{N_2,\infty}_{v,\cdot}   \right]\right]\,dt\,dv\\
&\leq \frac{2}{N_1}\int_{0}^{T}\EE\left[\Vert\Lagrang_{t}(x,
\cdot,\zeta^1)\Vert_{Lip}^2+\Vert\phi_{t}(x,\cdot,\zeta^1)\Vert_{Lip}^2
|P^{\xi^1,\zeta^1}_t(\nu^{N,\infty}_{s,\cdot})|^2\right]\,dt
\leq \frac{c}{N_1}\,.
\end{align*}
Next, we aim to show that there exists constant $c$ independent of $S,d,p$, such that
\begin{equation}\label{PropaChaos:proofst9}
I_2 := 2 \, \EE\left[\left|\nabla_a \bH_{t}(\theta^{1,\infty}_{v,t},{{\nu}}_{v,\cdot},\DataM)-\nabla_a \bH_{t}(\theta^{1,\infty}_{v,t},\nu^{N_2,\infty}_{v,\cdot},\DataM)\right|^2\right]\leq \frac{c}{N_2}\,.
\end{equation}
By Assumption \ref{as coefficients}, we have
\begin{align*}
&\left|\nabla_a \bH_{t}(\theta^{1,\infty}_{v,t},\nu_{v,\cdot},\DataM)-\nabla_a \bH_{t}(\theta^{1,\infty}_{v,t},\nu^{N_2,\infty}_{v,\cdot},\DataM)\right|\\
&\leq \int_{\mathbb R^d\times \mathcal S}\left(\Vert \nabla_a f_t(\cdot,\theta^{1,\infty}_{v,t},\zeta) \Vert_{Lip} +
\Vert \nabla_a \phi_t(\cdot,\theta^{1,\infty}_{v,t},\zeta) \Vert_{Lip} | P^{\xi,\zeta}_t(\nu_{v,\cdot})|\right)|X^{\xi,\zeta}_t(\nu_{v,\cdot})-X^{\xi,\zeta}_t(\nu^{N_2,\infty}_{v,\cdot})|
\DataM(d\xi,d\zeta)\\
&\quad + \lvert  \nabla_a \phi_t \rvert_{\infty} \int_{\mathbb R^d\times \mathcal S} |P^{\xi,\zeta}_t(\nu_{v,\cdot})-P^{\xi,\zeta}_t(\nu^{N_2,\infty}_{v,\cdot})| \DataM(d\xi,d\zeta) \\
&\leq c \left(\int_{\mathbb R^d\times \mathcal S}\left(1+  | P^{\xi,\zeta}_t(\nu_{v,\cdot})|\right)^2 \DataM(d\xi,d\zeta) \right)^{1/2}
\left(\int_{\mathbb R^d\times \mathcal S}  |X^{\xi,\zeta}_t(\nu_{v,\cdot})-X^{\xi,\zeta}_t(\nu^{N_2,\infty}_{v,\cdot})|^2 \DataM(d\xi,d\zeta) \right)^{1/2}\\
&\quad + c \left(\int_{\mathbb R^d\times \mathcal S} |P^{\xi,\zeta}_t(\nu_{v,\cdot})-P^{\xi,\zeta}_t(\nu^{N_2,\infty}_{v,\cdot})|^2 \DataM(d\xi,d\zeta) \right)^{1/2}
\end{align*}
Lemma~\ref{lem difference P and X}, tells us that that there exists (an explicit) constant $c$ such that
\[
\left|X^{\xi,\zeta}_t(\nu_{v,\cdot})-X^{\xi,\zeta}_t(\nu^{N_2,\infty}_{v,\cdot})\right| \leq
c \int_0^T\left|\frac{1}{N_2}\sum_{j=1}^{N_2}\phi_r(X^{\xi,\zeta}_r(\nu_{v,\cdot}),\theta^{j,\infty}_{v,r},\zeta)-
\int\phi_r(X^{\xi,\zeta}_r(\nu_{v,\cdot}),a,\zeta)\,\nu_{v,r}(da)\right|\,dr\,
\]
and
\[
\left|P^{\xi,\zeta}_t(\nu_{v,\cdot})-P^{\xi,\zeta}_t(\nu^{N_2,\infty}_{v,\cdot})\right|\leq
c\int_{0}^{T}\left|\frac{1}{N_2}\sum_{j=1}^{N_2}\phi_r(X^{\xi,\zeta}_r(\nu_{v,\cdot}),\theta^{j,\infty}_{v,r},\zeta)-
\int\phi_r(X^{\xi,\zeta}_r(\nu_{v,\cdot}),a,\zeta)\,\nu_{v,r}(da)\right|\,dr.
\]
Consequently, and due to Lemma \ref{lem:UniformBounds_Continuity}
\begin{align*}
&\EE\left[\left|\nabla_a \bH_{t}(\theta^{1,\infty}_{v,t},\nu_{v,\cdot},\DataM)-\nabla_a \bH_{t}(\theta^{1,\infty}_{v,t},\nu^{N_2,\infty}_{v,\cdot},\DataM)\right|^2\right]\\
&\leq c \int_0^T \int_{\mathbb R^d\times\mathcal S}
\EE\left[\left|\frac{1}{N_2}\sum_{j_2=1}^{N_2}\phi_r(X^{\xi,\zeta}_r(\nu_{v,\cdot}),\theta^{j_2,\infty}_{v,r},\zeta)-
\int\phi_r(X^{\xi,\zeta}_r(\nu_{v,\cdot}),a,\zeta)\,\nu_{v,r}(da)\right|^2\right]
\DataM(d\xi,d\zeta)\,dt.
\end{align*}
Note that $\mathbb E \left[ \phi_r(X^{\xi,\zeta}_r(\nu_{v,\cdot}),\theta^{j_2,\infty}_{v,r},\zeta)\middle| X^{\xi,\zeta}_r(\nu_{v,\cdot}),\zeta  \right] = \int\phi_r(X^{\xi,\zeta}_r(\nu_{v,\cdot}),a,\zeta)\,\nu_{v,r}(da) $.
Hence
\begin{equation}
\label{PropaChaos:proofst9-step}
\begin{split}
&\EE\left[\left|\nabla_a \bH_{t}(\theta^{1,\infty}_{v,t},\nu_{v,\cdot},\DataM)-\nabla_a \bH_{t}(\theta^{1,\infty}_{v,t},\nu^{N_2,\infty}_{v,\cdot},\DataM)\right|^2\right]\\
&\leq \frac{c}{N_2} \int_0^T \int_{\mathbb R^d\times\mathcal S}
\EE\left[   \mathbb Var \left[  \phi_r(X^{\xi,\zeta}_r(\nu_{v,\cdot}),\theta^{1,\infty}_{v,r},\zeta)    \middle| X^{\xi,\zeta}_r(\nu_{v,\cdot}),\zeta  \right] \right]
\DataM(d\xi,d\zeta)\,dt.
\end{split}
\end{equation}
By Assumption \ref{as coefficients} and Lemmas \ref{lemma odes} and \ref{lemma existence and uniqueness} we know that
\[
\begin{split}
\int_{\mathbb R^d\times\mathcal S}\!\! \EE\left[   \mathbb Var \left[  \phi_r(X^{\xi,\zeta}_r(\nu_{v,\cdot}),\theta^{1,\infty}_{v,r},\zeta)    \middle| X^{\xi,\zeta}_r(\nu_{v,\cdot}),\zeta  \right] \right]
\DataM(d\xi,d\zeta)\\
\leq c\int_{\mathbb R^d\times\mathcal S}(1 +\mathbb E |X^{\xi,\zeta}_r(\nu_{v,\cdot})|^2  +
  \mathbb E | \theta^{1,\infty}_{v,r}|^2 )\DataM(d\xi,d\zeta) <\infty\,.	
\end{split}
\]
From this and~\eqref{PropaChaos:proofst9-step} we conclude that the estimate \eqref{PropaChaos:proofst9} holds.

\paragraph{Step 3.} Coming back to \eqref{PropaChaos:proofst1}, with $\lambda = \frac{\sigma^2 \kappa}{2 } -\frac{L}{2}(3 + T) + \frac{1}{2} $
we get that
\begin{align*}
&\int_0^T\EE\left[e^{\lambda s}\left|\theta^{1}_{s,t}-\theta^{1,\infty}_{s,t}\right|^2\right]\,dt \leq c \left(\frac{1}{N_1}+\frac{1}{N_2}\right) \int_0^se^{\lambda v}\,dv\,.
\end{align*}
Calculating the integral on the right hand side concludes the proof.
\end{proof}


\subsection{Time discretisation of the gradient descent}\label{sec time disc}
\subsubsection{Proof of Theorem \ref{thm:EulerRate1}}
First, we consider only discretisation of the overdamped Langevin dynamics~\eqref{eq:ParticleApprox-bis}.
A simple, explicit numerical scheme for the numerical approximation of \eqref{eq:ParticleApprox-bis} can be introduced through Euler--Maruyama approximations with non-homogeneous time steps.
Fix an increasing sequence of times $0=s_0 < s_1 < \cdots<s_l<\cdots$ and set
\[
{\Lambda}(s)=\sup\{s_l\,:\,s_l\leq s\}\,.
\]
Now define the family of processes
 $(\widetilde{\theta}^{1}_{s,t})_{0\leq s\leq S, 0\leq t\leq T},...,(\widetilde{\theta}^{N_2}_{s,t})_{0\leq s\leq S, 0\leq t\leq T}$ satisfying, for any $i$,
\begin{equation}\label{TimeDiscrete:proofst1}
\widetilde{\theta}^{i}_{s,t}=\theta^{0,i}_{t}-\int_{0}^{s}\nabla_a{\bH}^\sigma_{t}
\left(\widetilde{\theta}^{i}_{{\Lambda}(v),t},\widetilde{{\nu}}^{N_2}_{{\Lambda}(v),\cdot},\DataM^{N_1}\right)\,dv
+\sigma\ControlBrownian^i_s,\,\,\,\widetilde{{\nu}}^{N_2}_{v,t}=\frac{1}{N_2}\sum_{j_2=1}^{N_2}\delta_{\{\widetilde{\theta}^{j_2}_{v,t}\}}.
\end{equation}
For this approximation, the rate of convergence is given by the following lemma.

\begin{lemma}\label{prop:EulerRate1} Let Assumptions~\ref{as coefficients} and~\ref{ass exist and uniq} hold. Assume also that $\{s_l\}$ is a non-decreasing sequence times, starting from $0$, such that $\{s_l-s_{l-1}\}$ is non-increasing, $\sum_{l}(s_l-s_{l-1})^2<\infty$ and, for $\kappa$ large enough, \eqref{TimeStepRestrict} holds.
Then, for all $i$, $1\leq l\leq n$,
\begin{align*}
%
\mathbb E\left [\int_0^T |\theta^i_{s_l,t} - \widetilde \theta^i_{s_l,t}|^2\,dt \right ]
&\leq c\max_{l'}(s_{l'}-s_{l'-1})\left(1+\max_{0\leq s\leq s_l}\int_{0}^{T}\EE\left[\left|\theta^i_{s,t}\right|^2\right]\,dt\right).
\end{align*}
\end{lemma}
\begin{proof}  As the components $(\xi^i,\zeta^i)$, will be fixed throughout the proof, for the sake of simplicity, $\DataM^{N_1}$ will be omitted from now on in the notation of ${\bH}^\sigma$, and will be re-introduced only when needed.

\paragraph{Step 1.} As, for all $s_{l-1}\leq s\leq s_l$, $0\leq t\leq T$, the difference $\triangle_{s,t} \theta^i:=\theta^i_{s,t}-\widetilde{\theta}^i_{s,t}$ is given by
\begin{equation}\label{EulerRate1:proofst2}
\triangle_{s,t}\theta^i=\triangle_{s_{l-1},t}\theta^i+\int_{s_{l-1}}^{s}
\left(\nabla_a\bH^{\sigma}_{t}\left(\theta^{i}_{v,t},\overline{{\nu}}^{N_2}_{v,\cdot}\right)
-\nabla_a\bH^{\sigma}_{t}\left(\widetilde{\theta}^{i}_{s_{l-1},t},\widetilde{{{\nu}}}^{N_2}_{s_{l-1},\cdot}\right)\right)\,dv.
\end{equation}

In particular, for $s=s_l$, and by integration by parts,
\begin{equation}\label{EulerRate1:proofst3}
\begin{aligned}
&\left|\triangle_{s_l,t}\theta^{i}\right|^2=\left|\triangle_{s_{l-1},t}\theta^{i}\right|^2\\
&\quad -2\int_{s_{l-1}}^{s_l} \left(\triangle_{v,t}\theta^i\right)\cdot\left(
\nabla_a{\bH}^\sigma_{t}\left(\theta^{i}_{v,t},\overline{{\nu}}^{N_2}_{v,\cdot}\right)
-\nabla_a{\bH}^\sigma_{t}\left(\widetilde{\theta}^{i}_{s_{l-1},t},\widetilde{{{\nu}}}^{N_2}_{s_{l-1},\cdot}\right)\right)\,dv\\
&=\left|\triangle_{s_{l-1},t}\theta\right|^2-2\int_{s_{l-1}}^{s_l}\left(\triangle_{v,t}\theta^i\right)\cdot\left(\nabla_a {\bH}^\sigma_{t}
\left(\theta^{i}_{v,t},\overline{{{\nu}}}^{N_2}_{v,\cdot}\right)
-\nabla_a{\bH}^\sigma_{t}\left(\theta^{i}_{s_{l-1},t},\overline{{\nu}}^{N_2}_{s_{l-1},\cdot}\right)\right)\,dv\\
&\quad -2\int_{s_{l-1}}^{s_l}\left(\triangle_{v,t}\theta^i\right)\cdot\left(\nabla_a {\bH}^\sigma_{t}
\left(\theta^{i}_{s_{l-1},t},\overline{{{\nu}}}^{N_2}_{s_{l-1},\cdot}\right)
-\nabla_a{\bH}^\sigma_{t}\left(\widetilde{\theta}^{i}_{s_{l-1},t},\widetilde{{{\nu}}}^{N_2}_{s_{l-1},\cdot}\right)\right)\,dv.
\end{aligned}
\end{equation}
Using again \eqref{EulerRate1:proofst2} and, by the Young inequality: $2a\cdot b\leq |a|^2+|b|^2$, observe then that
\begin{align*}
&-2\int_{s_{l-1}}^{s_l} \left(\triangle_{v,t}\theta^i\right)\cdot\left(\nabla_a {\bH}^\sigma_{t}
\left(\theta^{i}_{s_{l-1},t},\overline{{{\nu}}}^{N_2}_{s_{l-1},\cdot}\right)
-\nabla_a{\bH}^\sigma_{t}\left(\widetilde{\theta}^{i}_{s_{l-1},t},\widetilde{{{\nu}}}^{N_2}_{s_{l-1},\cdot}\right)\right)\,dv\\
&=-2(s_l-s_{l-1})\left(\triangle_{s_{l-1},t}\theta^i\right)\cdot\left(\nabla_a {\bH}^\sigma_{t}
\left(\theta^{i}_{s_{l-1},t},\overline{{{\nu}}}^{N_2}_{s_{l-1},\cdot}\right)
-\nabla_a{\bH}^\sigma_{t}\left(\widetilde{\theta}^{i}_{s_{l-1},t},\widetilde{{{\nu}}}^{N_2}_{s_{l-1},\cdot}\right)\right)\\
&\quad +(s_l-s_{l-1})^2\left|\nabla_a {\bH}^\sigma_{t}
\left(\theta^{i}_{s_{l-1},t},\overline{{{\nu}}}^{N_2}_{s_{l-1},\cdot}\right)
-\nabla_a{\bH}^\sigma_{t}\left(\widetilde{\theta}^{i}_{s_{l-1},t},\widetilde{{{\nu}}}^{N_2}_{s_{l-1},\cdot}\right)\right|^2\\
&\leq (1-\sigma^2\kappa)(s_l-s_{l-1})\left|\triangle_{s_{l-1},t}\theta^i\right|^2+
(s_l-s_{l-1})\left|\nabla_a {\bH}_{t}
\left(\theta^{i}_{s_{l-1},t},\overline{{{\nu}}}^{N_2}_{s_{l-1},\cdot}\right)
-\nabla_a{\bH}_{t}\left(\widetilde{\theta}^{i}_{s_{l-1},t},\widetilde{{{\nu}}}^{N_2}_{s_{l-1},\cdot}\right)\right|^2\\
&\quad +(s_l-s_{l-1})^2\left|\nabla_a {\bH}^\sigma_{t}
\left(\theta^{i}_{s_{l-1},t},\overline{{{\nu}}}^{N_2}_{s_{l-1},\cdot}\right)
-\nabla_a{\bH}^\sigma_{t}\left(\widetilde{\theta}^{i}_{s_{l-1},t},\widetilde{{{\nu}}}^{N_2}_{s_{l-1},\cdot}\right)\right|^2.
\end{align*}
In the same way, we have
\begin{align*}
&-2\int_{s_{l-1}}^{s_l}\left(\triangle_{v,t}\theta^i\right)\cdot\left(\nabla_a {\bH}^\sigma_{t}
\left(\theta^{i}_{v,t},\overline{{{\nu}}}^{N_2}_{v,\cdot}\right)
-\nabla_a{\bH}^\sigma_{t}\left(\theta^{i}_{s_{l-1},t},\overline{{{\nu}}}^{N_2}_{s_{l-1},\cdot}\right)\right)\,dv\\
%
&\leq (s_l-s_{l-1})\left|\triangle_{v,t}\theta^i\right|^2 +(s_l-s_{l-1})^2\left|\nabla_a {\bH}^\sigma_{t}
\left(\theta^{i}_{s_{l-1},t},\overline{{{\nu}}}^{N_2}_{s_{l-1},\cdot}\right)
-\nabla_a{\bH}^\sigma_{t}\left(\widetilde{\theta}^{i}_{s_{l-1},t},\widetilde{{{\nu}}}^{N_2}_{s_{l-1},\cdot}\right)\right|^2\\
&+\int_{s_{l-1}}^{s_l}(1+(v-s_{l-1}))\left|\nabla_a {\bH}^\sigma_{t}
\left(\theta^{i}_{v,t},\overline{{{\nu}}}^{N_2}_{v,\cdot}\right)
-\nabla_a{\bH}^\sigma_{t}\left(\theta^{i}_{s_{l-1},t},\overline{{{\nu}}}^{N_2}_{s_{l-1},\cdot}\right)\right|^2\,dv.
\end{align*}
Coming back to \eqref{EulerRate1:proofst3}, we obtain
\begin{equation}\label{EulerRate1:proofst4}
\begin{aligned}
&\left|\triangle_{s_l,t}\theta^{i}\right|^2\leq\left|\triangle_{s_{l-1},t}\theta^{i}\right|^2\left(1-(\sigma^2\kappa-3)(s_{l}-s_{l-1})\right)\\
&\quad +(s_l-s_{l-1})\left|\nabla_a {\bH}_{t}
\left(\theta^{i}_{s_{l-1},t},\overline{{{\nu}}}^{N_2}_{s_{l-1},\cdot}\right)
-\nabla_a{\bH}_{t}\left(\widetilde{\theta}^{i}_{s_{l-1},t},\widetilde{{{\nu}}}^{N_2}_{s_{l-1},\cdot}\right)\right|^2\\
&\quad +2(s_l-s_{l-1})^2\left|\nabla_a {\bH}^\sigma_{t}
\left(\theta^{i}_{s_{l-1},t},\overline{{{\nu}}}^{N_2}_{s_{l-1},\cdot}\right)
-\nabla_a{\bH}^\sigma_{t}\left(\widetilde{\theta}^{i}_{s_{l-1},t},\widetilde{{{\nu}}}^{N_2}_{s_{l-1},\cdot}\right)\right|^2\\
&\quad  +\int_{s_{l-1}}^{s_l}(1+(v-s_{l-1}))\left|\nabla_a {\bH}^\sigma_{t}
\left(\theta^{i}_{v,t},\overline{{{\nu}}}^{N_2}_{v,\cdot}\right)
-\nabla_a{\bH}^\sigma_{t}\left(\theta^{i}_{s_{l-1},t},\overline{{{\nu}}}^{N_2}_{s_{l-1},\cdot}\right)\right|^2\,dv.
\end{aligned}
\end{equation}
Taking the expectation of the above and integrating the resulting expression over $[0,T]$, we get
\begin{equation}\label{EulerRate1:proofst5}
\begin{aligned}
&\int_0^T\EE[\left|\triangle_{s_l,t}\theta^{i}\right|^2]\,dt\leq\left(1-(1+\sigma^2\kappa)(s_{l}-s_{l-1})\right)
\int_0^T\EE[|\triangle_{s_{l-1},t}\theta^{i}|^2]\,dt\\
&\quad +(s_l-s_{l-1})\int_0^T\EE\left[\left|\nabla_a {\bH}_{t}
\left(\theta^{i}_{s_{l-1},t},\overline{{{\nu}}}^{N_2}_{s_{l-1},\cdot}\right)
-\nabla_a{\bH}_{t}\left(\widetilde{\theta}^{i}_{s_{l-1},t},\widetilde{{{\nu}}}^{N_2}_{s_{l-1},\cdot}\right)\right|^2\right]\,dt\\
&\quad +2(s_l-s_{l-1})^2\int_0^T\EE\left[\left|\nabla_a {\bH}^\sigma_{t}
\left(\theta^{i}_{s_{l-1},t},\overline{{{\nu}}}^{N_2}_{s_{l-1},\cdot}\right)
-\nabla_a{\bH}^\sigma_{t}\left(\widetilde{\theta}^{i}_{s_{l-1},t},\widetilde{{{\nu}}}^{N_2}_{s_{l-1},\cdot}\right)\right|^2\right]\,dt\\
&\quad  +\int_{s_{l-1}}^{s_l}(1+(v-s_{l-1}))\int_0^T\EE\left[\left|\nabla_a {\bH}^\sigma_{t}
\left(\theta^{i}_{v,t},\overline{{{\nu}}}^{N_2}_{v,\cdot}\right)
-\nabla_a{\bH}^\sigma_{t}\left(\theta^{i}_{s_{l-1},t},\overline{{{\nu}}}^{N_2}_{s_{l-1},\cdot}\right)\right|^2\right]\,dt\,dv.
\end{aligned}
\end{equation}
Lemma~\ref{lem hamilton lipschitz} then yields
\[
\int_0^T\EE\left[\left| \nabla_a {\bH}_{t}(\theta^i_{s_{l-1},t},\overline{{{\nu}}}^{N_2}_{s_{l-1},\cdot})-
\nabla_a {\bH}_{t}(\widetilde{\theta}^i_{s_{l-1},t},\widetilde{{{\nu}}}^{N_2}_{s_{l-1},\cdot})\right|^2\right]\,dt
\leq L\int_0^T\EE\left[|\triangle_{s_{l-1},t}\theta^i|^2\right]\,dt,
\]
and
\[
\int_0^T\EE\left[\left| \nabla_a {\bH}^{\sigma}_{t}(\theta^i_{s_{l-1},t},\overline{{{\nu}}}^{N_2}_{s_{l-1},\cdot})-
\nabla_a {\bH}^{\sigma}_{t}(\widetilde{\theta}^i_{s_{l-1},t},\widetilde{{{\nu}}}^{N_2}_{s_{l-1},\cdot})\right|^2\right]\,dt\leq L(1+\frac{\sigma^4}{2}\Vert\nabla_aU\Vert^2_{Lip})\int_0^T\EE\left[|\triangle_{s_{l-1},t}\theta^i|^2\right]\,dt.
\]
Plugged into \eqref{EulerRate1:proofst5}, using the exchangeability of $\theta^i$ and $\widetilde{\theta}^i$, we get
\begin{align*}
&\int_0^T\EE\left[\left|\triangle_{s_l,t}\theta^i\right|^2\right]\,dt\\
&\leq  \left(1+(s_l-s_{l-1})\left\{\left(L-\sigma^2\kappa\right)
+2(s_l-s_{l-1})L\left(1+\frac{\sigma^4}{2}\Vert\nabla^2_aU\Vert^2_{\infty}\right)\right\}\right)
\times\int_0^T\EE\left[\left|\triangle_{s_{l-1},t}\theta^i\right|^2\right]\,dt\\
&\quad +\int_{s_{l-1}}^{s_l} (1+(v-s_{l-1}))
\int_0^T\EE\left[\left|\nabla_a{\bH}^\sigma_{t}\left(\theta^{i}_{v,t},\overline{{{\nu}}}^{N_2}_{v,\cdot}\right)
-\nabla_a{\bH}^\sigma_{t}\left(\theta^{i}_{s_{l-1},t},\overline{{{\nu}}}^{N_2}_{s_{l-1},\cdot}\right)\right|^2\right]\,dt\,dv.
\end{align*}
Recalling the recurrence estimate:
\begin{equation}\label{EulerRate1:proofst6}
u_{l+1}\leq c_{l+1}u_{l}+b_{l+1},\,\forall l \:\:\: \Rightarrow \:\:\:
 u_l\leq \sum_{l_1=1}^l\left(\Pi_{l_2=l_1}^l c_{l_2}\right)b_{l-1}+u_0\left(\Pi_{l_2=1}^l c_{l_2}\right),\,\forall l,
\end{equation}
we get, since $\triangle_{0,t}\theta=0$,
\begin{align*}
&\int_0^T\EE\left[\left|\triangle_{s_l,t}\theta\right|^2\right]\,dt\\
&\leq \sum_{l_1=1}^{l}\left(\Pi_{l_2=l_1}^l\left(1+(s_{l_1}-s_{l_1-1})\left(L-\sigma^2\kappa
\right)+(s_{l_1}-s_{l_1-1})^2L\left(1+\frac{\sigma^4}{2}\Vert\nabla_aU\Vert^2_{Lip}\right)\right)\right)\\
&\times\int_{s_{l_1-1}}^{s_{l_1}}(1+(v-s_{l-1}))
\int_0^T\EE\left[\left|\nabla_a{\bH}^\sigma_{t}\left(\theta^{i}_{v,t},\overline{{{\nu}}}^{N_2}_{v,\cdot}\right)
-\nabla_a{\bH}^\sigma_{t}\left(\theta^{i}_{s_{l_1-1},t},\overline{{{\nu}}}^{N_2}_{s_{l_1-1},\cdot}\right)\right|^2\right]\,dt\,dv\\
\end{align*}
\paragraph{Step 2.} Observing that
\begin{align*}
\EE\left[\left|\nabla_a{\bH}^\sigma_{t}\left(\theta^{i}_{v,t},\overline{{{\nu}}}^{N_2}_{v,\cdot}\right)
-\nabla_a{\bH}^\sigma_{t}\left(\theta^{i}_{s_{l-1},t},\overline{{{\nu}}}^{N_2}_{s_{l-1},\cdot}\right)\right|^2\right]
\leq L\left(1+\frac{\sigma^4}{2}\Vert\nabla_aU\Vert^2_{Lip}\right)
\int_0^T\EE\left[\left|\theta^{i}_{v,t}-\theta^{i}_{s_{l-1},t}\right|^2\right]\,dt,
\end{align*}
and since, for all $s_{l-1}\leq v\leq s_{l}$,
\[
\EE\left[\left|\theta^{i}_{v,t}-\theta^{i}_{s_{l-1},t}\right|^2\right]\leq L(s_l-s_{l-1})\left(1+2\sigma^2+\frac{\sigma^4}{2}(s_l-s_{l-1})\Vert \nabla_aU\Vert^2_{Lip}\int_{s_{l-1}}^{s_{l}}\EE\left[\left|\theta^i_{v,t}\right|^2\right]\,dv\right),
\]
we have
\begin{align*}
&\int_{s_{l_1-1}}^{s_{l_1}}(1+(v-s_{l-1}))
\int_0^T\EE\left[\left|\nabla_a{\bH}^\sigma_{t}\left(\theta^{i}_{v,t},\overline{{{\nu}}}^{N_2}_{v,\cdot}\right)
-\nabla_a{\bH}^\sigma_{t}\left(\theta^{i}_{s_{l_1-1},t},\overline{{{\nu}}}^{N_2}_{s_{l_1-1},\cdot}\right)\right|^2\right]\,dt\,dv\\
&\leq c(s_{l_1}-s_{l_1-1})^2\left(1+\max_{0\leq s\leq s_l}\int_{0}^{T}\EE\left[\left|\theta^i_{s,t}\right|^2\right]\,dt\right),
\end{align*}
from which we get
\begin{equation}
\label{EulerRate1:proofst7}
\begin{aligned}
&\int_0^T\EE\left[\left|\triangle_{s_l,t}\theta\right|^2\right]\,dt
\leq c\left(1+\max_{0\leq s\leq s_l}\int_{0}^{T}\EE\left[\left|\theta^i_{s,t}\right|^2\right]\,dt\right)\\
&\quad\quad\quad\times
\sum_{l_1=1}^{l}\left(\Pi_{l_2=l_1}^l\left(1+(s_{l_1}-s_{l_1-1})\Big(\left(L-\sigma^2\kappa
\right)+(s_{l_1}-s_{l_1-1})L\left(1+\frac{\sigma^4}{2}\Vert\nabla_aU\Vert^2_{Lip}\right)\Big)\right)\right)(s_{l_1}-s_{l_1-1})^2.
\end{aligned}
\end{equation}
By the assumption \eqref{TimeStepRestrict}, any time step $s_l-s_{l-1}$ is small enough so that the coefficient
\[
\left(L-\sigma^2\kappa
\right)+(s_{l_1}-s_{l_1-1})L\left(1+\frac{\sigma^4}{2}\Vert\nabla_aU\Vert^2_{Lip}\right)=:\overline{\kappa},
\]
is negative. It now remains to prove that the sum
\begin{align*}
&\sum_{l_1=1}^{l}\left(\Pi_{l_2=l_1}^l\left(1+(s_{l_1}-s_{l_1-1})\Big(\left(L-\sigma^2\kappa
\right)+(s_{l_1}-s_{l_1-1})L\left(1+\frac{\sigma^4}{2}\Vert\nabla_aU\Vert^2_{Lip}\right)\right)\Big)\right)(s_{l_1}-s_{l_1-1})\\
&=\sum_{l_1=1}^{l}\left(\Pi_{l_2=l_1}^l\left(1-|\overline{\kappa}|(s_{l_2}-s_{l_2-1})\right)\right)(s_{l_1}-s_{l_1-1})
\end{align*}
is finite. Since $1-x\leq e^{-x}$ for all $x\geq 0$, we have
\[
\sum_{l_1=1}^{l}\left(\Pi_{l_2=l_1}^l\left(1-|\overline{\kappa}|(s_{l_2}-s_{l_2-1})\right)\right)(s_{l_1}-s_{l_1-1})
\leq \sum_{l_1=1}^{l}\exp\{-|\overline{\kappa}|(s_{l}-s_{l_1})\}(s_{l_1 }-s_{l_1-1}).
\]
Comparing this upper-bound with the integral $\int_0^{s_l}\exp\{-|\overline{\kappa}|(s_l-v)\}\,dv$, the assumption $\sum_{l}(s_l-s_{l-1})^2<\infty$ is enough to ensures the finiteness of the sum as:
\begin{align*}
&\sum_{l_1=1}^{l} \exp\{-|\overline{\kappa}|(s_{l}-s_{l_1})\}(s_{l_1}-s_{l_1-1})-\int_0^{s_l}\exp\{-|\overline{\kappa}|(s_l-v)\}\,dv\\
&=\sum_{l_1=1}^{l-1}\int_{s_{l_1}}^{s_{l_1+1}}\Big(\exp\{-|\overline{\kappa}|(s_{l}-s_{l_1})\}(s_{l_1}-s_{l_1-1})-\exp\{-|\overline{\kappa}|(s_l-v)\}\Big)\,dv
-\int_0^{s_{1}}\exp\{-|\overline{\kappa}|(s_l-v)\}\,dv\\
&\leq \exp\{-|\overline{\kappa}|s_{l}\} \sum_{l_1=1}^{l-1}\int_{s_{l_1}}^{s_{l_1+1}}\Big(v-s_{l_1-1}\Big)\exp\{|\overline{\kappa}|v\}\,dv\leq \sum_{l_1=1}^{l-1}(s_{l_1}-s_{l_1-1})^2<\infty.
\end{align*}
This ends the proof.
\end{proof}

\subsection{Generalisation estimates}
\label{sec generalisation estimates}

\begin{lemma} \label{lem reg J}
Let Assumptions \ref{as coefficients} and \ref{ass exist and uniq} hold.
Then there exist constants $L_{1,\mathcal M}$ and $L_{2,\mathcal M}$, such that for any $\mu,\nu\in\mathcal V_2$ we have
\begin{enumerate}[i)]
\item	For all $\mathcal M\in\mathcal P(\DataS)$,
\[
|J^{\mathcal M}({\mu})-  J^{\mathcal M}({{\nu}})| \leq
 L_{1,\mathcal M} W_1^T (\mu,\nu)\,,
\]
\item For any stochastic processes $\eta$, $\eta'$
such that $\mathbb E\int_{0}^{T}[|\eta_t|^2 +|\eta'_t|^2\,dt]< \infty$ we have
\begin{align*}
	\mathbb E\left[ \sup_{\nu \in \mathcal V_2} \left| \int_0^T\frac{\delta J^{\mathcal M}}{\delta \nu}(\nu,t,\eta_t)\,dt\right|^2\right]
	+
	\mathbb E\left[ \sup_{\nu \in \mathcal V_2} \left|\int_{0}^T\int_{0}^T\frac{\delta^2 J^{\mathcal M}}{\delta \nu^2}(\nu,t,\eta_t,t',\eta_t')\,dtdt'\right|^2\right] \leq L_{2,\mathcal M}\,.
	\end{align*}
\end{enumerate}
\end{lemma}
The expression $\frac{\delta J^{\mathcal M}}{\delta \nu}(\nu,t,a)$ here is to the derivative of $\nu\in\mathcal V_2\mapsto J^{\mathcal M}(\nu)$ (see Definition \ref{def:ExtendedDerivative}) and $\frac{\delta^2 J^{\mathcal M}}{\delta \nu^2}(\nu,t,a,t',a')$ is the derivative of $\nu\in\mathcal V_2\mapsto \frac{\delta J^{\mathcal M}}{\delta \nu}(\nu,t,a)$ (Definition \ref{def:ExtendedSecondOrdDeriv}).
\begin{proof}

Let ${{\nu}}^\lambda:= \nu + \lambda(\mu - \nu)$ for $\lambda \in [0,1]$ and recall definition of $\bar J$ from \eqref{eq objective bar J}.
From Lemma \ref{lemma der of J as hamiltonian flat} (with $\sigma=0$) we have
\begin{equation*}
\frac{d}{d\varepsilon} \bar J^{\mathcal M}\left((\nu + (\lambda +\varepsilon)(\mu-\nu),\xi,\zeta\right)\bigg|_{\varepsilon=0}
=  \int_0^T \int h_t(X^{\xi,\zeta}_t(\nu^\lambda),P^{\xi,\zeta}_t(\nu^\lambda),a,\zeta) (\mu_t-\nu_t)(da)  \,dt \,.	
\end{equation*}	
Due to this and the fundamental theorem of calculus we have
\begin{equation}\label{eq J first flat derivative}
\begin{split}
\bar	J^{\mathcal M}({\mu},\xi,\zeta)- \bar J^{\mathcal M}(\nu,\xi,\zeta)
	&=\int_0^1 \frac{d}{d\varepsilon} \bar  J^{\mathcal M} \left(\nu + (\lambda +\varepsilon)(\mu-\nu),\xi,\zeta\right)\bigg|_{\varepsilon=0} \,d\lambda \\
& =\int_0^1  \int_0^T \int h_t(X^{\xi,\zeta}_t(\nu^\lambda),P^{\xi,\zeta}_t(\nu^\lambda),a,\zeta) (\mu_t-\nu_t)(da)  \,dt	
\,d\lambda\,.
\end{split}
\end{equation}
Assumptions~\ref{as coefficients}, \ref{ass exist and uniq} point iii) and Lemma \ref{lem:UniformBounds_Continuity} allow us to conclude that
\[
a \mapsto h_t(X^{\xi,\zeta}_t(\nu^\lambda),P^{\xi,\zeta}_t(\nu^\lambda),a,\zeta)\,  = \phi_t(X^{\xi,\zeta}_t(\nu^\lambda),a,\zeta) \,  P^{\xi,\zeta}_t(\nu^\lambda) +
  f(X^{\xi,\zeta}_t(\nu^\lambda),a,\zeta)\,,
\]
is uniformly Lipschitz in $a$.
From Fubini's Theorem and Kantorovich (dual) representation of the Wasserstein distance \cite[Th 5.10]{villani2008optimal} we conclude that
\begin{equation}\label{eq derivative of cost function}
	|\bar J^{\mathcal M}({\mu},\xi,\zeta)- \bar J^{\mathcal M}({{\nu}},\xi,\zeta)|
	\leq  L_1(\xi,\zeta) \, \mathcal W_1^T (\mu,\nu) \,,\,\,
L_1(\xi,\zeta) :=\sup_{t\in[0,T], \nu\in\mathcal V_2}\left\Vert h_t(X^{\xi,\zeta}_t(\nu),P^{\xi,\zeta}_t(\nu),\cdot,\zeta)  \right\Vert_{Lip}\,,
\end{equation}
and we see, by Lemma \ref{lem:UniformBounds_Continuity}, that
\[
| J^{\mathcal M}({\mu})-  J^{\mathcal M}({{\nu}})| \leq
\int_{\mathbb R^d \times \mathcal S} |\bar J({\mu},\xi,\zeta)- \bar J({{\nu}},\xi,\zeta)|\mathcal M(d\xi,d\zeta)
\leq \int_{\mathbb R^d \times \mathcal S}  L_1(\xi,\zeta)\mathcal M(d\xi,d\zeta) \mathcal W_1^T (\mu,\nu)\,.
\]
Define $L_{1,\mathcal M}:=\int_{\mathbb R^d \times \mathcal S}  L_1(\xi,\zeta)\mathcal M(d\xi,d\zeta) < \infty$.
This completes the proof of part i).

From \eqref{eq J first flat derivative}, we are able to identify the derivative of $\nu\in\Vv_2\mapsto \bar  J(\nu,\xi,\zeta)$ and see that
\[
\frac{\delta \bar  J}{\delta\nu}(\mu,t,a,\xi,\zeta)= \phi_t(X^{\xi,\zeta}_t(\mu),a,\zeta) \,  P^{\xi,\zeta}_t(\mu) +
  f_t(X^{\xi,\zeta}_t(\mu),a,\zeta)\,.
\]
Hence due to Definition~\ref{def:ExtendedDerivative} and a following similar computation as in Lemma \ref{lem:FirstDerivative} we have
\[
\begin{split}
\frac{\delta^2 \bar J}{\delta\nu^2}(\mu,t,a,t',a',\xi,\zeta)  = &
(\nabla_x\phi_t)(X^{\xi,\zeta}_t(\mu),a,\zeta) \frac{\delta X_t^{\xi,\zeta}}{\delta\nu}(\mu,t',a')   P^{\xi,\zeta}_t(\mu)
+ \phi_t(X^{\xi,\zeta}_t(\mu),a,\zeta)  \frac{\delta P_t^{\xi,\zeta}}{\delta\nu}(\mu,t',a')\\
 & +
  (\nabla_x f_t)(X^{\xi,\zeta}_t(\mu),a,\zeta)\frac{\delta X_t^{\xi,\zeta}}
 {\delta\nu}(\mu,t',a')\,.
  \end{split}
\]
Note that
\[
\begin{split}
	\frac{\delta   J^{\mathcal M}}{\delta\nu}(\mu,t,a)=& \int_{\mathbb R^d \times \mathcal S} \frac{\delta \bar  J}{\delta\nu}(\mu,t,a,\xi,\zeta) \, \mathcal M(d\xi,d\zeta)\,,\\
\frac{\delta^2  J^{\mathcal M}}{\delta\nu^2}(\mu,t,a,t',a')=& \int_{\mathbb R^d \times \mathcal S} \frac{\delta^2 \bar J}{\delta\nu^2}(\mu,t,a,t',a',\xi,\zeta)\,\mathcal M(d\xi,d\zeta)\,.
\end{split}
\]
From Lemma \ref{lem:FirstDerivative} we see that for $\eta$ and $\eta'$
such that $\int_{0}^{T}\mathbb E[|\eta_t|^2 +|\eta'_t|^2]dt < \infty$ we have
\begin{align*}
	\mathbb E\left[ \sup_{\nu \in \mathcal V_2} \left| \int_0^T\frac{\delta J^{\mathcal M}}{\delta \nu}(\nu,t,\eta_t)dt\right|^2\right]
	+
	\mathbb E\left[ \sup_{\nu \in \mathcal V_2} \left|\int_{0}^T\int_{0}^T\frac{\delta^2 J^{\mathcal M}}{\delta \nu^2}(\nu,t,\eta_t,t',\eta'_{t'})dtdt'\right|^2\right] \leq L_{2,\mathcal M}\,.
	\end{align*}
\end{proof}

\begin{lemma}\label{lem strong error}
We assume that the $2$nd order linear functional derivative, in a sense of definition in \ref{def:ExtendedDerivative}, of $J$ exists,  and that there is $L > 0$ such that
	for any random variables $\eta$, $\eta'$ such that $\mathbb E[|\eta|^2 + |\eta'|^2 ] < \infty$, it holds that
	\begin{align} \label{as int 2nd}
	\mathbb E\left[ \sup_{\nu \in \mathcal V_2} \left| \int_0^T\frac{\delta J^{\mathcal M}}{\delta \nu}(\nu_t,t,\eta_t)dt\right|^2\right]
	+
	\mathbb E\bigg[ \sup_{\nu \in \mathcal V_2} \left|\int_{0}^T\int_{0}^T\frac{\delta^2 J^{\mathcal M}}{\delta \nu^2}(\nu_t,t,\eta_t,t',\eta'_{t}) dt dt'\right|^2\bigg] \leq L\,.
	\end{align}
Let  $(\theta^i)_{i=1}^N$ be i.i.d such that $\theta^i \sim \mu$,  $i=1,\ldots, N$.
Let $\mu^N := \frac1N \sum_{i=1}^N \delta_{\theta^i}$.
Then there is $c$ (independent of $N$, $p$ and $d$) such that
\[
	\mathbb E \left[ |J^{\mathcal M}(\mu^N) -  J^{\mathcal M}(\mu) |^2 \right] \leq  \frac{c}N\,.
\]
\end{lemma}
\begin{proof}
\textbf{Step 1.}
Let $\mu^N_\lambda := \mu + \lambda (\mu^N-\mu)$ for $\lambda \in [0,1]$ and let  $(\tilde{\theta}^i)_{i=1}^N$ be i.i.d., independent of $(\theta^i)_{i=1}^N$ and with law $\mu$.
By the definition of linear functional derivatives, we have
\[
\begin{split}
J^{\mathcal M}(\mu^N) -  J^{\mathcal M}(\mu) &=
\int_0^1 \int_0^T \int  \frac{\delta J^{\mathcal M}}{\delta \nu}(\mu^N_\lambda,t, a)\, (\mu^N_t-\mu_t)(da) dt d\lambda \\
&=\int_0^1 \int_0^T \frac1N\sum_{i=1}^N\left( \frac{\delta J^{\mathcal M}}{\delta \nu}(\mu^N_\lambda,t, \theta_t^i) -\mathbb E\left[ \frac{\delta J^{\mathcal M}}{\delta \nu}(\mu^N_\lambda,t, \tilde{\theta}_t^i)\right] \right) d t d\lambda  \\
&=\int_0^1 \int_0^T \frac1N\sum_{i=1}^N
\varphi^i_{\lambda} d t d\lambda
 \,.	
\end{split}
\]
where, for $i \in \{1, \ldots, N\}$ and $\lambda \in [0,1]$,
\begin{equation}
\varphi^i_{\lambda}  =   \frac{\delta J^{\mathcal M}}{\delta \nu}(\mu^N_\lambda,t,{\theta}^i)  -   \mathbb E \bigg[ \frac{\delta J^{\mathcal M}}{\delta \nu}(\mu^N_\lambda,t,\tilde{\theta}^i) \bigg].
\end{equation}
Note that the expectation only applies to $\tilde \theta^i$ and that $\varphi^i_{\lambda}$ is zero mean random variable. We have the estimate
\begin{equation}\label{eq variance J}
\mathbb E \left[ |J^{\mathcal M}(\mu^N) -  J^{\mathcal M}(\mu) |^2 \right] \leq \frac{T}{N^2} \int_0^1 \int_0^T \mathbb E \left[ \sum_{i=1}^N
(\varphi^i_{\lambda})^2 + \sum_{i_1\neq i_2}^N \varphi^{i_1}_{\lambda}\varphi^{i_2}_{\lambda} \right] d t d\lambda  \,.
\end{equation}
\textbf{Step 2.} By our assumption~\eqref{as int 2nd} we have
\[
\mathbb E \left[  \sum_{i=1}^N (\varphi^i_{\lambda})^2 \right]
\leq  2\sum_{i=1}^N \mathbb E \left[  \left|\frac{\delta J^{\mathcal M}}{\delta \nu}(\mu^N_\lambda,t,{\theta}^i)  \right|^2\right]
\leq 2 L N\,.
\]
\textbf{Step 3.} For any $(i_1,i_2)\in \{1,\ldots,N\}^2$, we introduce the (random) measures
\begin{align*}
\mu^{N,-(i_1,i_2)}_{\lambda} :=  \mu^{N}_{\lambda} + \frac{\lambda}{N} \sum_{k\in\{i_1,i_2\}}(\delta_{\tilde{\theta}^k} - \delta_{{\theta}^k}) \,\,\, \text{and} \,\,\,\mu^{N}_{\lambda,\lambda_1}:= (\mu^{N,-(i_1,i_2)}_{\lambda} -\mu^{N}_\lambda)\lambda_1 + \mu^{N}_\lambda,
\,\,\, \lambda,\lambda_1 \in [0,1]\,.
\end{align*}
By the definition of the second order functional derivative
\begin{equation} \label{eq fmu}
\begin{split}
&   \frac{\delta J^{\mathcal M}}{\delta \nu}(\mu^{N,-(i_1,i_2)}_{\lambda} ,t,{\theta}^i) - \frac{\delta J^{\mathcal M}}{\delta \nu}(\mu^N_\lambda,t,{\theta}^i)    \\
= &  \int_0^1 \int_0^T  \int \frac{\delta^2 J^{\mathcal M}}{\delta \nu^2}( {\mu}^{N}_{\lambda,\lambda_1},t,\tilde{\theta}^1_t,t',y_1) (   \mu^{N,-(i_1,i_2)}_{\lambda,t'} - \mu^{N}_{\lambda,t'}) (dy_1) \, d\lambda_1 dt'  \\
= & \frac{\lambda}{N}    \int_0^1 \int_0^T \int \sum_{k\in\{i_1,i_2\}} \frac{\delta^2 J^{\mathcal M}}{\delta \nu^2}( {\mu}^{N}_{\lambda,\lambda_1},t,\tilde{\theta}^1_t,t',y_1) (\delta_{\tilde{\theta}^k_{t'}} - \delta_{\theta^k_{t'}}) (dy_1) \, d\lambda_1 dt' \,.
\end{split}
\end{equation}
By our assumption~\eqref{as int 2nd} we have
\begin{equation*}
 \mathbb E \left[ \left|  \frac{\delta J^{\mathcal M}}{\delta \nu}(\mu^{N,-(i_1,i_2)}_{\lambda} ,t,{\theta}^i) - \frac{\delta J^{\mathcal M}}{\delta \nu}(\mu^N_\lambda,t,{\theta}^i) \right|^2 \right]
\leq \frac{4 T L }{N^2} \,.
\end{equation*}
In the same way we can show that
\[
 \mathbb E \left[ \left| \mathbb E \bigg[ \frac{\delta J^{\mathcal M}}{\delta \nu}(\mu^{N,-(i_1,i_2)}_{\lambda},t,\tilde{\theta}^i) \bigg]
 - \mathbb E \bigg[ \frac{\delta J^{\mathcal M}}{\delta \nu}(\mu^N_\lambda,t,\tilde{\theta}^i) \bigg]
 \right|^2 \right]
\leq \frac{4 T L }{N^2} \,.
\]
Hence
\[
\mathbb E[|\varphi^{i}-\varphi^{i,-(i_1,i_2)}|^2]\leq \frac{8TL}{N^2}\,,\,\,\text{where   }\varphi^{i,-(i_1,i_2)}_{\lambda}  =   \frac{\delta J^{\mathcal M}}{\delta \nu}(\mu^{N,-(i_1,i_2)}_\lambda,t,{\theta}^i)  -   \mathbb E \bigg[ \frac{\delta J^{\mathcal M}}{\delta \nu}(\mu^{N,-(i_1,i_2)}_\lambda,t,\tilde{\theta}^i) \bigg].
\]
Finally, by writing  $\varphi^{i} = (\varphi^{i}-\varphi^{i,-(i_1,i_2)}) + \varphi^{i,-(i_1,i_2)}$, applying Cauchy-Schwarz inequality and using \eqref{as int 2nd}  we have
\[
\mathbb E \left[ \sum_{i_1\neq i_2} \varphi^{i_1}_{\lambda}\varphi^{i_2}_{\lambda} \right]
\leq \sum_{i_1\neq i_2} \left(\frac{1}{N} +\mathbb E [\varphi^{i_1,-(i_1,i_2)}\varphi^{i_2,-(i_1,i_2)}]  \right) = N + \sum_{i_1\neq i_2} \mathbb E [\varphi^{i_1,-(i_1,i_2)}\varphi^{i_2,-(i_1,i_2)}]\,.
\]
By conditional independence argument the last term above is zero.
Combining this,~\eqref{eq variance J}, Conclusions of Step 1 and 2 concludes the proof.
\end{proof}

%

\begin{proof}[Proof of Theorem \ref{th generalisation}]
Throughout the proof we write $J = J^{\mathcal M}$.
We decompose the error as follows:
\[
\mathbb E\Big|J^{\mathcal M}(\nu^{\star,\sigma,N_1})- J^{\mathcal M}(\nu_{S,\cdot}^{\sigma,N_1,N_2,{\Delta s}})\Big|^2 \leq 4\Big(\mathbb E |\mathcal E_1|^2 + \mathbb E|\mathcal E_2|^2 + \mathbb E|\mathcal E_3|^2\Big)\,,
\]
where
\[
\begin{split}
\mathcal E_1 & :=  \,J^{\mathcal M}(\nu^{\star,\sigma,N_1}) - J^{\mathcal M}(\nu_{S,\cdot}^{\sigma,N_1})\,,\,
\mathcal E_2 := J^{\mathcal M}(\nu_{S,\cdot}^{\sigma,N_1}) -  J^{\mathcal M}(\nu_{S,\cdot}^{\sigma,N_1,N_2})\,,\,
\mathcal E_3 := J^{\mathcal M}(\nu_{S,\cdot}^{\sigma,N_1,N_2}) - J^{\mathcal M}(\nu_{S,\cdot}^{\sigma,N_1,N_2,\Delta s})\,.
\end{split}
\]
Here $\mathcal E_1$ is the error arising from running the mean-field gradient descent only for finite time $S$.
The error arising from replacing the mean-field gradient descent by a particle approximation is $\mathcal E_2$ and finally $\mathcal E_3$ arises from doing a time discretisation of the particle gradient descent.
\paragraph{Step 1.}
From Lemma~\ref{lem reg J}-$i)$ and Theorem \ref{thm conv to inv meas rate} we conclude that
\[
\mathbb E|\mathcal E_1|^2 = \mathbb E|J^{\mathcal M}(\nu^{\star,\sigma,N_1}) - J^{\mathcal M}(\nu_{S,\cdot}^{\sigma,N_1})|^2\leq
e^{-\lambda S}\,
L_{1,\mathcal M}^2 \, \mathbb E \Big[\mathcal W_2^T \Big(\mathcal L(\theta^{0}), \mathcal L( v^{\star,\sigma,N_1})  \Big)^2\Big]\,.
\]
\paragraph{Step 2.}
Consider i.i.d copies of the mean-field Langevin dynamic $(\theta^{i})_{i=1}^{N_2}$
\begin{equation}
 \theta^{i,\infty}_{s,t} =  \theta^{i,\infty}_{0,t}  - \int_0^s
 \left((\nabla_a \mathbf h_t)(\theta^{i,\infty}_{v,t},\mathcal L(\theta^{i,\infty}_{v,\cdot}),\mathcal M^{N_1})
+  \frac{\sigma^2}{2}(\nabla_a U)(\theta^{i,\infty}_{v,t}) \right)\,dv
+ \sigma dB^i_v\,.	
\end{equation}
The associated empirical measure is defined as
$\bar \nu^{\sigma,N_1,N_2}=\frac{1}{N_2}\sum_{i=1}^{N_2}\delta_{ \theta^{i,\infty}}$.
We have
\begin{equation}\label{eq generalisation chaos}
\begin{split}
\mathbb E| \mathcal E_2 |^2 = &  \mathbb E\Big|  J^{\mathcal M}(\nu_{S,\cdot}^{\sigma,N_1}) -  J^{\mathcal M}(\nu_{S,\cdot}^{\sigma,N_1,N_2})\Big|^2
\leq 2\mathbb E\Big| J^{\mathcal M}(\nu_{S,\cdot}^{\sigma,N_1}) - J^{\mathcal M}(\bar \nu_{S,\cdot}^{\sigma,N_1,N_2})\Big|^2 \\
& +  2\mathbb E \Big| J^{\mathcal M}(\bar \nu_{S,\cdot}^{\sigma,N_1,N_2}) -  J^{\mathcal M}(\nu_{S,\cdot}^{\sigma,N_1,N_2})\Big|^2 =: 2\mathbb E |\mathcal E_{2,1}|^2 + 2\mathbb E|\mathcal E_{2,2}|^2 \,.
\end{split}
\end{equation}
From Lemmas \ref{lem reg J}-$ii)$ and \ref{lem strong error} we see that
$\mathbb E|\mathcal E_{2,1}|^2 \leq \frac{c}{N_2}$.
Next we observe that by Lemma \ref{lem reg J}-$i)$ and by the definition of Wasserstein distance
\[
\begin{split}
\mathbb E|\mathcal E_{2,2}|^2 & =
\mathbb E \Big| J^{\mathcal M}(\bar \nu_{S,\cdot}^{\sigma,N_1,N_2}) -  J^{\mathcal M}(\nu_{S,\cdot}^{\sigma,N_1,N_2})\Big|^2
\leq\,  (L_{1,\mathcal M})^2 \, \mathbb E\Big[\mathcal W_2^{T}\Big(\bar \nu_{S,\cdot}^{\sigma,N_1,N_2},
\nu_{S,\cdot}^{\sigma,N_1,N_2}\Big)^2\Big]\\
& \leq \frac{(L_{1,\mathcal M})^2}{N_2} \sum_{i=1}^{N_2}
\int_0^T\mathbb E| \theta^{i,\infty}_{S,t}- \theta^i_{S,t}|^2dt\,.	
\end{split}
\]
Finally, due to Theorem~\ref{thm:WellPosed&PropagationChaos} we see that
$\int_0^T\mathbb E| \theta^{i,\infty}_{S,t}- \theta^i_{S,t}|^2\,dt\leq c \left(\frac1{N_1}+\frac1{N_2}\right)$ and so
\[
\mathbb E|\mathcal E_2|^2 \leq  c \left(\frac1{N_1}+\frac1{N_2}\right)\,.
\]
\paragraph{Step 3.}
By  Lemma~\ref{lem reg J}, by Theorem~\ref{thm:EulerRate1} and by the definition of Wasserstein distance
\[
\begin{split}
\mathbb E |\mathcal E_3|^2 & \leq
\mathbb E |J^{\mathcal M}(\nu_{S,\cdot}^{\sigma,N_1,N_2}) - J^{\mathcal M}(\nu_{S,\cdot}^{\sigma,N_1,N_2,{\Delta s}})|^2 \leq\,  (L^{1,J})^2 \,\mathcal  W_2^{T}(\nu_{S,\cdot}^{\sigma,N_1,N_2},
\nu_{S,\cdot}^{\sigma,N_1,N_2,{\Delta s}})\\
& \leq (L^{1,J})^2\frac{1}{N_2} \sum_{i=1}^{N_2}
\int_0^T\mathbb E| \theta^i_{S,t} - \widetilde \theta^i_{S,t}|^2dt
\leq (L^{1,J})^2 c{\Delta s}
\,,	
\end{split}
\]
where ${\Delta s} := \max_{0 < s_l < S} (s_l - s_{l-1})$.
Collecting conclusions of Steps 1, 2 and 3 we obtain
\[
\mathbb E\Big|J^{\mathcal M}(\nu^{\star,\sigma,N_1})- J^{\mathcal M}(\nu_{S,\cdot}^{\sigma,N_1,N_2,{\Delta s}})\Big|^2 \leq c\left(e^{-\lambda S} + \frac1{N_1} + \frac1{N_2} + {\Delta s} \right)\,,
\]
where $c$ is independent of $\lambda$, $S$, $N_1$, $N_2$, $d$, $p$ and the time partition used in Theorem~\ref{thm:EulerRate1}. 	
\end{proof}




\acks{This was work has been supported by The Alan Turing Institute under the Engineering and Physical Sciences Research Council grant EP/N510129/1.
The first author acknowledges the support of the Russian Academic Excellence Project `5-100'.
The second and third author would like to thank Kaitong Hu (CMAP, Ecole Polytechnique) and Zhenjie Ren (Universit\'e Paris--Dauphine) for their hospitality in Paris in June 2019 and the extensive discussions on the topic of this project.
The authors thank Andrew Duncan (Imperial College, London) for helpful discussions on the topic of Pontryagin principle. }

\appendix

\section{Measure derivatives}
\label{sec measure derivatives}

We first define flat derivative on $\Pp_2(\mathbb R^p)$. See e.g.~\cite[Section 5.4.1]{carmona2018probabilistic} for more details.

\begin{definition}
\label{def:flatDerivative}
A functional $U:\mathcal P_2(\mathbb R^p) \to \mathbb R$ is said to admit a linear derivative
if there is a (continuous on $\mathcal P_2(\mathbb R^p)$) map $\frac{\delta U}{\delta m} : \mathcal P(\mathbb R^p) \times \mathbb R^d \to \mathbb R$, such that $|\frac{\delta U}{\delta m}(a,\mu)|\leq C(1+|a|^2)$ and, for all
$m, m' \in\mathcal P_2(\mathbb R^p)$, it holds that
\[
U(m) - U(m') = \int_0^1 \int \frac{\delta U}{\delta m}(m + \lambda(m' - m),a) \, (m'
-m)(da)\,d\lambda\,.
\]
Since $\frac{\delta U}{\delta m}$ is only defined up to a constant we make a choice by demanding $\int \frac{\delta U}{\delta m}(m,a)\,m(da) = 0$.
	
\end{definition}

We will also need the linear functional derivative on $\Vv_2$, which provides a slight extension of the one introduced in the above Definition~\ref{def:flatDerivative}.

\begin{definition}\label{def:ExtendedDerivative} A functional $F:\Vv_2 \to \er^d$, is said to admit a first order linear derivative, if there exists a functional $\frac{\delta F}{\delta {{\nu}}}:\Vv_2\times(0,T)\times\er^p\rightarrow \er^d$, such that
\begin{enumerate}[i)]
\item For all $(t,a)\in(0,T)\times\er^p$, ${{\nu}}\in\Vv_2\mapsto \frac{\delta F}{\delta \nu}({\nu},t,a,)$ is continuous (for $\Vv_2$ endowed with the weak topology of $\mathscr M^+_b((0,T)\times\er^p)$).
\item For any $\nu \in \mathcal V_2$ there exists $C=C_{\nu,T,d,p} >0$ such that for all $a\in \mathbb R^p$ we have that
\[\left|\frac{\delta F}{\delta \nu}({\nu},t,a)\right|\leq C(1+|a|^q)\,.
\]
\item For all ${{\nu}},{\rho}\in\Vv_2$,
\begin{equation}\label{def:FlatDerivativeOnVa}
F({\rho})-F({{\nu}})=\int_{0}^{1}\int_0^T\int\frac{\delta F}{\delta {{\nu}}}((1-\lambda){{\nu}}+\lambda {\rho},t,a)\left({\rho_t}-{{\nu_t}}\right)(da)\,dt\,d\lambda.
\end{equation}
\end{enumerate}
The functional $\frac{\delta F}{\delta {{\nu}}}$ is then called the linear (functional) derivative of $F$ on $\Vv_2$.
\end{definition}
The linear derivative $\frac{\delta F}{\delta \nu}$ is here also defined up to the additive constant $\int_0^T\int\frac{\delta F}{\delta \nu}({\nu},t,a){{\nu_t}}(da)\,dt$. By a centering argument, $\frac{\delta F}{\delta\nu}$ can be generically defined under the assumption that
$\int_0^T\int\frac{\delta F}{\delta \nu}({\nu},t,a){{\nu_t}}(da)\,dt=0$.
Note that if $\frac{\delta F}{\delta \nu}$ exists according to Definition~\ref{def:ExtendedDerivative} then
\begin{equation}\label{def:FlatDerivativeOnVb}
\forall\,\nu,{\rho}\in\Vv_2,\,\lim_{\epsilon\rightarrow 0^+}\frac{F(\nu+\epsilon(\rho-\nu))-F(\nu)}{\epsilon}=\int_0^T\int\frac{\delta F}{\delta \nu}({\nu},t,a)\left(\rho_t-\nu_t\right)(da)\,dt.
\end{equation}
Indeed~\eqref{def:FlatDerivativeOnVa} immediately implies~\eqref{def:FlatDerivativeOnVb}.
To see the implication in the other direction take $v^\lambda := \nu + \lambda(\rho-\nu)$
and $\rho^\lambda := \rho - \nu + \nu^\lambda$ and notice that~\eqref{def:FlatDerivativeOnVb} ensures for all $\lambda \in [0,1]$ that
\[
\begin{split}
& \lim_{\varepsilon\rightarrow 0^+}\frac{F(\nu^\lambda+\varepsilon(\rho-\nu))-F(\nu^\lambda)}{\varepsilon}
= \lim_{\varepsilon\rightarrow 0^+}\frac{F(\nu^\lambda+\varepsilon(\rho^\lambda-\nu^\lambda))-F(\nu^\lambda)}{\varepsilon} \\
& =\int_0^T\int\frac{\delta F}{\delta \nu}(\nu^\lambda,t,a)\big({\rho^\lambda_t}-{{\nu^\lambda_t}}\big)(da)\,dt
= \int_0^T\int\frac{\delta F}{\delta {{\nu}}}(\nu^\lambda,t,a)\left({\rho_t}-{{\nu_t}}\right)(da)\,dt
\,.	
\end{split}
\]
By the fundamental theorem of calculus
\[
F({\rho})-F({{\nu}})=\int_0^1\lim_{\varepsilon\rightarrow 0^+}\frac{F(\nu^{\lambda+\varepsilon})-F(\nu^{\lambda})}{\varepsilon}\,d\lambda
=\int_0^1\int_{0}^{T}\int\frac{\delta F}{\delta\nu}(\nu^\lambda,t,a)({\rho_t}-\nu_t)(da)dt\,d\lambda\,.
\]
For the estimate on the generalization error (Section \ref{sec generalisation estimates}), we will also need to use a second order variation of $\nu\in\Vv_2\mapsto F(\nu)$ which is given by:

\begin{definition}\label{def:ExtendedSecondOrdDeriv}
We will say that $F:\Vv_2\rightarrow \er^d$ admits a second order linear functional derivative if, for all $t,a$, $\nu\mapsto \frac{\delta F}{\delta {\nu}}(\nu,t,a)$ itself admits a linear functional derivative in the sense of Definition \ref{def:ExtendedDerivative}.
We will denote this second order linear derivative by $\frac{\delta^2 F}{\delta {\nu}^2}$.
In particular
\begin{align*}
&\frac{\delta F}{\delta {\nu}}({\nu'},t,a)-\frac{\delta F}{\delta {\nu}}({\nu},t,a)=\int_{0}^{1}
\int_0^T\int\frac{\delta^2 F}{\delta {\nu}^2}((1-\lambda){\nu}+\lambda{\nu'},t,a,t',a')\,\left({\nu'}_{t'}-{\nu}_{t'}\right)(da')\,dt'\,d\lambda,
\end{align*}
where $(t',a',\nu)\mapsto\frac{\delta F}{\delta {\nu}}(\nu,t,a,t',a')$ satisfies the properties i) and ii) of Definition \ref{def:ExtendedDerivative}.
\end{definition}
Let us finally point out the following chain rule:
\begin{lemma}\label{lem:ChainRule} Assume that $F:\Vv_2\rightarrow \er^d$ admits a linear functional derivative, in the sense of Definition \ref{def:ExtendedDerivative}, and $J:\er^d\times\Vv_2\rightarrow\er^{d}$ is such that, for all ${{\nu}}\in\Vv_2$, $x\mapsto J(x,{{\nu}})$ admits a continuous differential $\nabla_xJ(x,{{\nu}})$ such that ${{\nu}}\mapsto\nabla_xJ(x,{{\nu}})$ is continuous on $\Vv_2$, and for all $x$, ${{\nu}}\mapsto J(x,{{\nu}})$ admits a continuous differential $\frac{\delta J}{\delta {{\nu}}}(x,\nu,t,a)$ on $\Vv_2$.
Then $\nu\mapsto J(\nu,F(\nu))$ admits a linear functional derivative on $\Vv_2$ given by
\[
\frac{\delta J}{\delta \nu}(F(\nu),\nu,t,a)+\nabla_xJ(F(\nu),\nu)\frac{\delta F}{\delta \nu}(\nu,t,a).
\]
\end{lemma}
\begin{proof} For all ${{\nu}},{\rho}$, we have
\begin{align*}
&J(F({{\nu}}+\epsilon({\rho}-{{\nu}})),{{\nu}}+\epsilon({\rho}-{{\nu}}))-J(F({{\nu}}),{{\nu}})\\
&=J(F({{\nu}}+\epsilon({\rho}-{{\nu}})),{{\nu}}+\epsilon({\rho}-{{\nu}}))-J(F({{\nu}}),{{\nu}}+\epsilon({\rho}-{{\nu}}))
+J(F({{\nu}},{{\nu}}+\epsilon({\rho}-{{\nu}}))-J(F({{\nu}}),{{\nu}})\\
&=\epsilon\int_{0}^{1}\nabla_xJ(F({{\nu}}+(\lambda+\epsilon)({\rho}-{{\nu}})),{{\nu}}+\epsilon({\rho}-{{\nu}}))
\left(F({{\nu}}+\epsilon({\rho}-{{\nu}}))-F({{\nu}})\right)d\lambda\\
&\quad +\epsilon\int_0^1\int_0^T\int\frac{\delta J}{\delta {{\nu}}}(F({{\nu}}),{{\nu}}+(\lambda+\epsilon)({\rho}-{{\nu}}),t,a)({\rho}_{t}(da)-{{\nu}}_{t}(da))\,dt\,d\lambda.
\end{align*}
Dividing this expression by $\epsilon$, the limit $\epsilon\rightarrow 0$, which grants the derivative of  ${{\nu}}\mapsto J({{\nu}},F({{\nu}}))$ (using \eqref{def:FlatDerivativeOnVb}), follows by dominated convergence.
\end{proof}

The connection between the linear functional derivative $\frac{\delta}{\delta \nu}$ introduced in Definition~\ref{def:ExtendedDerivative}
and  $\frac{\delta}{\delta m}$ introduced in Definition~\ref{def:flatDerivative} is the following one:
Let $\pi^t,\,0\leq t\leq T$ be the family of operators, which, for any $0\leq t\leq T$ and $\nu$ in $\Vv_2$ assigns the measure $\pi^t(\nu)=\nu_t$ of $\mathcal P_2(\mathbb R^p)$.
For any functional $U:\mathcal P_2(\mathbb R^p) \to \mathbb R$, define its extension on $\Vv_2$ as $U^t(\nu)=U(\pi^t(\nu))$. Whenever the functional $U^t$ admits a linear functional derivative on $\Vv_2$, then
\[
U^t(\nu')-U^t(\nu)=\int_{0}^{1}\int_{0}^{T}\int \frac{\delta U^t}{\delta \nu}(\nu+\lambda(\nu'-\nu),r,a)({\nu'}_r(da)-{\nu}_r(da))\,dr\,d\lambda.
\]
For any $m$ in $\Pp_2(\er^p)$, define for the measure $\nu^m$ of $\Vv_2$ constant in the sense $\nu^m_t(da)=m(da)$, for a.e. $t$. Therefore
we have $U^t(\nu^m)=U(m)$, and for all $m,m'$ in $\Pp_2(\er^p)$
\begin{align*}
&U(m')-U(m)=U^t({\nu'}^m)-U^t({\nu}^m)\\
&=\int_{0}^{1}\int_{0}^{T}\int \frac{\delta U^t}{\delta \nu}(\nu^m+\lambda({\nu'}^m-{\nu}^m),r,a)({\nu'}^m_r(da)-{\nu}^m_r(da))\,dr\,d\lambda\\
&=\int_{0}^{1}\left(\int_{0}^{T}\int \frac{\delta U^t}{\delta \nu}({\nu}^m+\lambda({\nu'}^m-{\nu}^m),r,a)\,dr\right)(m'(da)-m(da))\,d\lambda.
\end{align*}

\begin{lemma} \label{lem constant}
Fix $\nu \in \mathcal P(\mathbb R^m)$.
Let $u: R^m \to \mathbb R$ be such that for all $\mu \in \mathcal P(\mathbb R^m)$ we have that
\[
0 \leq \int u(a)\,(\mu - \nu)(da)\,.
\]
Then $u$ is a constant function: for all $a\in \mathbb R^m$ we have $u(a) = \int u(a')\nu(da')$.
\end{lemma}
\begin{proof}
Let $M := \int u(a) \,\nu(da)$.
Fix $\varepsilon > 0$.
Assume that $\nu(u - M \leq -\varepsilon) > 0$.
Indeed take $d\mu := \frac{1}{\nu(u - M \leq -\varepsilon)} \mathds{1}_{\{u - M \leq -\varepsilon\}}\,d\nu$.
Then
\[
\begin{split}
0 & \leq \int u(a)\,(\mu-\nu)(da) = \int [ u(a) - M ]\,\mu(da) 	\\
& = \int \mathds{1}_{\{u - M \leq -\varepsilon\}} [u(a)-M]\,\mu(da)
+ \int \mathds{1}_{\{u - M > -\varepsilon\}} [u(a)-M]\,\mu(da)\\
& = \int \mathds{1}_{\{u - M \leq -\varepsilon\}} [u(a)-M] \frac1{\nu(u - M \leq -\varepsilon)} \nu(da) \leq -\varepsilon\,.
\end{split}
\] 	
As this is a contradiction we get $\nu(u - M \leq -\varepsilon) = 0$ and taking $\varepsilon \to 0$ we get $\nu(u - M < 0) = 0$.
On the other hand assume that $\nu(u-M\geq \varepsilon) > 0$.
Then, since $u - M \geq 0$ holds $\nu$-a.s., we have
\[
0 = \int [u(a) - M]\,\nu(da) \geq \int_{\{u-M\geq \varepsilon\}} [u(a) - M]\,\nu(da) \geq \varepsilon \nu(u-M\geq \varepsilon)>0
\]
which is again a contradiction meaning that for all $\varepsilon > 0$ we have
$\nu(u-M\geq \varepsilon) = 0$ i.e. $u=M$ $\nu$-a.s..
\end{proof}

\begin{lemma}
\label{lemma diff of Ent flat}
Let $\nu_t,\mu_t \in \mathcal V_2$ and let $\nu^\varepsilon = \nu + \varepsilon(\mu-\nu)$.
Then
\begin{enumerate}[i)]
\item for any $\varepsilon \in (0,1)$ we have
\[
\frac1\varepsilon \int_0^T \left[\text{Ent}(\nu^\varepsilon_t) - \text{Ent}(\nu_t)\right]\,dt \geq \int_0^T \int [\log \nu_t(a)  - \log \gamma(a) ](\mu_t - \nu_t)(da)\,dt\,,
\]
\item
\[
\limsup_{\varepsilon \to 0} \frac1\varepsilon \int_0^T \left[\text{Ent}(\nu^\varepsilon_t) - \text{Ent}(\nu_t)\right]\,dt
\leq \int_0^T \int [\log \nu_t(a)  - \log \gamma(a) ](\mu_t - \nu_t)(da)\,dt\,,
\]
\end{enumerate}
where $m \mapsto \text{Ent}(m)$ is defined by~\eqref{eq Ent} for $m\in \mathcal P(\mathbb R^p)$.
\end{lemma}
\begin{proof}
This follows the steps of~\cite[Proof of Proposition 2.4]{hu2019mean}.	
For i) we begin by observing that
\[
\begin{split}
& \frac1\varepsilon \left(\text{Ent}(\nu^\varepsilon_t) - \text{Ent}(\nu_t)\right) = \frac1\varepsilon\int \bigg[\log \frac{\nu^\varepsilon_t(a)}{\gamma(a)} \nu^\varepsilon_t(a)  - \log \frac{\nu_t(a)}{\gamma(a)}  \nu_t(a) \bigg]\,da\\
& =  \frac1\varepsilon \int\big(\nu^\varepsilon_t(a) - \nu_t(a)\big) \log \frac{\nu_t(a)}{\gamma(a)} \,da + \frac1\varepsilon  \int\nu^\varepsilon_t(a) \bigg[\log \frac{\nu^\varepsilon_t(a)}{\gamma(a)} - \log \frac{\nu_t(a)}{\gamma(a)} \bigg]\,da\\
& =  \int\big(\mu_t(a) - \nu_t(a)\big) \log \frac{\nu_t(a)}{\gamma(a)} \,da + \frac1\varepsilon  \int\nu^\varepsilon_t(a) \log \frac{\nu^\varepsilon_t(a)}{\nu_t(a)}  \,da\\
& =  \int [\log \nu_t(a) - \log\gamma(a)](\mu_t - \nu_t)(da) + \frac1\varepsilon  \int\frac{\nu^\varepsilon_t(a)}{\nu_t(a)} \log \frac{\nu^\varepsilon_t(a)}{\nu_t(a)} \nu_t(a) \,da\,.
\end{split}
\]
Since $x\log x \geq x - 1$ for $x\in (0,\infty)$ we get
\[
\begin{split}
& \frac1\varepsilon  \int\frac{\nu^\varepsilon_t(a)}{\nu_t(a)} \log \frac{\nu^\varepsilon_t(a)}{\nu_t(a)} \nu_t(a) \,da
\geq \frac1\varepsilon  \int \bigg[\frac{\nu^\varepsilon_t(a)}{\nu_t(a)}  -1 \bigg] \nu_t(a) \,da = \frac1\varepsilon  \int \big[\nu^\varepsilon_t(a)  - \nu_t(a)\big] \,da = 0\,.
\end{split}
\]
Hence
\[
\frac1\varepsilon \left(\text{Ent}(\nu^\varepsilon_t) - \text{Ent}(\nu_t)\right)
\geq \int [\log \nu_t(a) - \log\gamma(a)](\mu_t - \nu_t)(da)\,.
\]
From this i) follows.

To prove ii) start by noting that
\[
\begin{split}
& \frac1\varepsilon \left(\text{Ent}(\nu^\varepsilon_t) - \text{Ent}(\nu_t)\right) = \frac1\varepsilon\int \big[\big(\log \nu^\varepsilon_t(a) - \log \gamma(a)\big) \nu^\varepsilon_t(a)  - \big(\log \nu_t(a) - \log \gamma(a) \big) \nu_t(a) \big]\,da\\
& = \int \frac1\varepsilon\big[\nu^\varepsilon_t(a) \log \nu^\varepsilon_t(a) - \nu_t(a)\log \nu_t(a) - \log \gamma(a)\big( \nu^\varepsilon_t(a) - \nu_t(a)\big)    \big]\,da\\
\end{split}
\]
Now
\[
- \frac1\varepsilon \log \gamma(a) (\nu^\varepsilon_t(a) - \nu_t(a)) = -\log \gamma(a) (\mu_t(a) - \nu_t(a))\,.
\]
Moreover, since the map $x\mapsto x\log x$ is convex for $x > 0$ we have
\[
\frac1\varepsilon\left[
\nu^\varepsilon_t(a)\log(\nu^\varepsilon_t(a)) - \nu_t(a)\log(\nu_t(a))
\right] \leq \mu(a)\log \mu(a) - \nu(a)\log \nu(a)\,.
\]
Hence
\[
\frac1\varepsilon \left(\text{Ent}(\nu^\varepsilon_t) - \text{Ent}(\nu_t)\right) \leq \text{Ent}(\mu_t) - \text{Ent}(\nu_t)\,.
\]
Since $\mu, \nu \in \mathcal V_2^W$ the right hand side is finite.
Finally, by the reverse Fatou's lemma,
\[
\begin{split}
& \limsup_{\varepsilon \to 0} \frac1\varepsilon \left[\text{Ent}(\nu^\varepsilon_t) - \text{Ent}(\nu_t)\right]
\leq \int \limsup_{\varepsilon \to 0} \frac1\varepsilon\big[\nu^\varepsilon_t(a) \log \nu^\varepsilon_t(a) - \nu_t(a)\log \nu_t(a) - \log \gamma(a)\big( \nu^\varepsilon_t(a) - \nu_t(a)\big)    \big]\,da\,.\\
\end{split}
\]
Calculating the derivative of $x\mapsto x \log x$ for $x>0$ leads to
\[
\begin{split}
& \limsup_{\varepsilon \to 0} \frac1\varepsilon\big[\nu^\varepsilon_t(a) \log \nu^\varepsilon_t(a) - \nu_t(a)\log \nu_t(a) - \log \gamma(a)\big( \nu^\varepsilon_t(a) - \nu_t(a)\big)    \big] \\
& = (1 + \log \nu_t(a))(\mu_t(a)-\nu_t(a)) - \log \gamma(a) (\mu_t(a) -\nu_t(a))\,.
\end{split}
\]
Hence
\[
\limsup_{\varepsilon \to 0} \frac1\varepsilon \left[\text{Ent}(\nu^\varepsilon_t) - \text{Ent}(\nu_t)\right]
\leq \int [\log \nu_t(a)  - \log \gamma(a) ](\mu_t - \nu_t)(da)\,.
\]
This completes the proof.
\end{proof}

\section{Bounds and regularity for the forward-backward system \eqref{eq mfsgd}}\label{ssec:TechnicalResults}
In this section, we establish the boundedness and regularity of the mapping assigning to each ${{\nu}}\in\Vv_2$ the solution to
\begin{equation}\label{eq:ForwardBackwardMapAppendix}
\left\{
 \begin{aligned}
X^{\xi,\zeta}_{t}({{\nu}})&=\xi+\int_0^t\int\phi_{r}(X^{\xi,\zeta}_r({{\nu}}),a,\zeta)\,\nu_r(da)\,dr,\\
 P^{\xi,\zeta}_{t}({{\nu}})&=\nabla_xg(X^{\xi,\zeta}_{T}({{\nu}}),\zeta)+\int_t^T \int
 \left(\nabla_x f_{r}(X^{\xi,\zeta}_{r}({{\nu}}),a,\zeta)+\nabla_x\phi_{r}(X^{\xi,\zeta}_r({{\nu}}),a,\zeta)\cdot P^{\xi,\zeta}_{r}({{\nu}})\right)\,\nu_r(da)\,dr.
 \end{aligned}
 \right.
 \end{equation}

 Hereafter, we will work mostly under the sole assumption \ref{as coefficients}, and, for a fixed couple of data $\xi,\zeta$ chosen according to
 \eqref{as coefficients}-iii), so that for a.e. $0\leq t\leq T$,
 \[
 |\nabla_xg(0,\zeta)|+|\phi_{t}(0,0,\zeta)|+|\nabla_a f_{t}(0,0,\zeta)|+|\nabla_x f_{t}(0,0,\zeta)|<\infty.
 \]

%
%

\noindent
Let us also recall the function $\bbH$ given by:
\[
\bbH_{t}(x,a,p,\zeta)=f_{t}(x,a,\zeta)+\phi_t(x,a,\zeta) p.
\]

Let us point out that the process $(P^{\xi,\zeta}_t({{\nu}}))_{0\leq t\leq T}$ can be written as:
\begin{equation}\label{eq:ForwardReformulated}
\begin{aligned}
&P^{\xi,\zeta}_{t}({{\nu}})=\widehat \Xi^{\xi,\zeta,\nu}(T,t)\nabla_xg(X^{\xi,\zeta}_{T}({{\nu}}),\zeta)-\int_t^T \widehat \Xi^{\xi,\zeta,\nu}(r,t)\left(\int\nabla_x f_{r}(X^{\xi,\zeta}_{r}({{\nu}}),a,\zeta)\,\nu_r(da)\right)\,dr.
\end{aligned}
\end{equation}
where $\widehat \Xi^{\xi,\zeta,\nu}$ denote $\nu \in \mathcal V_2$, the continuous $\mathbb R^{d\times d}$-valued function, solution to: 
\begin{equation}\label{PrincipalSolutionBackward}
\frac{d\widehat \Xi^{\xi,\zeta,\nu}(t,t_0)}{dt_0}=\big(\int\nabla_x\phi_{t}(X^{\xi,\zeta}_{t}(\nu),a,\zeta)\nu_{t}(da)\big)\widehat \Xi^{\xi,\zeta,\nu}(t,t_0),\,t_0\leq t\leq T,\,\,\Xi^{\xi,\zeta,\nu}(t,t)=I_d,
\end{equation}
for $I_d$ denoting the identity matrix of size $d$.
As a first estimate, let us show the following lemma:
\begin{lemma}[Uniforms bounds and continuity]\label{lem:UniformBounds_Continuity} Under Assumption \ref{as coefficients}, for any ${{\nu}}\in\Vv_2$, $(X^{\xi,\zeta}_t({{\nu}}),P^{\xi,\zeta}_t({{\nu}}))_{0\leq t\leq T}$ given by \eqref{eq:ForwardBackwardMapAppendix} satisfies:
\[
|X^{\xi,\zeta}_t({{\nu}})|\leq \left(|\xi|+\int_0^t\int|\phi_t(0,a,\zeta)|\nu_r(da)\,dr\right)\times\exp\{T\sup_{t,a,\zeta}\Vert \phi_{t}(\cdot,a,\zeta)\Vert_{Lip}\},
\]
\begin{align*}
&|P^{\xi,\zeta}_t({{\nu}})| \leq  \Vert\nabla_xg (\cdot,\zeta)\Vert_{\infty}\wedge\big(\Vert \nabla_x g(\cdot,\zeta)\Vert_{Lip} |X^{\xi,\zeta}_t({{\nu}})|+| \nabla_xg(0,\zeta)|\big)
\times\exp\{T\sup_{t,a,\zeta}\Vert \phi_{t}(\cdot,a,\zeta)\Vert_{Lip}\}\\
&+\Vert\nabla_xf_t\Vert_{\infty}\times\exp\{T\sup_{t,a,\zeta}\Vert \phi_{t}(\cdot,a,\zeta)\Vert_{Lip}\}.
%
\end{align*}
Additionally, for all $0\leq t\leq T$, ${{\nu}}\mapsto (X^{\xi,\zeta}_t({{\nu}}),P^{\xi,\zeta}_t({{\nu}}))$ is continuous on $\Vv_2$ equipped with the topology related to the metric $\Ww^{T}_{2}$.
\end{lemma}
\begin{proof} Owing to the Lipschitz and differentiability properties of $(x,a)\mapsto\phi_{t}(x,a,\zeta)$,
\[
|X^{\xi,\zeta}_t({{\nu}})|\leq |\xi|+\sup_{t,a,\zeta}\Vert\phi_{t}(\cdot,a,\zeta)\Vert_{Lip}\int_{0}^{t}|X^{\xi,\zeta}_r({{\nu}})|\,dr+\int_{0}^{t}\int|\phi_{t}(0,a,\zeta)|\nu_r(da)\,dr.
\]
Applying Gronwall's inequality: For non-negative continuous functions $u,w,\alpha$
\begin{equation}\label{eq:GronwallLemma}
u(t)\leq w(t)+\alpha\int_0^t u(s)\,ds,\,\forall \,0\leq t\leq T\,\Rightarrow u(t)
\leq \sup_{t\leq T}w(t) \exp\{\alpha T\},\,\forall \,0\leq t\leq T,
\end{equation}
yields to the estimate of $X^{\xi,\zeta}_t({{\nu}})$. The estimate for $P^{\xi,\zeta}_t({{\nu}})$ follows directly from \eqref{eq:ForwardReformulated}.

For the continuity of ${{\nu}}\mapsto (X^{\xi,\zeta}_t({{\nu}}),P^{\xi,\zeta}_t({{\nu}}))$, let $\{{{\nu}}^\epsilon\}_{\epsilon>0}$ be a family of elements of $\Vv_2$ such that $\lim_{\epsilon\rightarrow 0^+}\Ww^{T}_{2}({{\nu}}^\epsilon,{{\nu}})=0$. Observe that
\begin{align*}
&\left|X^{\xi,\zeta}_t({{\nu}}^{\epsilon})-X^{\xi,\zeta}_t({{\nu}})\right|\leq \int_{0}^{t}\int\left\{\left|\phi_{r}(X^{\xi,\zeta}_r({{\nu}}^{\epsilon}),a,\zeta)-\phi_{r}(X_r({{\nu}}),a,\zeta)\right|\right\}
{{\nu}}^\epsilon_{r}(da)\,dr\\
&\quad +\left|\int_{0}^{t}\int\left\{\phi_{r}(X^{\xi,\zeta}_r({{\nu}}),a,\zeta)\right\}\left({{\nu}}^{\epsilon}_{r}(da)-\nu_r(da)\right)\,dr\right|\\
&\leq \sup_{t,a,\zeta}\Vert\phi_{t}(\cdot,a,\zeta)\Vert_{Lip} \int_{0}^{t}\left|X^{\xi,\zeta}_r({{\nu}}^{\epsilon})-X^{\xi,\zeta}_r({{\nu}})\right|\,dr+
\sup_{t,x,\zeta}\Vert\phi_{t}(x,.,\zeta)\Vert_{Lip}\int_{0}^{t}\Ww_1({{\nu}}^{\epsilon}(r),{{\nu}}(r))\,dr.
\end{align*}
Applying Gronwall's inequality \eqref{eq:GronwallLemma} and
since $\Ww^{T}_{1}\leq \Ww^{T}_{2}$, it follows that $\lim_{\epsilon\rightarrow 0}X^{\xi,\zeta}_t({{\nu}}^\epsilon)=X^{\xi,\zeta}_t({{\nu}})$.

Since $(x,a)\mapsto \nabla_xg(x,\zeta),\nabla_x\phi_{t}(x,a,\zeta),\nabla_xf_{t}(x,a,\zeta)$ are bounded continuous (uniformly in $t$), the continuity of $({{\nu}})\mapsto X^{\xi,\zeta}_t({{\nu}})$ also ensure, by dominated convergence that
\[
\lim_{\epsilon\rightarrow 0}\nabla_x g(X^{\xi,\zeta}_T({{\nu}}^\epsilon),\zeta)=\nabla_x g(X_T({{\nu}}),\zeta),
\]
\[
\lim_{\epsilon\rightarrow 0}\int_0^t\int\left\{\nabla_x f_{r}(X^{\xi,\zeta}_r({{\nu}}^\epsilon),a,\zeta)\right\}{{\nu}}^{\epsilon}_{r}(da)\,dr
=\int_0^t\int\left\{\nabla_x f_{r}(X^{\xi,\zeta}_r({{\nu}}),a,\zeta)\right\}\nu_r(da)\,dr,
\]
\[
\lim_{\epsilon\rightarrow 0}\int_0^t\int\left\{\nabla_x \phi_{r}(X^{\xi,\zeta}_r({{\nu}}^\epsilon),a,\zeta)\right\}{{\nu}}^{\epsilon}_{r}(da)\,dr
=\int_0^t\int\left\{\nabla_x \phi_{r}(X^{\xi,\zeta}_r({{\nu}}),a,\zeta)\right\}\nu_r(da)\,dr.
\]
This ensures that $\lim_{\epsilon\rightarrow 0}P^{\xi,\zeta}_t({{\nu}}^\epsilon)=P^{\xi,\zeta}_t({{\nu}})$.
\end{proof}


Let us prove the differentiability, in the sense of Definition \ref{def:ExtendedDerivative}, of the mapping ${{\nu}}\in \Vv_2\mapsto (X^{\xi,\zeta}_t({{\nu}}),Y^{\xi,\zeta}_t({{\nu}}))$:

\begin{lemma}\label{lem:FirstDerivative} Let $\widehat \Xi^{\xi,\zeta,\nu}$ be as in \eqref{PrincipalSolutionBackward} and let $(t,t_0)\mapsto \Xi^{\xi,\zeta,\nu}(t,t_0)$ be the solution to
\begin{equation}\label{PrincipalSolutionForward}
\frac{d\Xi^{\xi,\zeta,\nu}(t,t_0)}{dt}=\big(\int\nabla_x\phi_{t}(X^{\xi,\zeta}_{t}(\nu),a,\zeta)\nu_{t}(da)\big)\Xi^{\xi,\zeta,\nu}(t,t_0),\,t_0\leq t\leq T,\,\,\Xi^{\xi,\zeta,\nu}(t_0,t_0)=I_d.
\end{equation}
For all $0\leq t\leq T$, ${{\nu}}\mapsto (X^{\xi,\zeta}_t({{\nu}}),Y^{\xi,\zeta}_t({{\nu}}))$ admits a linear functional derivative given as the solution to the ODEs
\[
\frac{\delta X^{\xi,\zeta}_t}{\delta \nu}({\nu},r,a)=\1_{\{r\leq t\}}\Xi^{\xi,\zeta,\nu}(t,r)\phi_{r}(X^{\xi,\zeta}_{r}({{\nu}}),a,\eta),
\]
and
\begin{align*}
&\frac{\delta P^{\xi,\zeta}_t}{\delta \nu}({\nu},r,a)=\widehat \Xi^{\xi,\zeta,\nu}(T,t)\nabla^2_xg(X^{\xi,\zeta}_T(\nu),\zeta)\times\frac{\delta X^{\xi,\zeta}_{T}}{\delta \nu}({\nu},r,a)\\
&+\int_t^T\widehat \Xi^{\xi,\zeta,\nu}(r,t)\nabla_xh_{r}(X^{\xi,\zeta}_r(\nu),P^{\xi,\zeta}_r(\nu),a,\zeta)\,dr\\
&+\int_{t}^{T}\widehat \Xi^{\xi,\zeta,\nu}(r',t)
\nabla^2_xh_{r'}(X^{\xi,\zeta}_{r'}({{\nu}}),P^{\xi,\zeta}_{r'}({{\nu}}),a',\zeta)\nu_{r'}(da')\,dr'\\
&+\left\{\int_t^T
\widehat \Xi^{\xi,\zeta,\nu}({r'},t)\int \nabla^2_xh_{r'}(X^{\xi,\zeta}_{r'}({{\nu}}),P^{\xi,\zeta}_{r'}({{\nu}}),a',\zeta)\nu_{r'}(da')\,dr'\right\}\frac{\delta X^{\xi,\zeta}_T}{\delta \nu}({\nu},r,a),
\end{align*}
where $h$ is defined as in \eqref{eq hamiltonian}.
\end{lemma}
\begin{proof}
The derivative of ${{\nu}}\mapsto X^{\xi,\zeta}_t({{\nu}})$ follows directly Lemma \ref{lemma V process linear}, which grants
\begin{align*}
&\int_0^T\int\frac{\delta X^{\xi,\zeta}_t}{\delta {{\nu}}}(r,a,{{\nu}}+\epsilon({\rho}-{{\nu}}))\left({\rho}_{r}(da)-{{\nu}}_{r}(da)\right)\,dr=\lim_{\epsilon\rightarrow 0}\frac{X^{\xi,\zeta}_t({{\nu}}+\epsilon({\rho}-{{\nu}}))-X^{\xi,\zeta}_t({{\nu}})}{\epsilon}=V^{\xi,\zeta}_t
\end{align*}
for $(V^{\xi,\zeta}_{t})_{0\leq t\leq T}$ satisfying:
\[
V^{\xi,\zeta}_t=\int_0^t \big(\int\nabla_x\phi_{r}(X^{\xi,\zeta}_{r}({{\nu}}),a,\zeta)\,\nu_r(da)\big)V^{\xi,\zeta}_r\,dr+\int_{0}^{t}\int \phi_{r}(X^{\xi,\zeta}_r({{\nu}}),a,\zeta)
\left(\rho_{r}(da)-\nu_r(da)\right)\,dr,\,0\leq t\leq T.
\]
Recall that any solution to the ODE $du(t)/dt=b(t)u(t)+\alpha(t)\,\text{on}\,[0,T]$, $u(0)=u^0,$ admits the representation:
\begin{equation}\label{ProofTrivia1}
u(t)=\Psi(t,0)u^0+\int_0^t\Psi(t,s)\alpha(s)\,ds,\,
\frac{d\Psi(t,t_0)}{dt}=b(t)\Psi(t,t_0),\,\Psi(t_0,t_0)=I_d,
\end{equation}
we get, using the expression of $V^{\xi,\zeta}_t$,
\begin{align*}
&\int_0^T\int\frac{\delta X^{\xi,\zeta}_t}{\delta\nu}({\nu},r,a)\left({\rho}_{r}(da)-{{\nu}}_{r}(da)\right)\,dr\\
&=\int_0^t\Xi^{\xi,\zeta,\nu}(t,r)\phi_{r}(X^{\xi,\zeta}_r({{\nu}}),a,\zeta)(\rho_{r}(da)-\nu_r(da))\,dr.
\end{align*}
%
The Lipschitz properties of $(x,a)\mapsto \phi_{t}(x,a,\zeta)$ and $(x,a)\mapsto \nabla_x\phi_{t}(x,a,\zeta)$ ensure, with Lemma \ref{lem:UniformBounds_Continuity} that ${{\nu}}\mapsto X^{\xi,\zeta}_r({{\nu}})$, and by extension ${{\nu}}\mapsto \phi_{r}(X^{\xi,\zeta}_r({{\nu}}),a,\zeta),\int_r^t\int\nabla_x\phi_{r'}(X^{\xi,\zeta}_{r'}({{\nu}}),a',\zeta)\nu_{r'}(da')\,dr'$ are continuous. In particular,
 $\nabla_x\phi_{t}(x,a,\zeta)$ is uniformly bounded and
\[
|\phi_{r}(X^{\xi,\zeta}_r({{\nu}}),a,\zeta)|\leq C(1+|X^{\xi,\zeta}_r({{\nu}})|+|a|+|\phi_t(0,0,\zeta)|)\leq C(1+|\xi|+|a|+|\phi_t(0,0,\zeta)|),
\]
for some finite constant $C$ depending only on $T$, $d$, ${p}$, $\Vert\nabla_x\phi\Vert_{\infty}$ and $\Vert\nabla_a\phi\Vert_{\infty}$.
Therefore, as 
\[
\Vert \Xi^{\xi,\zeta,\nu}(t,t_0)\Vert :=\sup_{|v|= 1}|\Xi^{\xi,\zeta,\nu}(t,r) v|\leq 1+\Vert\nabla_x\phi\Vert_\infty \int_{t_0}^t \Vert \Xi^{\xi,\zeta,\nu}(r,t_0)\Vert\,dr 
\]
Gronwall's inequality yields
\[
\Vert \Xi^{\xi,\zeta,\nu}(t,t_0)\Vert \leq \exp\{(t-t_0)\Vert\nabla_x\phi\Vert_\infty\}
\]
and so
\[
\left|\Xi^{\xi,\zeta,\nu}(t,r)\phi_{r}(X^{\xi,\zeta}_r({{\nu}}),a,\zeta)\right|\leq C(1+|a|+|\xi|+|\phi_t(0,0,\zeta)|).
\]
This enables us to conclude that
\[
\1_{\{r\leq t\}}\Xi^{\xi,\zeta,\nu}(t,r)\phi_{r}(X^{\xi,\zeta}_r({{\nu}}),a,\zeta),
\]
is the linear derivative functional of ${{\nu}}\mapsto X^{\xi,\zeta}_t({{\nu}})$.

In the same way, for ${{\nu}}_\epsilon={{\nu}}+\epsilon({\rho}-{{\nu}})$, we have
\begin{align*}
&P^{\xi,\zeta}_t({{\nu}}_\epsilon)-P^{\xi,\zeta}_t({{\nu}})=\nabla_xg(X^{\xi,\zeta}_T({{\nu}}_\epsilon),\zeta)-\nabla_xg(X^{\xi,\zeta}_T({{\nu}}),\zeta)\\
&\quad +\int_{t}^{T} \left(\int \nabla_xf_{r}(X^{\xi,\zeta}_r({{\nu}}_\epsilon),a,\zeta){{\nu}}^{\epsilon}_{r}(da)-\int \nabla_xf_{r}(X^{\xi,\zeta}_r({{\nu}}),a,\zeta)\nu_r(da)\right)\,dr\\
&\quad +\int_{t}^{T}\left(\int\nabla_x\phi_{r}(X^{\xi,\zeta}_r({{\nu}}_\epsilon),a,\zeta){{\nu}}^{\epsilon}_{r}(da)
-\int\nabla_x\phi_{r}(X^{\xi,\zeta}_r({{\nu}}),a,\zeta)\nu_r(da)\right)P^{\xi,\zeta}_r({{\nu}})\,dr\\
&\quad +\int_{t}^{T}\int\nabla_x\phi_{r}(X^{\xi,\zeta}_r({{\nu}}_\epsilon),a)\left(P^{\xi,\zeta}_r({{\nu}}_\epsilon)-P^{\xi,\zeta}_r({{\nu}})\right)\nu_r(da)\,dr\\
&=\nabla_xg(X^{\xi,\zeta}_T({{\nu}}_\epsilon),\zeta)-\nabla_xg(X^{\xi,\zeta}_T({{\nu}}),\zeta)\\
&\quad +\int_{t}^{T}\int \left(\nabla_xh_{r}(X^{\xi,\zeta}_r({{\nu}}_\epsilon),P^{\xi,\zeta}_r({{\nu}}),a){{\nu}}^{\epsilon}_{r}(da)-\nabla_xh_{r}(X^{\xi,\zeta}_r({{\nu}}),
P^{\xi,\zeta}_r({{\nu}}),a,\zeta)\nu_r(da)\right)\,dr\\
&\quad +\int_{t}^{T}\big(\int\nabla_x\phi_{r}(X^{\xi,\zeta}_r({{\nu}}_\epsilon),a,\zeta)\nu_r(da)\big)\left(P^{\xi,\zeta}_r({{\nu}}_\epsilon)-P^{\xi,\zeta}_r({{\nu}})\right)\,dr.
\end{align*}
Using \eqref{PrincipalSolutionBackward}, we deduce the formulation:
\begin{align*}
&P^{\xi,\zeta}_t({{\nu}}_\epsilon)-P^{\xi,\zeta}_t({{\nu}})=\widehat \Xi^{\xi,\zeta,\nu_\epsilon}(T,t)
\times\big(\nabla_xg(X^{\xi,\zeta}_T({{\nu}}_\epsilon),\zeta)-\nabla_xg(X^{\xi,\zeta}_T({{\nu}}),\zeta)\big)\\
&\quad+\int_{t}^{T}\widehat \Xi^{\xi,\zeta,\nu_\epsilon}(r,t)\\
&\quad\quad\quad\times\left(\int\nabla_xh_{r}(X^{\xi,\zeta}_r({{\nu}}_\epsilon),P^{\xi,\zeta}_r({{\nu}}),a,\zeta){{\nu}}^{\epsilon}_{r}(da)-
\int\nabla_xh_{r}(X^{\xi,\zeta}_r({{\nu}}),P^{\xi,\zeta}_r({{\nu}}),a,\zeta)\nu_r(da)\right)\,dr.
\end{align*}
As $\lim_{\epsilon\rightarrow 0}\widehat \Xi^{\xi,\zeta,\nu_\epsilon}(T,t)=\widehat \Xi^{\xi,\zeta,\nu}(T,t)$, using Lemma \ref{lem:ChainRule} and Lemma \ref{lem:UniformBounds_Continuity} (observing that $\Ww_{T,2}({{\nu}}_\epsilon,{{\nu}})\leq \epsilon\Ww_{T,2}({{\nu}}',{{\nu}})\rightarrow 0$ as $\epsilon\rightarrow 0+$), we obtain
\begin{align*}
&\lim_{\epsilon\rightarrow 0}\frac{P^{\xi,\zeta}_t({{\nu}}+\epsilon({\rho}-{{\nu}}))-P^{\xi,\zeta}_t({{\nu}})}{\epsilon}\\
&=\int_0^T\widehat \Xi^{\xi,\zeta,\nu}(T,t)\int\left(\nabla^2_xg(X^{\xi,\zeta}_T({{\nu}}),\zeta) \frac{\delta X^{\xi,\zeta}_T}{\delta {{\nu}}}({\nu},r,a) \right)(\rho_{r}(da)-\nu_r(da))\,dr\\
&\quad+\int_t^T\widehat \Xi^{\xi,\zeta,\nu}(r,t)\left(\int_0^T\int \nabla_xh_{r}(X^{\xi,\zeta}_r({{\nu}}),P^{\xi,\zeta}_r({{\nu}}),a,\zeta)\left(\rho_{r}(da)-\nu_r(da)\right)\,dr\right)\\
&\quad+\int_t^T \widehat \Xi^{\xi,\zeta,\nu}(r,t)\\
&\quad\quad\times\left(\int_0^T\int\nabla^2_xh_{r}(X^{\xi,\zeta}_r({{\nu}}),P^{\xi,\zeta}_r({{\nu}}),a,\zeta)\frac{\delta X^{\xi,\zeta}_r}{\delta {{\nu}}}({\nu},r'',a'')\left(\rho_{r''}(da'')-\nu_{r''}(da'')\right)d{r''}\right)\nu_r(da)\,dr,
\end{align*}
from which we identify the value of $\frac{\delta P^{\xi,\zeta}_t}{\delta {{\nu}}}(\nu,r,a)$.
\end{proof}

Lemma \ref{lem:FirstDerivative} together with the definition of linear derivative allows to compute $X(\nu)-X(\mu)$ and  $P(\nu)-P(\mu)$. However, to establish propagation of chaos an alternative representation is more convenient.
\begin{lemma} \label{lem difference P and X}
	Let Assumptions \ref{as coefficients} and \ref{ass exist and uniq} hold.
	Then 	
\[
|X^{\xi,\zeta}_t(\nu) - X^{\xi,\zeta}_t(\mu)|  \leq \exp(\| \nabla_x \phi\|_{\infty}(T-t)) \int_{0}^t \bigg| \int \phi_t(X^{\xi,\zeta}_{r}(\mu),a,\zeta) \,(\nu_r - \mu_r)(da) \bigg|\,dr\,,
\]
and
\[
 |P^{\xi,\zeta}_{t}(\nu) - P^{\xi,\zeta}_{t}(\mu)|  \leq c_1
|X^{\xi,\zeta}_T(\nu)-X^{\xi,\zeta}_T(\mu)|  + c_2
\int_t^T | X^{\xi,\zeta}_r(\nu)-X^{\xi,\zeta}_r(\mu)|dr\,,
\]
where $c_1$ and $c_2$ are given in \eqref{eq constants P}.
\end{lemma}
\begin{proof}
Let $X_r^{\lambda}:=X^{\xi,\zeta}_r(\nu) + \lambda (X^{\xi,\zeta}_r(\mu) - X^{\xi,\zeta}_r(\nu)) $ and write
\begin{align*}
&X^{\xi,\zeta}_t(\nu)-X^{\xi,\zeta}_t(\mu)
=\int_{0}^{t}\left(\int \phi_r(X^{\xi,\zeta}_r(\nu),a,\zeta) \nu_r(da)-
\int \phi_r(X^{\xi,\zeta}_r(\mu),a,\zeta)\mu_r(da)\right)\,dr\\
&=\int_{0}^{t} \left(\int \left( \phi_r(X^{\xi,\zeta}_r(\nu),a,\zeta)-
\phi_r(X^{\xi,\zeta}_r(\mu),a,\zeta) \nu_r(da)\right) \right)\,dr\\
& + \int_{0}^{t} \left(\int  \phi_r(X^{\xi,\zeta}_r(\mu),a,\zeta)(\nu_r - \mu_r)(da) \right)\,dr \\
&=\int_{0}^{t} \left(\int \int_0^1 (\nabla_x \phi_r)(X^{\lambda}_r,a,\zeta)\, d\lambda\, \nu_r(da) \right)\left (X^{\xi,\zeta}_r(\nu) -X^{\xi,\zeta}_r(\mu)\right)\,dr\\
& + \int_{0}^{t} \left(\int  \phi_r(X^{\xi,\zeta}_r(\mu),a,\zeta)(\nu_r - \mu_r)(da) \right)\,dr\,.
\end{align*}
Let $(r,t)\mapsto\Gamma_{r,t}$ be the solution to 
\[
\frac{d\Gamma_{r,t}}{dr}=\big(\int \int_0^1 (\nabla_x \phi_r)(X^{\lambda}_{r'},a,\zeta) \nu_{r'}(da)\,d\lambda\big),\,\Gamma_{r,r}=I_d.
\]
We then have
\begin{align*}
&X^{\xi,\zeta}_t(\nu)-X^{\xi,\zeta}_t(\mu)
=\int_{0}^{t} \Gamma_{r,t}\left(\int  \phi_r(X^{\xi,\zeta}_r(\mu),a,\zeta)(\nu_r - \mu_r)(da) \right)dr\,.
\end{align*}
Assumption \ref{as coefficients} implies that $|\Gamma_{r,t}|\leq \exp((t-r)\| \nabla_x \phi|_{\infty})$ and leads immediately to the estimate for $|X^{\xi,\zeta}_t(\nu) - X^{\xi,\zeta}_t(\mu)|$.

Let us fix $\xi,\zeta$.
Let us write 
$\mathbf f_t(x,a):=f_t(x,a,\zeta)$ and $\mathbf G(x) := g(x,\zeta)$.
From \eqref{eq:ForwardReformulated} we have
\begin{equation*}
\begin{aligned}
&P^{\xi,\zeta}_{t}(\nu) - P^{\xi,\zeta}_{t}(\mu)  =\widehat \Xi^{\xi,\zeta,\nu}(T,t)
\left(\nabla_x\bG(X_{T}({{\nu}}))-  \nabla_x\bG(X_{T}({{\mu}}))\right)\\
&+ \left(\widehat \Xi^{\xi,\zeta,\mu}(T,t)-\widehat \Xi^{\xi,\zeta,\nu}(T,t)\right)\nabla_x\bG(X_{T}({{\mu}})) \\
&- \int_t^T \widehat \Xi^{\xi,\zeta,\nu}(r,t)\left(  \int\nabla_x \bL_{r}(X_{r}({{\nu}}),a)\,\nu_r(da) -  \int\nabla_x \bL_{r}(X_{r}({{\mu}}),a)\,\nu_r(da) \right)
\,dr \\
&-\int_t^T \left(  \widehat \Xi^{\xi,\zeta,\mu}(r,t)-\widehat \Xi^{\xi,\zeta,\nu}(r,t)\right)\left(\int\nabla_x \bL_{r}(X_{r}({{\mu}}),a)\,\nu_r(da)\right)\,dr.
\end{aligned}
\end{equation*}
Applying the mean-value theorem in $X$ and using Assumption \ref{as coefficients} implies that
\[
 |P^{\xi,\zeta}_{t}(\nu) - P^{\xi,\zeta}_{t}(\mu)|  \leq c_1
|X^{\xi,\zeta}_T(\nu)-X^{\xi,\zeta}_T(\mu)|  + c_2
\int_t^T | X^{\xi,\zeta}_r(\nu)-X^{\xi,\zeta}_r(\mu)|dr\,,
\]
where
\begin{equation}\label{eq constants P}
\begin{split}
	c_1 & =\|\nabla_x^2 g\|_{\infty} e^{T \|\nabla_x \phi\|_{\infty}}\,\\
	c_2 &=
\| \nabla_x g\|_{\infty }e^{T \|\nabla_x \phi\|_{\infty}} \|\nabla_x^2 \phi\|_{\infty}
+ \| \nabla_x^2 f\|_{\infty }e^{T \|\nabla_x \phi\|_{\infty}}
+ T \| \nabla_x f\|_{\infty }e^{T \|\nabla_x \phi\|_{\infty}} \|\nabla^2_x \phi \|_{\infty}\,.
\end{split}	
\end{equation}
%
\end{proof}

\section{Regularity estimates on the Hamiltonian}
In this section, we prove the following result:
\begin{theorem}\label{thm:RegularityHamiltonian}
Let Assumption \ref{as coefficients} hold. Let $\nabla_a\bH$ be the function defined on $[0,T]\times\er^p\times\Vv_2\times\Pp(\DataS)$ by
\[
\nabla_a\bH_t(a,\nu,\Mm)=\int_{\DataS} \nabla_a\Lagrang_t(a,X^{\xi,\zeta}_t(\nu))+\nabla_a\phi_t(a,X^{\xi,\zeta}_t(\nu))\cdot P^{\xi,\zeta}_t(\nu)\,\Mm(d\xi,d\zeta),
\]
for $(X^{\xi,\zeta}_t(\nu),P^{\xi,\zeta}_t(\nu))$ satisfying
\begin{equation}\label{eq:ForwardBackwardMap}
 \left\{
 \begin{aligned}
&X^{\xi,\zeta}_{t}({{\nu}})=\xi+\int_0^t\int_{\er^{\ControlS}} \phi_{r}(X^{z}_r({{\nu}}),a,\zeta)\,{{\nu}}_{r}(da)\,dr,\\
 &P^{\xi,\zeta}_{t}({{\nu}})=\nabla_xg(X^{\xi,\zeta}_{T}({{\nu}}),\zeta)+\int_t^T \int_{\er^{\ControlS}}
 \nabla_x h_{r}(X^{\xi,\zeta}_{r}({{\nu}}),a,P^{\xi,\zeta}_{r}({{\nu}}),{\zeta})\,{{\nu}}_{r}(da)\,dr.
 \end{aligned}
 \right.
 \end{equation}
Then there exists $L>0$ such that for all $\DataM \in \mathcal P_2(\mathbb R^d\times\mathcal S)$, for all $a,a' \in \mathbb R^p$ and $\mu,\mu' \in \mathcal V_2$
\begin{equation}
\begin{aligned} \label{con lipschitz appendix}
&| (\nabla_a \bH_t)(a,\mu,\DataM) - (\nabla_a \bH_t)(a',\mu',\DataM) |\\
& \leq L \left(1+\max_t\big(\max_{\mu\in\text{Lin}(\nu,\nu')}\int_{\DataS}|P^{\xi,\zeta}_t(\mu)|\,\Mm(d\xi,d\zeta) \big)\right) \left( |a-a'| + \mathcal W^T_1(\mu,\mu')\right).
\end{aligned}
\end{equation}
for $\text{Lin}(\nu,\nu'):=\left\{\mu\in\Vv_2\,:\, \mu=(1-\lambda)\nu'+\lambda \nu,\,\text{for some}\,0\leq \lambda\leq 1\right\}$.
%
\end{theorem}
\begin{proof}
From the definition of $\nabla_a\bH$, we have
\[
(\nabla_a \bH_t)(a,\nu,\DataM)
=  \int_{\mathbb R^d \times \mathcal S}\Big[\nabla_a\phi_t(X^{\xi,\zeta}_{t}(\nu),a,\zeta)\cdot P^{\xi,\zeta}_{t}(\mu) + \nabla_a f_t(X^{\xi,\zeta}_{t}(\mu),a,\zeta)\Big] \Mm(d\xi,d\zeta)\,.
\]
This, Lemma~\ref{lemma odes} and Assumption \ref{as coefficients} $ii)$ allow us to conclude that for any $\DataM \in \mathcal P_2(\mathbb R^d\times\mathcal S)$ there exists $L$
such that for all $\mu \in \mathcal V_2$ we have
\begin{equation}\label{HamiltonRegul:proofstp1}
| (\nabla_a \bH_t)(a, \nu, \DataM) - (\nabla_a \bH_t)(a', \nu, \DataM) | \\
\leq L\left(1+ \int_{\DataS} |P^{\xi,\zeta}_t(\nu)|\,\DataM(d\xi,d\zeta)\right) |a-a'| \,,
\end{equation}
for some constant $L$ depending only on the Lipschitz coefficients of $\nabla_a\Lagrang$ and $\nabla_a\phi$.

On the other hand, for all $a\in\er^p$, we have
\[
\begin{split}
& | (\nabla_a \bH_t)(a, \nu, \DataM) - (\nabla_a \bH_t)(a,\nu',\DataM) | \\
& \leq \int_{\er^{p}\times\mathcal{S}} \left[\Vert\nabla_a \Lagrang_t(a,.,\zeta)\Vert_{Lip} \left|X^{\xi,\zeta}_{t}(\nu) - X^{\xi,\zeta}_{t}(\nu')\right|  \right]\,\DataM (d\xi,d\zeta)\\
&+ \int_{\er^{p}\times\mathcal{S}} \left[\Vert\phi_{t}(x,.,\zeta)\Vert_{Lip} \left|P^{\xi,\zeta}_{t}(\nu) - P^{\xi,\zeta}_{t}(\nu')\right| + \|\nabla_a \phi(\cdot,a,\zeta)\|_{Lip} |P^{\xi,\zeta}_{t}(\nu)|\left|X^{\xi,\zeta}_{t}(\nu) - X_{t}(\nu')\right|\right]\,\DataM (d\xi,d\zeta).
\end{split}
\]
Recalling Lemma \ref{lem:FirstDerivative}, $\nu\in\Vv_2\mapsto X^{\xi,\zeta}_t(\nu)$ and $P^{\xi,\zeta}_t(\nu)$ both admit a linear functional derivative $\frac{\delta X^{\xi,\zeta}_t}{\delta {{\nu}}}(r,a,\nu)$ and $\frac{\delta P^{\xi,\zeta}_t}{\delta {{\nu}}}(r,a,\nu)$ (see Definition \ref{def:ExtendedDerivative}), which are given by
\begin{equation}\label{HamiltonRegul:proofstp2}
\frac{\delta X^{\xi,\zeta}_t}{\delta {{\nu}}}({\nu},r,a)=\1_{\{0\leq r\leq t\}}\Xi^{\xi,\zeta,\nu}(t,r)\phi_{r}(X^{\xi,\zeta}_r({{\nu}}),a,\zeta),
\end{equation}
\begin{equation}\label{HamiltonRegul:proofstp3}
\begin{aligned}
&\frac{\delta P^{\xi,\zeta}_t}{\delta \nu}({\nu},r,a)=\widehat \Xi^{\xi,\zeta,\nu}(T,t)\nabla^2_xg(X^{\xi,\zeta}_T(\nu),\zeta)\times\frac{\delta X^{\xi,\zeta}_{T}}{\delta \nu}({\nu},r,a)\\
&+\int_t^T\widehat \Xi^{\xi,\zeta,\nu}(r,t)\nabla_xh_{r}(X^{\xi,\zeta}_r(\nu),P^{\xi,\zeta}_r(\nu),a,\zeta)\,dr\\
&+\int_{t}^{T}\widehat \Xi^{\xi,\zeta,\nu}(r',t)
\nabla^2_xh_{r'}(X^{\xi,\zeta}_{r'}({{\nu}}),P^{\xi,\zeta}_{r'}({{\nu}}),a',\zeta)\nu_{r'}(da')\,dr'\\
&+\left\{\int_t^T
\widehat \Xi^{\xi,\zeta,\nu}({r'},t)\int \nabla^2_xh_{r'}(X^{\xi,\zeta}_{r'}({{\nu}}),P^{\xi,\zeta}_{r'}({{\nu}}),a',\zeta)\nu_{r'}(da')\,dr'\right\}\frac{\delta X^{\xi,\zeta}_T}{\delta \nu}({\nu},r,a),
\end{aligned}
\end{equation}
Since $a\mapsto \phi_{t}(x,a,\zeta)$ is uniformly Lipschitz continuous, $a\mapsto \frac{\delta X^{\xi,\zeta}_{t}}{\delta\nu}({\nu},r,a)$ is also uniformly Lipschitz continuous, uniformly in $r,\nu,\xi$ and $\zeta$ with
\begin{equation}\label{PropaChaos:proofst5}
\begin{aligned}
&\Vert \frac{\delta X^{\xi,\zeta}_t}{\delta \nu}({\nu},r,\cdot)\Vert_{Lip}:=\sup_{a\neq a'}\frac{\left|\frac{\delta X^{\xi,\zeta}_t}{\delta {{\nu}}}({\nu},r,a')-\frac{\delta X^{\xi,\zeta}_t}{\delta\nu}({\nu},r,a)\right|}{|a-a'|}\\
&\leq\Vert \phi_t(a,\cdot,\zeta)\Vert_{Lip}\exp\left\{T\sup_{t,a,\zeta}\Vert \phi_t(a,\cdot,\zeta)\Vert_{Lip}\right\}.
\end{aligned}
\end{equation}
In the same way, $a\mapsto \frac{\delta P_{t}}{\delta \nu}({\nu},r,a)$ is also uniformly Lipschitz continuous with
\begin{equation}\label{PropaChaos:proofst6}
\begin{aligned}
&\Vert \frac{\delta P^{\xi,\zeta}_t}{\delta\nu}({\nu},r,\cdot)\Vert_{Lip}:=\sup_{a\neq a'}\frac{\left|\frac{\delta P^{\xi,\zeta}_t}{\delta \nu}({\nu},r,a')-\frac{\delta P^{\xi,\zeta}_t}{\delta \nu}({\nu},r,a)\right|}{|a-a'|}\\
&\leq L\left(1+ \int_{\DataS}|P^{\xi,\zeta}_r({{\nu}})|\DataM(d\xi,d\zeta)+\int_0^T\int_{\DataS}|P^{\xi,\zeta}_t({{\nu}})|\DataM(d\xi,d\zeta)\,dt\right).
\end{aligned}
\end{equation}
For $\lambda \in (0,1)$, define $\mu^{\lambda}=(1-\lambda) \mu' + \lambda \mu$. Let $Lip(1)$ denote the class of Lipschitz functions with Lipschitz constant bounded by $1$. From the definition of the functional derivative
\ref{def:ExtendedDerivative},
\begin{align*}
&\left| X^{\xi,\zeta}_{t}(\nu)-X^{\xi,\zeta}_{t}(\nu')\right| = \left|\int_{0}^{1}\int_0^T\int \frac{\delta X^{\xi,\zeta}_{t}}{\delta \nu}(\nu^{\lambda},r,a)
\left(\nu_{r}(da)-\nu'_{r}(da)\right)\,dr\,d\lambda \right| \\
&\leq  \sup_{t,r\in(0,T),\nu \in \mathcal V_2} \|  \frac{\delta X^{\xi,\zeta}_{t}}{\delta\nu}(\nu,r,\cdot) \|_{Lip}\left|\sup_{c\in Lip(1)} \int_{0}^{T}\int c(a) \left(\nu_{r}(da)-\nu'_{r}(da)\right)dr\right|\\
&\leq L\left|\sup_{c\in Lip(1)} \int_{0}^{T}\int c(a) \left(\nu_{r}(da)-\nu'_{r}(da)\right)dr\right|\,.
\end{align*}
We can estimate $P^{\xi,\zeta}_{t}(\mu) - P^{\xi,\zeta}_{t}(\nu)$ in the same way, obtaining here:
\begin{align*}
&\left| P^{\xi,\zeta}_{t}(\nu)-P^{\xi,\zeta}_{t}(\nu')\right| \\
&=\left|\int_{0}^{1}\int_0^T\int \frac{\delta P^{\xi,\zeta}_{t}}{\delta {{\nu}}}(r,a,\nu^{\lambda})
\left(\nu_{r}(da)-\nu'_{r}(da)\right)\,dr\,d\lambda \right|\\
&\leq  L\left(1+\sup_{t\in(0,T),\mu \in \text{Lin}(\nu,\nu')}\int_{\DataS}\left|P^{\xi,\zeta}_t(\mu)\right|\,\DataM(d\xi,d\zeta) \right)\left|\sup_{c\in Lip(1)} \int_{0}^{T}\int c(a) \left(\nu_{r}(da)-\nu'_{r}(da)\right)dr\right|,
\end{align*}
for $\text{Lin}(\nu,\nu'):=\left\{\mu\in\Vv_2\,:\, \mu=(1-\lambda)\nu'+\lambda \nu,\,\text{for some}\,0\leq \lambda\leq 1\right\}$.

By Kantorovich (dual) representation of the Wasserstein distance \cite[Th 5.10]{villani2008optimal} we conclude that
there is $L>0$ such that
\begin{equation}\label{HamiltonRegul:proofstp6}
 | (\nabla_a \bH_t)(a, \nu,\DataM) - (\nabla_a \bH_t)(a, \nu',\DataM) |\leq L \, \mathcal W^T_1(\nu,\nu')\,.
\end{equation}
Combining \eqref{HamiltonRegul:proofstp1} and \eqref{HamiltonRegul:proofstp6}, we then conclude.
\end{proof}

\section{Full discretisation scheme}\label{ssec:FullDiscrete}
In this section, we illustrate a simple example of a full-time discretization of the particle \eqref{eq:ParticleApprox-bis}, complementing the time-discretization \eqref{TimeDiscrete:proofst1} by adding a discretization in the $t$ variable.

Consider a finite partition $0=t_0< t_1< ...< t_n=T$ of the interval $[0,T]$. Let $\{\widehat{X}^{\xi,\zeta}_{k}(\nu)\}_{k}$ and $\{\widehat{P}^{\xi,\zeta}_{k}(\nu)\}_{k}$ be a uniform $\beta$-order approximation ($\beta> 0$) of $(X^{\xi,\zeta}_t(\nu))$ and $(P^{\xi,\zeta}_t(\nu))$, in the sense that

\begin{equation}\label{eq:DiscreteFB}
\sup_{\xi,\zeta,\nu}|X^{\xi,\zeta}_{t_k}(\nu)-\widehat{X}^{\xi,\zeta}_k(\nu)|+|P^{\xi,\zeta}_{t_k}(\nu)-\widehat{P}^{\xi,\zeta}_k(\nu)|\leq C\max_{k_1\leq k}|t_{k_1}-t_{k_1-1}|^\beta.
\end{equation}
We again refer to \cite{hairer:norsett:wanner:1993} for an exhaustive presentation of numerical approximation scheme for ODEs.

On the other hand, from the partial discretization $(\widetilde{\theta}^{i}_{s_l,t_k})$ defined in \eqref{TimeDiscrete:proofst1},
we introduce the discretization of the time variable $t$ with from the frozen dynamic:
\begin{equation}\label{eq:ParticleDiscreteApproximation_a}
\widehat{\theta}^{i}_{s,t}=\theta^{0,i}_{\eta_n(t)}-
\int_{0}^{s}\nabla_a\overline{\bH}^{\sigma,n}_{\eta_n(t)}
\left(\widehat{\theta}^{i}_{{\Lambda}_M(v),\eta_n(t)},\widehat{{{\nu}}}^{N_2}_{{\Lambda}_M(v),\cdot},\DataM^{N_1}\right)\,dv+\sigma \ControlBrownian^{i}_s,\,\widehat{{{\nu}}}^{N_2}_{v,t}=\frac{1}{N_2}\sum_{j_2=1}^{N_2}\delta_{\{\widehat{\theta}^{j_2}_{v,t}\}},
\end{equation}
where $\eta_n(t)=\inf\{t_k\,:\,t_k\leq t\}$, and $\nabla \overline{\bH}^{\sigma,n}_t$ is defined on $[0,T]\times\Vv_2\times\Pp(\DataS)$ by
\begin{align*}
&\nabla_a\overline{\bH}^{\sigma,n}_{t_k}(a,{{\nu}},\DataM^{N_1})=\nabla_aU(a)+
\nabla_a\overline{\bH}^{n}_{t_k}(a,{{\nu}},\DataM^{N_1}),\\
&\overline{\bH}^{n}_{t_k}(a,{{\nu}},\DataM^{N_1})= \int_{\mathbb R^d \times \mathcal S} h_{t_k}(\widehat{X}^{\xi,\zeta}_k(\nu_{\eta_n(\cdot),\cdot}),\widehat{P}^{\xi,\zeta}_k(\nu_{\eta_n(\cdot),\cdot}),a,\zeta)\DataM^{N_1}(d\xi, d\zeta).
\end{align*}
where $\nu_{\eta_n(\cdot)}$ is the discretized version of $\nu$ at times $t_0,t_1,\cdots$.

The rate of convergence between \eqref{TimeDiscrete:proofst1} and \eqref{eq:ParticleDiscreteApproximation_a} is given by:

\begin{proposition}\label{prop:EulerRate2} Assume that the assumptions of Lemma \ref{prop:EulerRate1} hold. Assume also that the following properties hold:

\noindent
$(D_1)$ The discrete schemes $\widehat{X}$ and $\widehat{P}$ satisfy \eqref{eq:DiscreteFB} as well as the properties:
\begin{enumerate}[i)]
  \item $\sup_{k,\nu}\int |\widehat{P}^{\xi,\zeta}_k(\nu)|\DataM^{N_1}(d\xi,d\zeta)<\infty$,
  \item there exists some constant $L'$ such that for all $\nu,\nu'$ in $\Vv_2$,
\[
|\widehat{X}^{\xi,\zeta}_{k}(\nu)-\widehat{X}^{\xi,\zeta}_{k}(\nu')|\leq L'T\sup_{0\leq t\leq T}W_1(\nu_{t},\nu'_{t}),
\]
and
\[
|\widehat{P}^{\xi,\zeta}_{k}(\nu)-\widehat{P}^{\xi,\zeta}_{k}(\nu')|\leq L'T\sup_{0\leq t\leq T}W_1(\nu_{t},\nu'_{t}).
\]
\end{enumerate}
$(D_2)$ The functions
$$
t\mapsto \phi_{t}(x,a,\zeta),\nabla_x\phi_{t}(x,a,\zeta),\nabla_x\Lagrang_{t}(x,a,\zeta),\nabla_a\phi_{t}(x,a,\zeta),\nabla_a\Lagrang_{t}(x,a,\zeta)
$$
are all (uniformly in $x,a,\zeta$) of class $\Cc^\alpha$ (for $0<\alpha\leq 1$), that is, for some $0<L''<\infty$,
\begin{align*}
&\left|\phi_{t}(x,a,\zeta)-\phi_{t'}(x,a,\zeta)\right|+\left|\nabla_x\phi_{t}(x,a,\zeta)-\nabla_x\phi_{t'}(x,a,\zeta)\right|
+\left|\nabla_a\phi_{t}(x,a,\zeta)-\nabla_a\phi_{t'}(x,a,\zeta)\right|\\
&+\left|\nabla_x\Lagrang_{t}(x,a,\zeta)-\nabla_x\Lagrang_{t'}(x,a,\zeta)\right|+\left|\nabla_a\Lagrang_{t}(x,a,\zeta)-\nabla_a\Lagrang_{t'}(x,a,\zeta)\right|\leq L''|t-t'|^{\alpha},
\end{align*}
for all $t,t'\in[0,T]$.

\noindent
$(D_3)$ The initial flow $t\mapsto \theta^{0,i}_t$ satisfies the properties: $\sup_{0\leq t\leq T}\EE[|\theta^{0,i}_t|^2]<\infty$ and
\begin{equation*}
\EE\left[\left|\theta^{0,i}_{t}-\theta^{0,i}_{t'}\right|^2\right]\leq L|t-t'|^{2\alpha},\,\forall t,t'\in[0,T].
\end{equation*}

\noindent
Then, there exists $0<c<\infty$ independent of $(s_l)_l$, $N_1$ and $N_2$,  such that, for all integers $i$, $L$,
\begin{align*}
&\sup_{k,1\leq l\leq L}\EE\left[\left|\widehat{\theta}^i_{l,k}-\widehat{\theta}^i_{s_l,t_k}\right|^2\right]\\
&\leq c\big(1+\max_{l\leq L}(s_l-s_{l-1})\big)(\max_t|t-\eta_n(t)|^{2(\alpha\wedge \beta\wedge 1)}
\left(1+\int_{0}^T\EE\left[\left|\widetilde{\theta}^i_{s_{l},{r'}}\right|^2\right]\,d{r'}\right).
\end{align*}
\end{proposition}
\begin{proof}

\textbf{Step 1}. The assumption $(D_1)$ immediately ensures that: for all $a,a'$, $\nu,\nu'\in\Vv_2$,
\begin{equation}\label{EulerRate2:proofstp1}
|\nabla_a\overline{\bH}^n_{t}(a,\nu,\Mm^{N_1})-\nabla_a\overline{\bH}^n_{t}(a',\nu',\Mm^{N_1})|\leq L\left(|a-a'|+\sup_{0\leq t\leq T}W_1(\nu_{\eta_n(t)},\nu'_{\eta_n(t)})\right).
\end{equation}
For simplicity, we will again omit from now on the explicit notation of the component $\DataM^{N_1}$ in most of the calculations below.

Setting $\triangle_{s_l,t_k}\theta^i:=\widetilde{\theta}^i_{s_l,t_k}-\widehat{\theta}^i_{s_l,t_k}$, observe that, from all $1\leq l\leq M$,
$1\leq k\leq n$,
\begin{align*}
&\triangle_{s_l,t_k}\theta^i=\triangle_{s_{l-1},t_k}\theta^i-\frac{\sigma^2}{2}\left(s_{l}-s_{l-1}\right)
\left(\nabla_aU(\widetilde{\theta}^{i}_{s_{l-1},t_k})-
\nabla_aU(\widehat{\theta}^{i}_{s_{l-1},t_k})\right)\\
&\quad -\left(s_{l}-s_{l-1}\right)\left(
\nabla_a\bH_{t_k}\left(\widetilde{\theta}^{i}_{s_{l-1},t_k},\widehat{{{\nu}}}^{N_2}_{s_{l-1},\cdot},\DataM^{N_1}\right)
-\nabla_a\overline{\bH}^{n}_{t_k}
\left(\widetilde{\theta}^{i}_{s_{l-1},t_k},\widetilde{{{\nu}}}^{N_2}_{s_{l-1},\cdot},\DataM^{N_1}\right)\right)\\
&\quad -\left(s_{l}-s_{l-1}\right)\left(
\nabla_a\overline{\bH}^{n}_{t_k}\left(\widetilde{\theta}^{i}_{s_{l-1},t_k},\widetilde{{{\nu}}}^{N_2}_{s_{l-1},\cdot},\DataM^{N_1}\right)
-\nabla_a\overline{\bH}^{n}_{t_k}
\left(\widehat{\theta}^{i}_{s_{l-1},t_k},\widehat{{{\nu}}}^{N_2}_{s_{l-1},\cdot},\DataM^{N_1}\right)\right).
\end{align*}
Proceeding as in \eqref{EulerRate1:proofst4},
\begin{align*}
&\left|\triangle_{s_l,t_k}\theta^{i}\right|^2\leq\left|\triangle_{s_{l-1},t_k}\theta^{i}\right|^2\left(1-(\sigma^2\kappa-3)(s_{l}-s_{l-1})\right)\\
&\quad +(s_l-s_{l-1})\left|\nabla_a \overline{\bH}^n_{t_k}
\left(\widetilde{\theta}^{i}_{s_{l-1},t_k},\widetilde{{{\nu}}}^{N_2}_{s_{l-1},\cdot},\DataM^{N_1}\right)
-\nabla_a\overline{\bH}^n_{t_k}\left(\widehat{\theta}^{i}_{s_{l-1},t_k},\widehat{{{\nu}}}^{N_2}_{s_{l-1},\cdot},\DataM^{N_1}\right)\right|^2\\
&\quad +2(s_l-s_{l-1})^2\left|\nabla_a \overline{\bH}^{n}_{t_k}
\left(\widetilde{\theta}^{i}_{s_{l-1},t_k},\widetilde{{{\nu}}}^{N_2}_{s_{l-1},\cdot},\DataM^{N_1}\right)
-\nabla_a\overline{\bH}^{n}_{t_k}\left(\widehat{\theta}^{i}_{s_{l-1},t_k},\widehat{{{\nu}}}^{N_2}_{s_{l-1},\cdot},\DataM^{N_1}\right)\right|^2\\
&\quad  +(s_l-s_{l-1})(1+(s_l-s_{l-1})/2)\left|\nabla_a \bH_{t_k}
\left(\widetilde{\theta}^{i}_{s_{l-1},t_k},\widetilde{{{\nu}}}^{N_2}_{s_{l-1},\cdot},\DataM^{N_1}\right)
-\nabla_a\overline{\bH}^{n}_{t_k}\left(\widetilde{\theta}^{i}_{s_{l-1},t_k},\widetilde{{{\nu}}}^{N_2}_{s_{l-1},\cdot},\DataM^{N_1}\right)\right|^2.
\end{align*}
From \eqref{EulerRate2:proofstp1}, we deduce
\begin{align*}
\EE\left[\left|\nabla_a \overline{\bH}^n_{t_k}
\left(\widetilde{\theta}^{i}_{s_{l-1},t_k},\widetilde{{{\nu}}}^{N_2}_{s_{l-1},\cdot},\DataM^{N_1}\right)
-\nabla_a\overline{\bH}^n_{t_k}\left(\widehat{\theta}^{i}_{s_{l-1},t_k},\widehat{{{\nu}}}^{N_2}_{s_{l-1},\cdot},\DataM^{N_1}\right)\right|^2\right]
\leq L\max_k\EE\left[\left|\triangle_{s_{l-1},t_k}\theta^{i}\right|^2\right],
\end{align*}
and
\begin{align*}
&\EE\left[\left|\nabla_a \overline{\bH}^{n}_{t_k}
\left(\widetilde{\theta}^{i}_{s_{l-1},t_k},\widetilde{{{\nu}}}^{N_2}_{s_{l-1},\cdot},\DataM^{N_1}\right)
-\nabla_a\overline{\bH}^{n}_{t_k}\left(\widehat{\theta}^{i}_{s_{l-1},t_k},\widehat{{{\nu}}}^{N_2}_{s_{l-1},\cdot},\DataM^{N_1}\right)\right|^2\right]\\
&\leq L(1+\frac{\sigma^4}{2}\Vert\nabla_aU(\cdot)\Vert^2_{Lip})\max_k\EE\left[\left|\triangle_{s_{l-1},t_k}\theta^{i}\right|^2\right].
\end{align*}
Then it follows that
\begin{equation}\label{EulerRate2:proofstp2}
\begin{aligned}
&\EE\left[\left|\triangle_{s_l,t_k}\theta^{i}\right|^2\right]\leq\EE\left[\left|\triangle_{s_{l-1},t_k}\theta^{i}\right|^2\right]
\left(1-(\sigma^2\kappa-3)(s_{l}-s_{l-1})\right)\\
&+L(s_l-s_{l-1})\Big(1+(s_l-s_{l-1})(1+\frac{\sigma^4}{2}\Vert\nabla_aU(\cdot)\Vert^2_{Lip})\Big)
\max_k\EE\left[\left|\triangle_{s_{l-1},t_k}\theta^{i}\right|^2\right] \\
&+(s_l-s_{l-1})(1+(s_l-s_{l-1})/2)\EE\left[\left|\nabla_a \bH_{t_k}
\left(\widetilde{\theta}^{i}_{s_{l-1},t_k},\widetilde{{{\nu}}}^{N_2}_{s_{l-1},\cdot},\DataM^{N_1}\right)
-\nabla_a\overline{\bH}^{n}_{t_k}\left(\widetilde{\theta}^{i}_{s_{l-1},t_k},\widetilde{{{\nu}}}^{N_2}_{s_{l-1},\cdot},\DataM^{N_1}\right)\right|^2\right].
\end{aligned}
\end{equation}
\textbf{Step $2$}. Owing to the regularity of $\Lagrang$, $\phi$ and $\nabla_aU$, we have,
for all $0\leq t\leq T$,
\begin{align*}
&\EE\left[|\nabla_a \bH_{t}
\left(\widetilde{\theta}^{i}_{s_{l-1},t},\widetilde{{{\nu}}}^{N_2}_{s_{l-1},\cdot},\DataM^{N_1}\right)
-\nabla_a\overline{\bH}^{n}_{t}\left(\widetilde{\theta}^{i}_{s_{l-1},t},\widetilde{{{\nu}}}^{N_2}_{s_{l-1},\cdot},\DataM^{N_1}\right)|^2\right]\\
&\leq \int \EE\left[\left| h_{t_k}(X^{\xi,\zeta}_{t_k}(\widetilde{\nu}_{s_{l-1},\eta_n(\cdot)}),P^{\xi,\zeta}_{t_k}(\widetilde{\nu}_{s_{l-1},\eta_n(\cdot)}),a,\zeta)
-h_{t_k}(\widehat{X}^{\xi,\zeta}_k(\widetilde{\nu}_{s_{l-1},\eta_n(\cdot)}),\widehat{P}^{\xi,\zeta}_k( \widetilde{\nu}_{s_{l-1},\eta_n(\cdot)}),a,\zeta)\right|^2\right]\DataM^{N_1}(d\xi,d\zeta)\\
&\quad +\int \EE\left[\left| h_{t_k}(X^{\xi,\zeta}_{t_k}(\widetilde{\nu}_{s_{l-1},\cdot}),P^{\xi,\zeta}_{t_k}(\widetilde{\nu}_{s_{l-1},\cdot}),a,\zeta) -h_{t_k}(X^{\xi,\zeta}_{t_k}(\widetilde{\nu}_{s_{l-1},\eta_n(\cdot)}),P^{\xi,\zeta}_{t_k}(\widetilde{\nu}_{s_{l-1},\eta_n(\cdot)}),a,\zeta)
\right|^2\right]\DataM^{N_1}(d\xi,d\zeta).
\end{align*}
From \eqref{eq:DiscreteFB} and \eqref{EulerRate2:proofstp1}, we deduce directly that
\begin{equation}\label{EulerRate2:proofstp3}
\begin{aligned}
&\EE\left[|\nabla_a \bH_{t}
\left(\widetilde{\theta}^{i}_{s_{l-1},t},\widetilde{{{\nu}}}^{N_2}_{s_{l-1},\cdot},\DataM^{N_1}\right)
-\nabla_a\overline{\bH}^{n}_{t}\left(\widetilde{\theta}^{i}_{s_{l-1},t},\widetilde{{{\nu}}}^{N_2}_{s_{l-1},\cdot},\DataM^{N_1}\right)|^2\right]\\
&\leq c\sup_k|t_{k}-t_{k-1}|^{2\beta}+c\int_{0}^{T}\EE[|\widetilde{\theta}^i_{s_{l-1},t}-\widetilde{\theta}^i_{s_{l-1},\eta_n(t)}|^2]\,dt.
\end{aligned}
\end{equation}
It now remains to estimate $\EE[|\widetilde{\theta}^i_{s_{l},t}-\widetilde{\theta}^i_{s_{l},\eta_n(t)}|^2]$.
When $l=0$, $\EE[|\widetilde{\theta}^i_{s_{l},t}-\widetilde{\theta}^i_{s_{l},\eta_n(t)}|^2]$ is an immediate consequence of $(D_3)$. When $l>1$, we have, by $(D_2)$, \ref{as coefficients} and \ref{ass exist and uniq}:
\begin{align*}
&\EE\left[\left|\widetilde{\theta}^{i}_{s_{l},t}-\widetilde{\theta}^{i}_{s_{l},\eta_n(t)}\right|^2\right]\\
&\leq \EE\left[\left|\widetilde{\theta}^{i}_{s_{l-1},t}-\widetilde{\theta}^{i}_{s_{l-1},\eta_n(t)}+(s_l-s_{l-1})
\left(\nabla_a\bH^\sigma_t(\widetilde{\theta}^{i}_{s_{l-1},t},\widetilde{\nu}^{N_2}_{s_{l-1},\cdot},\DataM^{N_1})-
\nabla_a\bH^\sigma_{\eta_n(t)}(\widetilde{\theta}^{i}_{s_{l-1},\eta_n(t)},\widetilde{\nu}^{N_2}_{s_{l-1},\cdot},\DataM^{N_1})\right)\right|^2\right]\\
&\leq \EE\left[\left|\widetilde{\theta}^{i}_{s_{l-1},t}-\widetilde{\theta}^{i}_{s_{l-1},\eta_n(t)}\right|^2\right]
\left(1-(\sigma^2\kappa-3)(s_{l}-s_{l-1})\right)\\
&\quad +\EE\left[\left|\widetilde{\theta}^{i}_{s_{l-1},t}-\widetilde{\theta}^{i}_{s_{l-1},\eta_n(t)}\right|^2\right]
\times L(s_l-s_{l-1})\Big(1+(s_l-s_{l-1})(1+\frac{\sigma^4}{2}\Vert\nabla_a U(\cdot)\Vert_{Lip})\Big)
\\
&\quad +2c(s_{l}-s_{l-1})(1+(s_l-s_{l-1}))\Big((t-\eta_n(t))^{2\alpha}\\
&\quad +\int \EE\left[\left|X^{\xi,\zeta}_t(\widetilde{\nu}^{N_2}_{s_{l-1},\cdot})-X^{\xi,\zeta}_{\eta_n(t)}
(\widetilde{\nu}^{N_2}_{s_{l-1},\cdot})\right|^2+\left|P^{\xi,\zeta}_t(\widetilde{\nu}^{N_2}_{s_{l-1},\cdot})-P^{\xi,\zeta}_{\eta_n(t)}
(\widetilde{\nu}^{N_2}_{s_{l-1},\cdot})\right|^2\right]\,\DataM^{N_1}(d\xi,d\zeta)\Big).
\end{align*}

With the assumptions ~\ref{as coefficients}  and~\ref{ass exist and uniq},
one can check that
\begin{align*}
&\EE\left[|X^{\xi,\zeta}_t(\widetilde{{{\nu}}}^{N_2}_{s_{l-1},\cdot})-X^{\xi,\zeta}_{\eta_n(t)}(\widetilde{{{\nu}}}^{N_2}_{s_{l-1},\cdot})|^2\right]
+\EE\left[|P^{\xi,\zeta}_t(\widetilde{{{\nu}}}^{N_2}_{s_{l-1},\cdot})-P^{\xi,\zeta}_{\eta_n(t)}(\widetilde{{{\nu}}}^{N_2}_{s_{l-1},\cdot})|^2\right]\\
&\leq c\max_t|t-\eta_n(t)|^{2}\left(1+\EE\left[\int_{0}^{T}|\widetilde{\theta}^i_{s_{l-1},t}|^2\,dt\right]\right)
\end{align*}
so that
\begin{align*}
&\EE\left[\left|\widetilde{\theta}^{i}_{s_{l},t}-\widetilde{\theta}^{i}_{s_{l},\eta_n(t)}\right|^2\right]
\leq c(s_{l}-s_{l-1})(1+(s_l-s_{l-1}))(t-\eta_n(t))^{2(\alpha \wedge 1)}\\
&\quad +\EE\left[\left|\widetilde{\theta}^{i}_{s_{l-1},t}-\widetilde{\theta}^{i}_{s_{l-1},\eta_n(t)}\right|^2\right]
\left(1-(\sigma^2\kappa-3)(s_{l}-s_{l-1})+ L(s_l-s_{l-1})\Big(1+(s_l-s_{l-1})(1+\frac{\sigma^4}{2}\Vert\nabla_a U\Vert_{Lip})\right).
\end{align*}
Proceeding as in \textbf{Step 2} of the proof of Lemma \ref{prop:EulerRate1}, we get
\begin{equation}\label{EulerRate2:proofstp4}
\EE\left[\left|\widetilde{\theta}^{i}_{s_{l},t}-\widetilde{\theta}^{i}_{s_{l},\eta_n(t)}\right|^2\right]
\leq c(1+\max_{l'}(s_{l'}-s_{l'-1}))(t-\eta_n(t))^{2(\alpha \wedge 1)}\\
\end{equation}
\textbf{Step 3}. Plugging \eqref{EulerRate2:proofstp4} into \eqref{EulerRate2:proofstp3}, coming back to \eqref{EulerRate2:proofstp2}, and then following again \textbf{Step 2} from the proof of Lemma \ref{prop:EulerRate1}, we conclude.

\end{proof} 




\vskip 0.2in
\bibliography{Bibliography}

\begin{thebibliography}{65}
\providecommand{\natexlab}[1]{#1}
\providecommand{\url}[1]{\texttt{#1}}
\expandafter\ifx\csname urlstyle\endcsname\relax
  \providecommand{\doi}[1]{doi: #1}\else
  \providecommand{\doi}{doi: \begingroup \urlstyle{rm}\Url}\fi

\bibitem[Arora et~al.(2019)Arora, Du, Hu, Li, and Wang]{arora2019fine}
S.~Arora, S.~Du, W.~Hu, Z.~Li, and R.~Wang.
\newblock Fine-grained analysis of optimization and generalization for
  overparameterized two-layer neural networks.
\newblock In \emph{International Conference on Machine Learning}, pages
  322--332. PMLR, 2019.

\bibitem[Baydin et~al.(2018)Baydin, Pearlmutter, Radul, and
  Siskind]{baydin2018automatic}
A.~G. Baydin, B.~A. Pearlmutter, A.~A. Radul, and J.~M. Siskind.
\newblock Automatic differentiation in machine learning: a survey.
\newblock \emph{Journal of machine learning research}, 18\penalty0 (153), 2018.

\bibitem[Belkin et~al.(2019)Belkin, Hsu, Ma, and Mandal]{belkin2019reconciling}
M.~Belkin, D.~Hsu, S.~Ma, and S.~Mandal.
\newblock Reconciling modern machine-learning practice and the classical
  bias--variance trade-off.
\newblock \emph{Proceedings of the National Academy of Sciences}, 116\penalty0
  (32):\penalty0 15849--15854, 2019.

\bibitem[Bellman(1966)]{bellman1966dynamic}
R.~Bellman.
\newblock Dynamic programming.
\newblock \emph{Science}, 153\penalty0 (3731):\penalty0 34--37, 1966.

\bibitem[Bensoussan(2004)]{bensoussan2004stochastic}
A.~Bensoussan.
\newblock \emph{Stochastic control of partially observable systems}.
\newblock Cambridge University Press, 2004.

\bibitem[Bensoussan and Lions(2011)]{bensoussan2011applications}
A.~Bensoussan and J.-L. Lions.
\newblock \emph{Applications of variational inequalities in stochastic
  control}.
\newblock Elsevier, 2011.

\bibitem[Bertsekas(1995)]{bertsekas1995dynamic}
D.~P. Bertsekas.
\newblock \emph{Dynamic programming and optimal control}.
\newblock Athena scientific Belmont, MA, 1995.

\bibitem[Bo et~al.(2019)Bo, Capponi, and Liao]{bo2019relaxed}
L.~Bo, A.~Capponi, and H.~Liao.
\newblock Relaxed control and gamma-convergence of stochastic optimization
  problems with mean field.
\newblock \emph{arXiv preprint arXiv:1906.08894}, 2019.

\bibitem[Carmona and Delarue(2018)]{carmona2018probabilistic}
R.~Carmona and F.~Delarue.
\newblock \emph{Probabilistic Theory of Mean Field Games with Applications
  I-II}.
\newblock Springer, 2018.

\bibitem[Chassagneux et~al.(2019)Chassagneux, Szpruch, and
  Tse]{chassagneux2019weak}
J.-F. Chassagneux, L.~Szpruch, and A.~Tse.
\newblock Weak quantitative propagation of chaos via differential calculus on
  the space of measures.
\newblock \emph{arXiv:1901.02556}, 2019.

\bibitem[Chen et~al.(2018)Chen, Rubanova, Bettencourt, and
  Duvenaud]{chen2018neural}
R.~T.~Q. Chen, Y.~Rubanova, J.~Bettencourt, and D.~Duvenaud.
\newblock Neural ordinary differential equations.
\newblock In \emph{Advances in neural information processing systems}, pages
  6571--6583, 2018.

\bibitem[Cheng et~al.(2019)Cheng, Bartlett, and Jordan]{cheng2019quantitative}
X.~Cheng, P.~L. Bartlett, and M.~I. Jordan.
\newblock Quantitative $w_1$ convergence of {L}angevin-like stochastic
  processes with non-convex potential state-dependent noise.
\newblock \emph{arXiv:1907.03215}, 2019.

\bibitem[Chizat and Bach(2018)]{chizat2018global}
L.~Chizat and F.~Bach.
\newblock On the global convergence of gradient descent for over-parameterized
  models using optimal transport.
\newblock In \emph{Advances in neural information processing systems}, pages
  3040--3050, 2018.

\bibitem[Chizat et~al.(2018)Chizat, Oyallon, and Bach]{chizat2018lazy}
L.~Chizat, E.~Oyallon, and F.~Bach.
\newblock On lazy training in differentiable programming.
\newblock \emph{arXiv preprint arXiv:1812.07956}, 2018.

\bibitem[Cuchiero et~al.(2019)Cuchiero, Larsson, and
  Teichmann]{cuchiero2019deep}
C.~Cuchiero, M.~Larsson, and J.~Teichmann.
\newblock Deep neural networks, generic universal interpolation, and controlled
  odes.
\newblock \emph{arXiv preprint arXiv:1908.07838}, 2019.

\bibitem[Delarue et~al.(2019)Delarue, Lacker, and Ramanan]{delarue2019master}
F.~Delarue, D.~Lacker, and K.~Ramanan.
\newblock From the master equation to mean field game limit theory: A central
  limit theorem.
\newblock \emph{Electronic Journal of Probability}, 24, 2019.

\bibitem[Dereich et~al.(2013)Dereich, Scheutzow, and
  Schottstedt]{dereich2013constructive}
S.~Dereich, M.~Scheutzow, and R.~Schottstedt.
\newblock Constructive quantization: approximation by empirical measures.
\newblock \emph{Ann. Inst. Henri Poincar\'{e} Probab. Stat.}, 49\penalty0
  (4):\penalty0 1183--1203, 2013.

\bibitem[Dupuis and Ellis(2011)]{dupuis2011weak}
P.~Dupuis and R.~S. Ellis.
\newblock \emph{A weak convergence approach to the theory of large deviations}.
\newblock John Wiley \& Sons, 2011.

\bibitem[Durmus and Moulines(2017)]{durmus2017nonasymptotic}
A.~Durmus and E.~Moulines.
\newblock Nonasymptotic convergence analysis for the unadjusted {L}angevin
  algorithm.
\newblock \emph{The Annals of Applied Probability}, 27\penalty0 (3):\penalty0
  1551--1587, 2017.

\bibitem[Dziugaite and Roy(2018)]{dziugaite2018data}
G.~K. Dziugaite and D.~M. Roy.
\newblock Data-dependent pac-bayes priors via differential privacy.
\newblock \emph{arXiv preprint arXiv:1802.09583}, 2018.

\bibitem[E et~al.(2018)E, Han, and Li]{han2018mean}
W.~E, J.~Han, and Q.~Li.
\newblock A mean-field optimal control formulation of deep learning.
\newblock \emph{arXiv:1807.01083}, 2018.

\bibitem[Eberle(2016)]{eberle2016reflection}
A.~Eberle.
\newblock Reflection couplings and contraction rates for diffusions.
\newblock \emph{Probability theory and related fields}, 166\penalty0
  (3-4):\penalty0 851--886, 2016.

\bibitem[Fleming and Soner(2006)]{fleming2006controlled}
W.~H. Fleming and H.~M. Soner.
\newblock \emph{Controlled Markov processes and viscosity solutions}.
\newblock Springer, 2006.

\bibitem[Fournier and Guillin(2015)]{fournier2015rate}
N.~Fournier and A.~Guillin.
\newblock On the rate of convergence in {W}asserstein distance of the empirical
  measure.
\newblock \emph{Probability Theory and Related Fields}, 162\penalty0
  (3-4):\penalty0 707--738, 2015.

\bibitem[Gal and Ghahramani(2015)]{gal2015bayesian}
Y.~Gal and Z.~Ghahramani.
\newblock Bayesian convolutional neural networks with bernoulli approximate
  variational inference.
\newblock \emph{arXiv preprint arXiv:1506.02158}, 2015.

\bibitem[Ghorbani et~al.(2019)Ghorbani, Mei, Misiakiewicz, and
  Montanari]{ghorbani2019linearized}
B.~Ghorbani, S.~Mei, T.~Misiakiewicz, and A.~Montanari.
\newblock Linearized two-layers neural networks in high dimension.
\newblock \emph{arXiv preprint arXiv:1904.12191}, 2019.

\bibitem[Ghorbani et~al.(2020)Ghorbani, Mei, Misiakiewicz, and
  Montanari]{ghorbani2020neural}
B.~Ghorbani, S.~Mei, T.~Misiakiewicz, and A.~Montanari.
\newblock When do neural networks outperform kernel methods?
\newblock \emph{arXiv preprint arXiv:2006.13409}, 2020.

\bibitem[Gy\"{o}ngy and \v{S}i\v{s}ka(2009)]{gyongy2009normalized}
I.~Gy\"{o}ngy and D.~\v{S}i\v{s}ka.
\newblock On finite-difference approximations for normalized {B}ellman
  equations.
\newblock \emph{Appl. Math. Optim.}, 60\penalty0 (3):\penalty0 297--339, 2009.
\newblock \doi{10.1007/s00245-009-9082-0}.
\newblock URL \url{https://doi.org/10.1007/s00245-009-9082-0}.

\bibitem[Hairer et~al.(1987)Hairer, N{\o}rsett, and
  Wanner]{hairer:norsett:wanner:1993}
E.~Hairer, S.~P. N{\o}rsett, and G.~Wanner.
\newblock \emph{Solving ordinary differential equations. {I}}.
\newblock Springer, Berlin, 1987.
\newblock ISBN 3-540-17145-2.
\newblock \doi{10.1007/978-3-662-12607-3}.
\newblock URL \url{https://doi.org/10.1007/978-3-662-12607-3}.

\bibitem[Hastie et~al.(2019)Hastie, Montanari, Rosset, and
  Tibshirani]{hastie2019surprises}
T.~Hastie, A.~Montanari, S.~Rosset, and R.~J. Tibshirani.
\newblock Surprises in high-dimensional ridgeless least squares interpolation.
\newblock \emph{arXiv preprint arXiv:1903.08560}, 2019.

\bibitem[He et~al.(2016)He, Zhang, Ren, and Sun]{he2016deep}
K.~He, X.~Zhang, S.~Ren, and J.~Sun.
\newblock Deep residual learning for image recognition.
\newblock In \emph{Proceedings of the IEEE conference on computer vision and
  pattern recognition}, pages 770--778, 2016.

\bibitem[Heiss et~al.(2019)Heiss, Teichmann, and Wutte]{heiss2019implicit}
J.~Heiss, J.~Teichmann, and H.~Wutte.
\newblock How implicit regularization of neural networks affects the learned
  function--part i.
\newblock \emph{arXiv preprint arXiv:1911.02903}, 2019.

\bibitem[Hu et~al.(2019{\natexlab{a}})Hu, Kazeykina, and Ren]{hu2019meanode}
K.~Hu, A.~Kazeykina, and Z.~Ren.
\newblock Mean-field {L}angevin system, optimal control and deep neural
  networks.
\newblock \emph{arXiv:1909.07278}, 2019{\natexlab{a}}.

\bibitem[Hu et~al.(2019{\natexlab{b}})Hu, Ren, \v{S}i\v{s}ka, and
  Szpruch]{hu2019mean}
K.~Hu, Z.~Ren, D.~\v{S}i\v{s}ka, and L.~Szpruch.
\newblock Mean-field {L}angevin dynamics and energy landscape of neural
  networks.
\newblock \emph{ar{X}iv:1905.07769}, 2019{\natexlab{b}}.

\bibitem[Jabir(2019)]{jabir2019rate}
J.-F. Jabir.
\newblock Rate of propagation of chaos for diffusive stochastic particle
  systems via {G}irsanov transformation.
\newblock \emph{arXiv:1907.09096}, 2019.

\bibitem[Jacot et~al.(2018)Jacot, Gabriel, and Hongler]{jacot2018neural}
A.~Jacot, F.~Gabriel, and C.~Hongler.
\newblock Neural tangent kernel: Convergence and generalization in neural
  networks.
\newblock \emph{arXiv preprint arXiv:1806.07572}, 2018.

\bibitem[Karatzas and Shreve(2012)]{karatzas2012brownian}
I.~Karatzas and S.~Shreve.
\newblock \emph{Brownian motion and stochastic calculus}.
\newblock Springer, 2012.

\bibitem[Kushner and Dupuis(2001)]{kushner2001numerical}
H.~J. Kushner and P.~Dupuis.
\newblock \emph{Numerical methods for stochastic control problems in continuous
  time}.
\newblock Springer, 2001.
\newblock \doi{10.1007/978-1-4613-0007-6}.

\bibitem[Ladyzenskaja et~al.(1968)Ladyzenskaja, Solonnikov, and
  Ural'ceva]{LSU68}
O.~A. Ladyzenskaja, V.~A. Solonnikov, and N.~N. Ural'ceva.
\newblock \emph{Linear and quasi-linear equations of parabolic type}.
\newblock Translations of Mathematical Monographs. AMS, 1968.

\bibitem[LeCun et~al.(2015)LeCun, Bengio, and Hinton]{lecun2015deep}
Y.~LeCun, Y.~Bengio, and G.~Hinton.
\newblock Deep learning.
\newblock \emph{Nature}, 521\penalty0 (7553):\penalty0 436--444, 2015.

\bibitem[Li et~al.(2017)Li, Chen, Tai, and E]{li2017maximum}
Q.~Li, L.~Chen, C.~Tai, and W.~E.
\newblock Maximum principle based algorithms for deep learning.
\newblock \emph{The Journal of Machine Learning Research}, 18\penalty0
  (1):\penalty0 5998--6026, 2017.

\bibitem[Liu and Markowich(2020)]{liu2020selection}
H.~Liu and P.~Markowich.
\newblock Selection dynamics for deep neural networks.
\newblock \emph{Journal of Differential Equations}, 269\penalty0 (12):\penalty0
  11540--11574, 2020.

\bibitem[Ma et~al.(2019)Ma, Wu, et~al.]{ma2019barron}
C.~Ma, L.~Wu, et~al.
\newblock Barron spaces and the compositional function spaces for neural
  network models.
\newblock \emph{arXiv preprint arXiv:1906.08039}, 2019.

\bibitem[MacKay(1995)]{mackay1995probable}
D.~J. MacKay.
\newblock Probable networks and plausible predictions---a review of practical
  bayesian methods for supervised neural networks.
\newblock \emph{Network: computation in neural systems}, 6\penalty0
  (3):\penalty0 469--505, 1995.

\bibitem[Majka et~al.(2018)Majka, Mijatovi{\'c}, and Szpruch]{majka2018non}
M.~B. Majka, A.~Mijatovi{\'c}, and L.~Szpruch.
\newblock Non-asymptotic bounds for sampling algorithms without log-concavity.
\newblock \emph{arXiv:1808.07105}, 2018.

\bibitem[Mallat(2016)]{mallat2016understanding}
S.~Mallat.
\newblock Understanding deep convolutional networks.
\newblock \emph{Philosophical Transactions of the Royal Society A:
  Mathematical, Physical and Engineering Sciences}, 374\penalty0 (2065), 2016.

\bibitem[Mei and Montanari(2019)]{mei2019generalization}
S.~Mei and A.~Montanari.
\newblock The generalization error of random features regression: Precise
  asymptotics and double descent curve.
\newblock \emph{arXiv:1908.05355}, 2019.

\bibitem[Mei et~al.(2018)Mei, Montanari, and Nguyen]{mei2018mean}
S.~Mei, A.~Montanari, and P.-M. Nguyen.
\newblock A mean field view of the landscape of two-layer neural networks.
\newblock \emph{Proceedings of the National Academy of Sciences}, 115\penalty0
  (33):\penalty0 E7665--E7671, 2018.

\bibitem[Mei et~al.(2019)Mei, Misiakiewicz, and Montanari]{mei2019mean}
S.~Mei, T.~Misiakiewicz, and A.~Montanari.
\newblock Mean-field theory of two-layers neural networks: dimension-free
  bounds and kernel limit.
\newblock In \emph{Conference on Learning Theory}, pages 2388--2464. PMLR,
  2019.

\bibitem[Montanari et~al.(2019)Montanari, Ruan, Sohn, and
  Yan]{montanari2019generalization}
A.~Montanari, F.~Ruan, Y.~Sohn, and J.~Yan.
\newblock The generalization error of max-margin linear classifiers:
  High-dimensional asymptotics in the overparametrized regime.
\newblock \emph{arXiv:1911.01544}, 2019.

\bibitem[Neal(2012)]{neal2012bayesian}
R.~M. Neal.
\newblock \emph{Bayesian learning for neural networks}, volume 118.
\newblock Springer Science \& Business Media, 2012.

\bibitem[Neyshabur et~al.(2017)Neyshabur, Tomioka, Salakhutdinov, and
  Srebro]{neyshabur2017geometry}
B.~Neyshabur, R.~Tomioka, R.~Salakhutdinov, and N.~Srebro.
\newblock Geometry of optimization and implicit regularization in deep
  learning.
\newblock \emph{arXiv:1705.03071}, 2017.

\bibitem[Neyshabur et~al.(2018)Neyshabur, Li, Bhojanapalli, LeCun, and
  Srebro]{neyshabur2018towards}
B.~Neyshabur, Z.~Li, S.~Bhojanapalli, Y.~LeCun, and N.~Srebro.
\newblock Towards understanding the role of over-parametrization in
  generalization of neural networks.
\newblock \emph{arXiv:1805.12076}, 2018.

\bibitem[Rotskoff and Vanden-Eijnden(2018)]{rotskoff2018neural}
G.~M. Rotskoff and E.~Vanden-Eijnden.
\newblock Neural networks as interacting particle systems: Asymptotic convexity
  of the loss landscape and universal scaling of the approximation error.
\newblock \emph{arXiv:1805.00915}, 2018.

\bibitem[Silver et~al.(2016)Silver, Huang, Maddison, Guez, Sifre, Van
  Den~Driessche, Schrittwieser, Antonoglou, Panneershelvam, Lanctot,
  et~al.]{silver2016mastering}
D.~Silver, A.~Huang, C.~J. Maddison, A.~Guez, L.~Sifre, G.~Van Den~Driessche,
  J.~Schrittwieser, I.~Antonoglou, V.~Panneershelvam, M.~Lanctot, et~al.
\newblock Mastering the game of {G}o with deep neural networks and tree search.
\newblock \emph{Nature}, 529\penalty0 (7587):\penalty0 484--489, 2016.

\bibitem[Sirignano and Spiliopoulos(2018)]{sirignano2018mean}
J.~Sirignano and K.~Spiliopoulos.
\newblock Mean field analysis of neural networks.
\newblock \emph{arXiv:1805.01053}, 2018.

\bibitem[Sontag and Sussmann(1997)]{sontag1997complete}
E.~Sontag and H.~Sussmann.
\newblock Complete controllability of continuous-time recurrent neural
  networks.
\newblock \emph{Systems \& control letters}, 30\penalty0 (4):\penalty0
  177--183, 1997.

\bibitem[Szpruch and Tse(2019)]{szpruch2019antithetic}
{\L}.~Szpruch and A.~Tse.
\newblock Antithetic multilevel particle system sampling method for
  {M}c{K}ean--{V}lasov {SDE}s.
\newblock \emph{arXiv:1903.07063}, 2019.

\bibitem[Szpruch et~al.(2019)Szpruch, Tan, and Tse]{szpruch2017iterative}
L.~Szpruch, S.~Tan, and A.~Tse.
\newblock Iterative particle approximation for {M}c{K}ean-{V}lasov {SDE}s with
  application to multilevel {M}onte {C}arlo estimation.
\newblock \emph{To appear in Annals of Applied Probability}, 2019.

\bibitem[Tzen and Raginsky(2020)]{tzen2020mean}
B.~Tzen and M.~Raginsky.
\newblock A mean-field theory of lazy training in two-layer neural nets:
  entropic regularization and controlled mckean-vlasov dynamics.
\newblock \emph{arXiv preprint arXiv:2002.01987}, 2020.

\bibitem[Villani(2008)]{villani2008optimal}
C.~Villani.
\newblock \emph{Optimal transport: old and new}.
\newblock Springer, 2008.

\bibitem[Weinan(2017)]{weinan2017proposal}
E.~Weinan.
\newblock A proposal on machine learning via dynamical systems.
\newblock \emph{Communications in Mathematics and Statistics}, 5\penalty0
  (1):\penalty0 1--11, 2017.

\bibitem[Weinan et~al.(2019)Weinan, Ma, and Wu]{weinan2019machine}
E.~Weinan, C.~Ma, and L.~Wu.
\newblock Machine learning from a continuous viewpoint.
\newblock \emph{arXiv preprint arXiv:1912.12777}, 2019.

\bibitem[Young(2000)]{young2000lectures}
L.~C. Young.
\newblock \emph{Lectures on the calculus of variations and optimal control
  theory}, volume 304.
\newblock American Mathematical Soc., 2000.

\bibitem[Zhang et~al.(2016)Zhang, Bengio, Hardt, Recht, and
  Vinyals]{zhang2016understanding}
C.~Zhang, S.~Bengio, M.~Hardt, B.~Recht, and O.~Vinyals.
\newblock Understanding deep learning requires rethinking generalization.
\newblock \emph{arXiv:1611.03530}, 2016.

\end{thebibliography}

\end{document}